\documentclass[a4,12pt]{amsart}

\usepackage{amssymb}
\usepackage{amstext}
\usepackage{amsmath}
\usepackage{amscd}
\usepackage{amsthm}
\usepackage{amsfonts}
\usepackage{enumerate}
\usepackage{graphicx}
\usepackage{latexsym}
\usepackage[all]{xy}
\usepackage[usenames]{color}
\usepackage[total={6.5in,8.75in},
top=1.2in, right=1.0in, left=1.0in, bottom=1.0in,includefoot]{geometry} %

\newtheorem*{corollary*}{Corollary}
\newtheorem{theorem}{Theorem}[section]
\newtheorem{corollary}[theorem]{Corollary}
\newtheorem{lemma}[theorem]{Lemma}
\newtheorem{proposition}[theorem]{Proposition}
\newtheorem*{proposition*}{Proposition}

\newtheorem{claim}{Claim}
\newtheorem*{claim*}{Claim}

\theoremstyle{definition}
\newtheorem{definition}[theorem]{Definition}
\newtheorem{remark}[theorem]{Remark}
\newtheorem{condition}[theorem]{ Condition}
\newtheorem{example}[theorem]{Example}

\newtheorem{definitiontheorem}[theorem]{Definition-Theorem}

\theoremstyle{remark}

\numberwithin{equation}{theorem}

\newenvironment{pfclaim}{

\begin{proof}
}
{
\end{proof}

}

\renewcommand{\mod}{\operatorname{mod}}
\newcommand{\proj}{\operatorname{proj}}

\newcommand{\End}{\operatorname{End}}
\newcommand{\Hom}{\operatorname{Hom}}
\newcommand{\add}{\operatorname{\mathsf{add}}}

\newcommand{\Kb}{\mathsf{K}^{\rm b}}

\renewcommand{\H}{\mathcal{H}}

\newcommand{\silt}{\operatorname{\mathsf{silt}}}
\newcommand{\tsilt}{\operatorname{\mathsf{2silt}}}

\newcommand{\sttilt}{\operatorname{\mathsf{s\tau -tilt}}}

\newcommand{\trigid}{\operatorname{\mathsf{\tau\text{-}rigid}}}
\newcommand{\supp}{\operatorname{Supp}}
\newcommand{\Z}{{\mathbb{Z}}}

\newcommand{\R}{{\mathbb{R}}}

\newcommand{\Rad}{\operatorname{Rad}}

\renewcommand{\Im}{\operatorname{Im}}

\newcommand{\fac}{\operatorname{\mathsf{Fac}}}
\newcommand{\Sym}{\mathfrak S}
\newcommand{\ind}{\operatorname{ind}}

\newcommand{\dip}{\operatorname{dp}}
\newcommand{\dis}{\operatorname{ds}}

\begin{document}
\title[Weak orders on symmetric groups and  posets of support $\tau$-tilting modules]
{Weak orders on symmetric groups and posets of support $\tau$-tilting modules}

\author{ Ryoichi Kase}
\address{Department of Mathematics,
Nara Women's University, Kitauoya-Nishimachi, Nara city, Nara 630-8506, Japan}
\email{r-kase@cc.nara-wu.ac.jp}
\urladdr{}
\begin{abstract}
We give a necessary and sufficient condition for that the support $\tau$-tilting poset of a finite dimensional
algebra $\Lambda$ is isomorphic to the poset of symmetric group $\Sym_{n+1}$ with weak order. 
Moreover we show that there are infinitely many finite dimensional algebras
whose support $\tau$-tilting posets are isomorphic  to $\Sym_{n+1} $.
\end{abstract}
\maketitle
\section{Introduction}
The notion of tilting modules was introduced in \cite{BrB}.
It is known that they control derived equivalence \cite{H}.
 Therefore to obtain many tilting modules
 is an important problem in representation theory of finite dimensional algebras.
Tilting mutation given by Riedtmann-Schofield \cite{RS} is an approach to this problem.
It is an operation which gives a new tilting module from given one by replacing an
 indecomposable direct summand. However tilting mutation is not always possible
 depending on a choice of an indecomposable direct summand.

 Adachi-Iyama-Reiten introduced the notion of support $\tau$-tilting modules as a 
   generalization of tilting modules \cite{AIR}. 
   They give a mutation of support $\tau$-tilting modules and complemented
   that of tilting modules. i.e. the support $\tau$-tilting mutation has following nice properties:
   \begin{itemize}  
\item Support $\tau$-tilting mutation is always possible.
\item There is a partial order on the set of (isomorphism classes of) basic support $\tau$-tilting modules
such that its Hasse quiver realizes the support $\tau$-tilting mutation. (An analogue of Happel-Unger's result \cite{HU} for tilting modules.)
\end{itemize}
   Moreover they showed deep connections between $\tau$-tilting theory, silting theory, torsion theory and cluster tilting theory.
   
   Then for several classes of algebras, support $\tau$-tilting posets are calculated.
One interesting example is a preprojective algebra of Dynkin type. Preprojective
algebras play an important role in representation theory of algebras and Lie theory.
Mizuno shows the following result.
\begin{theorem}
\label{mizuno}\cite[Theorem\;2.30]{M} Let $\Lambda$ be a preprojective algebra of Dynkin type.
Then the support $\tau$-tilting poset of $\Lambda$ is isomorphic to 
corresponding Weyl group with weak order.
\end{theorem}
In particular, the support $\tau$-tilting poset of preprojective algebra of 
 type $A$ is realized by the symmetric group with weak order.
 Such an algebra is not only preprojective algebra of type $A$. Iyama-Zhang shows that
 support $\tau$-tilting poset of the Auslander algebra of the truncated polynomial ring is also
 isomorphic to the symmetric group with weak order \cite{IZ}.
In this paper we classify such algebras.
  

    
\subsection*{Notation} Throughout this paper, let $\Lambda=kQ/I$ be a basic finite dimensional algebra over an algebraically
closed field $k$, where $Q$ is a finite quiver and $I$ is an admissible ideal of $kQ$.

We denote by $Q_0$ the set of 
vertices  of $Q$ and $Q_1$ the set of arrows of $Q$. We set $Q^{\circ}$ the quiver obtained from $Q$ by 
deleting all loops.

\begin{enumerate}[1.]

\item For arrows $\alpha:a_0\to a_1$ and $\beta: b_0\to b_1$ of $Q$,
we mean by $\alpha\beta$ the path $a_0\xrightarrow{\alpha}a_1\xrightarrow{\beta}b_1$ if $a_1=b_0$, otherwise 0 in $kQ$.

\item We denote by $\mod \Lambda\ (\proj\Lambda)$ the category of finitely generated (projective) right $\Lambda$-modules.

\item By a module, we always mean a finitely generated right module.

\item For a poset $\mathbb{P}$ and $a,b \in \mathbb{P}$, we denote by $\H(\mathbb{P})$ the Hasse quiver
of $\mathbb{P}$ and put $[a,b]:=\{x\in \mathbb{P}
\mid a\leq x\leq b\}$. We denote by $\dip(a)$ the set of direct predecessor of $a$ in $\H(\mathbb{P})$
and by $\dis(a)$ the set of direct successor of $a$ in $\H(\mathbb{P})$.
We say that $\mathbb{P}$ is $n$-\emph{regular} provided $\dip(a)+\dis(a)=n$ holds for any element $a\in \mathbb{P}$. 
We call a subposet $\mathbb{P'}$ of $\mathbb{P}$
a \emph{full subposet} if the inclusion $\mathbb{P'}\subset \mathbb{P}$ induces
 a quiver inclusion from $\H(\mathbb{P'})$ to $\H(\mathbb{P})$. 
By definition if $\mathbb{P'}$ is a full subposet of $\mathbb{P}$, then $\H(\mathbb{P'})$
is a full subquiver of $\H(\mathbb{P})$. 
\end{enumerate} 
\section{Preliminary}
In this section, we recall the definitions and their basic properties of 
support $\tau$-tilting modules, silting complexes and the weak order on Symmetric groups.


\subsection{Support $\tau$-tilting modules}
\label{subsec:2.2}
For a module $M$, we denote by $|M|$ the number of non-isomorphic indecomposable direct summands of $M$.
The Auslander-Reiten translation is denoted by $\tau$. 
(Refer to \cite{ASS, ARS} for definition and properties.)

Let us recall the definition of support $\tau$-tilting modules.

\begin{definition}
Let $M$ be a $\Lambda$-module and $P$ a projective $\Lambda$-module.
\begin{enumerate}
\item We say that $M$ is $\tau$-\emph{rigid} if it satisfies
$\Hom_{\Lambda}(M, \tau M)=0$.
\item A pair $(M,P)$ is said to be $\tau$-\emph{rigid} if $M$ is $\tau$-rigid and $\Hom_\Lambda(P,M)=0$.
\item 
A support $\tau$-tilting pair $(M,P)$ is defined to be a $\tau$-rigid pair with $|M|+|P|=|\Lambda|$. 
\item We call $M$ a \emph{support $\tau$-tilting module} if there exists a projective module $P$ such that
$(M,P)$ is a support $\tau$-tilting pair. 
The set of isomorphism classes of basic support $\tau$-tilting modules of $\Lambda$ is denoted by $\sttilt\Lambda$.
\end{enumerate}
\end{definition}
We denote by $e_i$ the primitive idempotent corresponding to a vertex $i$ of $Q$.
For a module $M$, we define a subset of $Q_0$ by
\[\supp(M):=\{i\in Q_0\ |\ Me_i\neq0 \}.\]
If $(M, e\Lambda)$ is a support $\tau$-tilting pair for some idempotent $e$, 
then $\supp(M)$ coincides with the set of vertices $i$ satisfying $ee_i=0$.
\begin{proposition}
\cite[Proposition 2.3]{AIR}
Let $M$ be a support $\tau$-tilting module. If $(M,P)$ and $(M,P')$ are support $\tau$-tilting pairs,
then $\add P=\add P'=\add e\Lambda$, where $e=\sum_{i\in Q_0\setminus \supp(M) }e_i$. 
\end{proposition}


\begin{proposition}\cite[Proposition 1.3, Lemma 2.1]{AIR}
\label{basicfact} The following hold.
 \begin{enumerate}[{\rm (1)}]
\item A $\tau$-rigid pair $(M,P)$ satisfies the inequality $|M|+|P|\leq |\Lambda|$. 
\item Let $J$ be an ideal of $\Lambda$. Let $M$ and $N$ be  $(\Lambda/J)$-modules .
If $\Hom_{\Lambda}(M,\tau N)=0$, then $\Hom_{\Lambda/I}(M,\tau_{\Lambda/J} N)=0$.
Moreover, if $J=(e)$ is an two-sided ideal generated by an idempotent $e$, then the converse holds.

\end{enumerate}
\end{proposition}

Denote by $\fac M$ the category of factor modules of finite direct sums of copies of $M$.

\begin{definitiontheorem}\cite[Lemma 2.25]{AIR}
For support $\tau$-tilting modules $M$ and $M'$, we write $M\geq M' $ if $\fac M\supseteq \fac M'$.
 Then one has the following equivalent conditions:
\begin{enumerate}
\item $M\geq M'$.
\item  $\Hom_{\Lambda}(M',\tau M)=0$ and $\supp(M)\supseteq \supp(M')$.
\end{enumerate}
Moreover, $\geq$ gives a partial order on $\sttilt\Lambda$.
\end{definitiontheorem} 

Let $(N,R)$ be a pair of a module $N$ and a projective module $R$.

We say that $(N,R)$ is \emph{basic} if so are $N$ and $R$.
A direct summand $(N', R')$ of $(N,R)$ is also a pair of a module $N'$ and a projective module $R'$
 which are direct summands of $N$ and $R$, respectively.

A pair $(N,R)$ is said to be \emph{almost complete support $\tau$-tilting} provided it is a $\tau$-rigid pair with $|N|+|R|=|\Lambda|-1$.

\begin{theorem}\label{basicprop}
\begin{enumerate}[{\rm (1)}]
\item \cite[Theorem\;2.18]{AIR} 
Every basic almost complete support $\tau$-tilting pair is a direct summand of exactly two basic support $\tau$-tilting pairs.

\item \cite[Corollary\;2.34]{AIR} Let $(M,P)$ and $(M',P')$ be basic support $\tau$-tilting pairs.
Then $M$ and $M'$ are connected by an arrow of $\H(\sttilt\Lambda)$ if and only if $(M,P)$ and $(M',P')$
 have a common basic almost complete support $\tau$-tilting pair as a direct summand.
In particular, $\sttilt\Lambda$ is $|\Lambda|$-regular. 

\item \cite[Corollary\;2.38]{AIR} If $\H(\sttilt\Lambda)$ has a finite connected component $\mathcal{C}$, then
$\mathcal{C}=\H(\sttilt\Lambda)$. 
\end{enumerate}   
\end{theorem}

%

For a basic $\tau$-rigid pair $(N,R)$, we define
\[\sttilt_{N\oplus R^-}\Lambda:=\{M\in \sttilt\Lambda\mid N\in \add M,\;\Hom_{\Lambda}(R, M)=0\},\]
equivalently, which consists of all support $\tau$-tilting pairs having $(N,R)$ as a direct summand. 
For simplicity, we omit 0 if $N=0$ or $R=0$.

Given an idempotent $e=e_{i_1}+\cdots+e_{i_\ell}$ of $\Lambda$ so that $R=e\Lambda$,
we see that $M$ belongs to $\sttilt_{R^-}\Lambda$ if and only if it is a basic support $\tau$-tilting module
with $\supp(M)=Q_0\setminus\{i_1,\dots,i_\ell\}$.
Hence, by Proposition\;\ref{basicfact} this leads to a poset isomorphism $\sttilt_{R^-}\Lambda\simeq \sttilt\Lambda/(e)$.
More generally, we have following reduction theorem.
\begin{theorem}\cite{J}
\label{reduction theorem}
Let $(N,R)$ be a basic  $\tau$-rigid pair and let $T$ be the Bongartz completion of $(N,R)$.
 If we set $\Gamma:=\End_{\Lambda}(T)/( e )$, then
  $|\Gamma|=|\Lambda|-|N|-|R|$ and $\sttilt_{N\oplus R^-}(\Lambda)\simeq \sttilt(\Gamma)$,
   where $e$ is the idempotent corresponding to the projective $\End_{\Lambda}(T)$-module $\Hom_{\Lambda}(T,N )$.
  \end{theorem}
 Theorem\;\ref{reduction theorem} implies that for an idempotent $e\in \Lambda$, we have a poset isomorphism
 $\sttilt_{e\Lambda} \Lambda\simeq \sttilt \Lambda/(e).$ 
\subsection{Silting complexes}

We denote by $\Kb(\proj\Lambda)$ the bounded homotopy category of $\proj\Lambda$.

A complex $T=[\cdots\rightarrow T^i\rightarrow T^{i+1}\rightarrow\cdots]$ in $\Kb(\proj\Lambda)$ 
is said to be \emph{two-term} provided $T^i=0$ unless $i=0,-1$.

We recall the definition of silting complexes.

\begin{definition} Let $T$ be a complex in $\Kb(\proj \Lambda)$.
\begin{enumerate}
\item We say that $T$ is \emph{presilting} if $\Hom_{\Kb(\proj \Lambda)}(T,T[i])=0$ for any positive integer $i$.
\item A \emph{silting complex} is defined to be presilting and generate $\Kb(\proj \Lambda)$ by taking direct summands, mapping cones and shifts.   
\end{enumerate}
We denote by $\silt\Lambda\ (\tsilt \Lambda)$ the set of isomorphism classes of basic (two-term) silting complexes in $\Kb(\proj \Lambda)$.
\end{definition}

We give an easy property of (pre)silting complexes. 

\begin{lemma}\label{factformpp}\cite[Lemma\;2.25]{AI}
Let $M$ be a $\tau$-rigid module and $P_1\stackrel{d}{\to} P_0\to M\to 0$ a minimal
projective presentation of $M$. Then $\add P_1\cap \add P_0=\{0\}.$
\end{lemma}
\begin{remark} Let $[P_1\stackrel{d}{\to} P_0]\in \tsilt \Lambda.$ 
By Theorem\;\ref{bijection} and Lemma\;\ref{factformpp}, we may assume that $\add P_1\cap \add P_0=\{0\}$.
\end{remark}

The set $\silt\Lambda$ also has poset structure as follows.

\begin{definitiontheorem}\cite[Theorem 2.11]{AI}
For silting complexes $T$ and $T'$ of $\Kb(\proj\Lambda)$, we write $T\geq T'$
 if $\Hom_{\Kb(\proj\Lambda)}(T, T'[i])=0$ for every positive integer $i$.
Then the relation $\geq$ gives a partial order on $\silt\Lambda$.
\end{definitiontheorem}

The following result connects silting theory with $\tau$-tilting theory.

\begin{theorem}\label{bijection}
\cite[Corollary\;3.9]{AIR}
We consider an assignment 
\[\begin{array}{ccccr}
&  &\;\;(-1\mathrm{th})& & (0\mathrm{th})\; \\
\mathbf{S}:(M,P)&\mapsto &[\;P_1\oplus P&\stackrel{(p_M,0)}{\longrightarrow}&P_0\;\;]\\
\end{array} \vspace{5pt}
\]
 where $p_M:P_1\to P_0$ is a minimal projective presentation of $M$.
\begin{enumerate}[{\rm (1)}]
\item\cite[Lemma\;3.4]{AIR} For modules $M,N$, the following are equivalent:
\begin{enumerate}[{\rm (a)}]
\item $\Hom_{\Lambda}(M,\tau N)=0$.
\item $\Hom_{\Kb(\proj \Lambda)}(\mathbf{S}(M),\mathbf{S}(N)[1])=0$.
\end{enumerate}
\item\cite[Lemma\;3.5]{AIR} For any projective module $P$ and any module $M$, the following are equivalent:
\begin{enumerate}[{\rm (a)}]
\item $\Hom_{\Lambda}(P,M)=0$.
 \item $\Hom_{\Kb(\proj \Lambda)}(\mathbf{S}(0,P),\mathbf{S}(M)[1])=0.$
\end{enumerate} 
\end{enumerate}
In particular, the assignment $\mathbf{S}$ gives rise to a poset isomorphism 
$\sttilt\Lambda\xrightarrow{\sim}\tsilt\Lambda$.
\end{theorem}

  
In the end of this subsection, we recall g-vector of 2-term objects of $\Kb(\proj \Lambda)$.
\begin{definition}
Let $X=[P'\to P]$ be a 2-term object of $\Kb(\proj \Lambda)$.
If $[P]-[P']=\sum_{i\in Q_0}g_i [e_i \Lambda]$ in the Grothendieck group $K_0(\proj \Lambda)$
of $\proj \Lambda$, then we call $(g_i)_{i\in Q_0}\in \Z^{Q_0}$ the {\bf g-vector} of $X$
and denote it by $g^X$.
\end{definition}

\begin{theorem}
\label{gvector}
\cite[Theorem\;5.5]{AIR}
The map $T\to g^T$ gives an injection from the set of isomorphism classes of 2-term presilting objects
to $K_0(\proj \Lambda)$.
\end{theorem}
\subsection{Weak orders on Symmetric groups}
Let $\Sym_{n+1}$ be the $(n+1)$-th symmetric group and $s_i\in \Sym_{n+1}$ denotes 
an adjacent transposition $(i,i+1)$. Then each element $w\in \Sym_{n+1}$ can be written in the from
$w=s_{i_{\ell}}s_{i_{\ell-1}}\cdots s_{i_1}$. If $\ell$ is minimum, then we call $\ell$  the
\emph{length} of $w$ and denote it by $\ell (w)$. In this case, an expression $s_{i_{\ell}}s_{i_{\ell-1}}\cdots s_{i_1}$
of $w$ is said to be a \emph{reduced expression} of $w$. The following is well known (see \cite[Section\;1]{BjB} for example).
\begin{theorem}
\label{matsumoto}
Let $w=s_{i_{\ell}}\cdots s_{i_1}$. 
\begin{enumerate}[{\rm(1)}]
\item  Assume that  $j<\ell$ satisfies 
\[\mathrm{(i)}\  s_{i_{\ell}}\cdots s_{i_{j+1}}(i_j)>s_{i_{\ell}}\cdots s_{i_{j+1}}(i_j+1).\]
 Then there 
exists $k<j$
  such that 
\[\mathrm{(ii)}\ s:=s_{i_{k-1}}\cdots s_{i_{j+1}}(i_j)<s_{i_{k-1}}\cdots s_{i_{j+1}}(i_j+1)=:t\ \mathrm{and}\ s_{i_k}(s)>s_{i_k}(t).\]
Moreover, we have 
\[w=s_{i_{\ell}}\cdots \widehat{s_{i_k}}\cdots \widehat{s_{i_j}}\cdots s_{i_1}. \]
\item The inversion number of $w$ is coinsides to $\ell(w)$.
\item$(${\rm Matsumoto's exchange condition}$).$ If $s_{i_{\ell}}\cdots s_1$ is a non-reduced expression, then
there exists $j<\ell$ satisfying {\rm (i)} of above and so
\[w=s_{i_{\ell}}\cdots \widehat{s_{i_k}}\cdots \widehat{s_{i_j}}\cdots s_{i_1}. \]
\end{enumerate}
  \end{theorem}
 We give a proof for reader's convenience. 
  \begin{proof}
   For $w\in \Sym_{n+1}$, we denote by $\gamma(w)$  the inversion number of $w$.
  It is well known that 
  \begin{itemize}
  \item $\gamma(s_i w )=\gamma(w)+1\Leftrightarrow w^{-1}(i)<w^{-1}(i+1).$
  \item $\gamma(s_i w )=\gamma(w)-1\Leftrightarrow w^{-1}(i)>w^{-1}(i+1).$
  \item $\gamma(ws_i )=\gamma(w)+1\Leftrightarrow w(i)<w(i+1).$
  \item $\gamma(ws_i )=\gamma(w)-1\Leftrightarrow w(i)>w(i+1).$
  \end{itemize}
    We show (1).
     Assume that $j<\ell$ satisfies (i).
  Then (ii) follows from (i) and $i_j<i_j+1$.
  It is easy to check that $(s,t)=(i_k,i_k+1)$.
  Hence we conclude that
  \[s_{i_{k-1}}\cdots s_{i_{j+1}}s_{i_j}s_{i_{j+1}}\cdots s_{i_{k-1}}=s_{i_k}. \]
  In fact, we have that
  \[s_{i_{k-1}}\cdots s_{i_{j+1}}(i_j)=i_k\ \mathrm{and}\ s_{i_{k-1}}\cdots s_{i_{j+1}}(i_j+1)=i_k+1. \] 
  Then we obtain that 
  \[s_{i_k}s_{i_{k-1}}\cdots s_{i_{j+1}}=s_{i_{k-1}}\cdots s_{i_{j+1}}s_{i_j}.\]
  In particular, we have
  \[w=s_{i_{\ell}}\cdots s_{i_k}\cdots s_{i_{j+1}} s_{i_j}\cdots s_{i_1}=
  s_{i_{\ell}}\cdots \widehat{s_{i_k}}\cdots \widehat{s_{i_j}}\cdots s_{i_1}. \]
  
  Next we prove (2). Let $s_{i_{\ell}}\cdots s_{i_1}$ be a reduced expression of $w$.
  By (1), we have that
 \[ s_{i_{\ell}}\cdots s_{i_{j+1}}(i_j)<s_{i_{\ell}}\cdots s_{i_{j+1}}(i_j+1)\]
for any $j$. Hence, we obtain that
\[\gamma(w)=\gamma(s_{i_{\ell}}\cdots s_{i_2})+1=\cdots=\gamma(1)+\ell=\ell.\]  

Finally, we show the assertion (3).
  Suppose that 
  \[s_{i_{\ell}}\cdots s_{i_{j+1}}(i_j)< s_{i_{\ell}}\cdots s_{i_{j+1}}(i_j+1)\]
  for any $j\in\{1,\dots, \ell-1\}$.
   Then same argument used in the proof of (2) gives that
  \[\ell(w)=\gamma(w)=\ell.\]
  This is a contradiction.
  Therefore we have (i).
  \end{proof}
    
    We recall definition of the (left) weak order on $\Sym_{n+1}$.
\begin{definition}
Let $w,w'\in \Sym_{n+1}$. We write $w\leq w'$ if there exists $s_{i_1},\dots,s_{i_k}$ such that
\[w'=s_{i_k}\cdots s_{i_1}w\ \mathrm{and\ }\ell(w')=\ell(w)+k. \]
It is obvious that $\leq$ gives a partial order on $\Sym_{n+1}$. We call this partial order
the \emph{left weak order} on $\Sym_{n+1}$.  
\end{definition}

Clearly $(\Sym_{n+1},\leq)$ is a ranked poset by the length function $\ell$. Moreover,
 $(\Sym_{n+1},\leq)$ has the lattice properties i.e. for any $w,w'\in \Sym_{n+1}$, 
 $\{\sigma\in \Sym_{n+1} \mid \sigma\geq w,w'\}$ admits a maximum element $w\wedge w'$ and 
 $\{\sigma\in \Sym_{n+1} \mid \sigma\leq w,w'\}$ admits a minimum element $w\vee w'$.
  (see \cite[Section\;3.2]{BjB}). 
By definition the minimum element
of $(\Sym_{n+1},\leq)$ is the identity $1\in \Sym_{n+1}$ and the maximum element is the longest element
$w_0:=(n+1,n,\dots,1)\in \Sym_{n+1}$. 
Then the assignment $w\mapsto ww_0$ gives a poset isomorphism
\[(\Sym_{n+1},\leq)\stackrel{\sim}{\to}(\Sym_{n+1},\leq).\] 

For a non-empty subset $J\subset\{1,2,\dots,n\}$, we denote by $w_0(J)\in \Sym_{n+1}$ the longest element
of $\langle s_j\mid j\in J \rangle\subset \Sym_{n+1}$. Then we have the following.
\begin{proposition}
\label{fresultsym}
 Let $J$ be a non-empty subset of $\{1,\dots,n\}$.
\begin{enumerate}[{\rm (1)}]
\item\cite[Lemma\;3.2.3]{BjB} $\bigvee_{j\in J} s_j=w_0(J).$
\item $[1,w_0(J)]=\langle s_j\mid j\in J \rangle$.
\item\cite[Lemma\;3.24]{BjB} If $w\leq s_j w$ for any $j\in J$, then we have
\[\bigvee_{j\in J}(s_jw)=w_0(J)w.\]
\end{enumerate}
\end{proposition}
\begin{proof} We prove (2). Let $w\leq w_0(J)$.
 Suppose that $w\not\in \langle s_j\mid j\in J\rangle$.
$w\leq w_0(J)$ implies that there exists a reduced expression
\[s_{i_{\ell}}\cdots s_{i_1}\]
of $w_0(J)$ such that $R=\{r\mid i_t\not\in J \}\neq\emptyset$.  
We take a minimum element $r$ of $R$. 
Since $s_{i_{r-1}}\cdots s_{i_1}, w_0(J)\in \langle s_j\mid j\in J \rangle$,
we have that $w':=s_{i_{\ell}}\cdots s_{i_r}\in \langle s_j\mid j\in J \rangle$.
Then $s_{i_{\ell}}\cdots s_{i_{r+1}}(i_r)=w'(i_r+1)>w'(i_r)=s_{i_{\ell}}\cdots s_{i_{r+1}}(i_r+1)$.
Hence Theorem\;\ref{matsumoto} gives that $s_{i_{\ell}}\cdots s_{i_r}$ is non-reduced.
This is a contradiction.
\end{proof}

\section{Main result}
Let $\Lambda=kQ/I$ be a basic finite dimensional algebra, where $I$ is an admissible ideal of $kQ$.
We consider the following condition.
\begin{condition}
\label{nscd}
\begin{enumerate}[(a)]
\item 
$Q^{\circ}$ is isomorphic to the following quiver:
\[
{\unitlength 0.1in%
\begin{picture}( 16.3000,  1.3400)( 22.2000, -7.5400)%
\put(22.2000,-7.5000){\makebox(0,0)[lb]{$1$}}%
%
\special{pn 8}%
\special{pa 2360 670}%
\special{pa 2560 670}%
\special{fp}%
\special{sh 1}%
\special{pa 2560 670}%
\special{pa 2493 650}%
\special{pa 2507 670}%
\special{pa 2493 690}%
\special{pa 2560 670}%
\special{fp}%
%
\special{pn 8}%
\special{pa 2530 720}%
\special{pa 2330 720}%
\special{fp}%
\special{sh 1}%
\special{pa 2330 720}%
\special{pa 2397 740}%
\special{pa 2383 720}%
\special{pa 2397 700}%
\special{pa 2330 720}%
\special{fp}%
\put(25.7000,-7.6000){\makebox(0,0)[lb]{$2$}}%
%
\special{pn 8}%
\special{pa 2720 670}%
\special{pa 2920 670}%
\special{fp}%
\special{sh 1}%
\special{pa 2920 670}%
\special{pa 2853 650}%
\special{pa 2867 670}%
\special{pa 2853 690}%
\special{pa 2920 670}%
\special{fp}%
%
\special{pn 8}%
\special{pa 2890 720}%
\special{pa 2690 720}%
\special{fp}%
\special{sh 1}%
\special{pa 2690 720}%
\special{pa 2757 740}%
\special{pa 2743 720}%
\special{pa 2757 700}%
\special{pa 2690 720}%
\special{fp}%
\put(29.3000,-7.6000){\makebox(0,0)[lb]{$3$}}%
%
\special{pn 8}%
\special{pa 3080 677}%
\special{pa 3280 677}%
\special{fp}%
\special{sh 1}%
\special{pa 3280 677}%
\special{pa 3213 657}%
\special{pa 3227 677}%
\special{pa 3213 697}%
\special{pa 3280 677}%
\special{fp}%
%
\special{pn 8}%
\special{pa 3250 727}%
\special{pa 3050 727}%
\special{fp}%
\special{sh 1}%
\special{pa 3050 727}%
\special{pa 3117 747}%
\special{pa 3103 727}%
\special{pa 3117 707}%
\special{pa 3050 727}%
\special{fp}%
%
\special{pn 8}%
\special{pa 3630 677}%
\special{pa 3830 677}%
\special{fp}%
\special{sh 1}%
\special{pa 3830 677}%
\special{pa 3763 657}%
\special{pa 3777 677}%
\special{pa 3763 697}%
\special{pa 3830 677}%
\special{fp}%
%
\special{pn 8}%
\special{pa 3800 727}%
\special{pa 3600 727}%
\special{fp}%
\special{sh 1}%
\special{pa 3600 727}%
\special{pa 3667 747}%
\special{pa 3653 727}%
\special{pa 3667 707}%
\special{pa 3600 727}%
\special{fp}%
%
\special{pn 4}%
\special{sh 1}%
\special{ar 3320 700 8 8 0  6.28318530717959E+0000}%
\special{sh 1}%
\special{ar 3520 700 8 8 0  6.28318530717959E+0000}%
%
\special{pn 4}%
\special{sh 1}%
\special{ar 3420 700 8 8 0  6.28318530717959E+0000}%
\put(38.5000,-7.6000){\makebox(0,0)[lb]{$n$}}%
\end{picture}}%
\]
\item  For each arrow $x:i\to j$ with $i\neq j$ in $Q$, 
$x\Lambda e_j=e_{i}\Lambda e_{j}=e_{i}\Lambda x$.
\item  For any pair $(i,j)$ of $Q_0$, $w_j^i\neq 0$ in $\Lambda$, where $w_j^i$ is the shortest path from $i$ to $j$ in $Q$.
\end{enumerate} 
\end{condition}
\begin{example}
\label{examplenscd}
\begin{enumerate}[{\rm (1)}]
\item Let $\overrightarrow{\Delta}$ be a quiver of type $A_n$ with linear orientation. $Q$ denotes the double quiver 
of $\overrightarrow{\Delta}$ i.e.
$Q_0:=\overrightarrow{\Delta}_0$ and 
$Q_1:=\overrightarrow{\Delta}_1\sqcup\{\alpha^*:t(\alpha)\to s(\alpha)\mid \alpha\in \overrightarrow{\Delta}_1\}$.
Then $\Pi_n:=kQ/(\sum_{\alpha\in \overrightarrow{\Delta}_1}\alpha\alpha^*-\alpha^*\alpha)$ is said to be the
 \emph{preprojective algebra} of type $A_n$. 
 
We can easily check  that the preprojective algebra $\Pi_n$ of type $A_n$  satisfies the Condition\;\ref{nscd}.
In fact, (a) and (c) of the Condition\;\ref{nscd} obviously hold. 
Let $\alpha$ be an arrow from $x$ to $y$. Then the relation 
$\sum_{\alpha\in \overrightarrow{\Delta}_1} (\alpha\alpha^*-\alpha^*\alpha) $
 induces that for any path $w$ from $x$ to $y$ on $Q$,
there exists $N$ such that
\[w=(\alpha\alpha^*)^N\alpha=\alpha(\alpha^*\alpha)^N.\]
This gives (b) of the Condition\;\ref{nscd}.  
\item The Auslander algebra of the truncated polynomial ring $k[X]/(X^n)$ satisfies the Condition\;\ref{nscd}.
\item Let $Q$ be the following quiver:
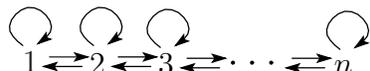
\[
{\unitlength 0.1in%
\begin{picture}( 18.8000,  3.4400)( 21.5000, -7.5400)%
\put(22.2000,-7.5000){\makebox(0,0)[lb]{$1$}}%
%
\special{pn 8}%
\special{pa 2360 670}%
\special{pa 2560 670}%
\special{fp}%
\special{sh 1}%
\special{pa 2560 670}%
\special{pa 2493 650}%
\special{pa 2507 670}%
\special{pa 2493 690}%
\special{pa 2560 670}%
\special{fp}%
%
\special{pn 8}%
\special{pa 2530 720}%
\special{pa 2330 720}%
\special{fp}%
\special{sh 1}%
\special{pa 2330 720}%
\special{pa 2397 740}%
\special{pa 2383 720}%
\special{pa 2397 700}%
\special{pa 2330 720}%
\special{fp}%
\put(25.7000,-7.6000){\makebox(0,0)[lb]{$2$}}%
%
\special{pn 8}%
\special{pa 2720 670}%
\special{pa 2920 670}%
\special{fp}%
\special{sh 1}%
\special{pa 2920 670}%
\special{pa 2853 650}%
\special{pa 2867 670}%
\special{pa 2853 690}%
\special{pa 2920 670}%
\special{fp}%
%
\special{pn 8}%
\special{pa 2890 720}%
\special{pa 2690 720}%
\special{fp}%
\special{sh 1}%
\special{pa 2690 720}%
\special{pa 2757 740}%
\special{pa 2743 720}%
\special{pa 2757 700}%
\special{pa 2690 720}%
\special{fp}%
\put(29.3000,-7.6000){\makebox(0,0)[lb]{$3$}}%
%
\special{pn 8}%
\special{pa 3080 677}%
\special{pa 3280 677}%
\special{fp}%
\special{sh 1}%
\special{pa 3280 677}%
\special{pa 3213 657}%
\special{pa 3227 677}%
\special{pa 3213 697}%
\special{pa 3280 677}%
\special{fp}%
%
\special{pn 8}%
\special{pa 3250 727}%
\special{pa 3050 727}%
\special{fp}%
\special{sh 1}%
\special{pa 3050 727}%
\special{pa 3117 747}%
\special{pa 3103 727}%
\special{pa 3117 707}%
\special{pa 3050 727}%
\special{fp}%
%
\special{pn 8}%
\special{pa 3630 677}%
\special{pa 3830 677}%
\special{fp}%
\special{sh 1}%
\special{pa 3830 677}%
\special{pa 3763 657}%
\special{pa 3777 677}%
\special{pa 3763 697}%
\special{pa 3830 677}%
\special{fp}%
%
\special{pn 8}%
\special{pa 3800 727}%
\special{pa 3600 727}%
\special{fp}%
\special{sh 1}%
\special{pa 3600 727}%
\special{pa 3667 747}%
\special{pa 3653 727}%
\special{pa 3667 707}%
\special{pa 3600 727}%
\special{fp}%
%
\special{pn 4}%
\special{sh 1}%
\special{ar 3320 700 8 8 0  6.28318530717959E+0000}%
\special{sh 1}%
\special{ar 3520 700 8 8 0  6.28318530717959E+0000}%
%
\special{pn 4}%
\special{sh 1}%
\special{ar 3420 700 8 8 0  6.28318530717959E+0000}%
\put(38.5000,-7.6000){\makebox(0,0)[lb]{$n$}}%
%
\special{pn 8}%
\special{pa 2220 610}%
\special{pa 2192 587}%
\special{pa 2169 562}%
\special{pa 2153 536}%
\special{pa 2151 507}%
\special{pa 2161 476}%
\special{pa 2183 447}%
\special{pa 2212 424}%
\special{pa 2245 411}%
\special{pa 2280 413}%
\special{pa 2314 427}%
\special{pa 2342 452}%
\special{pa 2362 481}%
\special{pa 2370 512}%
\special{pa 2364 540}%
\special{pa 2345 564}%
\special{pa 2318 587}%
\special{pa 2300 600}%
\special{fp}%
%
\special{pn 8}%
\special{pa 2318 587}%
\special{pa 2300 600}%
\special{fp}%
\special{sh 1}%
\special{pa 2300 600}%
\special{pa 2366 577}%
\special{pa 2343 569}%
\special{pa 2342 545}%
\special{pa 2300 600}%
\special{fp}%
%
\special{pn 8}%
\special{pa 2590 610}%
\special{pa 2562 587}%
\special{pa 2539 562}%
\special{pa 2523 536}%
\special{pa 2521 507}%
\special{pa 2531 476}%
\special{pa 2553 447}%
\special{pa 2582 424}%
\special{pa 2615 411}%
\special{pa 2650 413}%
\special{pa 2684 427}%
\special{pa 2712 452}%
\special{pa 2732 481}%
\special{pa 2740 512}%
\special{pa 2734 540}%
\special{pa 2715 564}%
\special{pa 2688 587}%
\special{pa 2670 600}%
\special{fp}%
%
\special{pn 8}%
\special{pa 2688 587}%
\special{pa 2670 600}%
\special{fp}%
\special{sh 1}%
\special{pa 2670 600}%
\special{pa 2736 577}%
\special{pa 2713 569}%
\special{pa 2712 545}%
\special{pa 2670 600}%
\special{fp}%
%
\special{pn 8}%
\special{pa 2940 620}%
\special{pa 2912 597}%
\special{pa 2889 572}%
\special{pa 2873 546}%
\special{pa 2871 517}%
\special{pa 2881 486}%
\special{pa 2903 457}%
\special{pa 2932 434}%
\special{pa 2965 421}%
\special{pa 3000 423}%
\special{pa 3034 437}%
\special{pa 3062 462}%
\special{pa 3082 491}%
\special{pa 3090 522}%
\special{pa 3084 550}%
\special{pa 3065 574}%
\special{pa 3038 597}%
\special{pa 3020 610}%
\special{fp}%
%
\special{pn 8}%
\special{pa 3038 597}%
\special{pa 3020 610}%
\special{fp}%
\special{sh 1}%
\special{pa 3020 610}%
\special{pa 3086 587}%
\special{pa 3063 579}%
\special{pa 3062 555}%
\special{pa 3020 610}%
\special{fp}%
%
\special{pn 8}%
\special{pa 3880 630}%
\special{pa 3852 607}%
\special{pa 3829 582}%
\special{pa 3813 556}%
\special{pa 3811 527}%
\special{pa 3821 496}%
\special{pa 3843 467}%
\special{pa 3872 444}%
\special{pa 3905 431}%
\special{pa 3940 433}%
\special{pa 3974 447}%
\special{pa 4002 472}%
\special{pa 4022 501}%
\special{pa 4030 532}%
\special{pa 4024 560}%
\special{pa 4005 584}%
\special{pa 3978 607}%
\special{pa 3960 620}%
\special{fp}%
%
\special{pn 8}%
\special{pa 3978 607}%
\special{pa 3960 620}%
\special{fp}%
\special{sh 1}%
\special{pa 3960 620}%
\special{pa 4026 597}%
\special{pa 4003 589}%
\special{pa 4002 565}%
\special{pa 3960 620}%
\special{fp}%
\end{picture}}%
\vspace{5pt}\]
 $I_m$ denotes an admissible ideal of $kQ$ generated by 
 \[\{l_i\alpha_i,l_{i+1}\alpha_i^*, 
  \alpha_i l_{i+1},\alpha_i^* l_i,l_i^m,\alpha_i\alpha_i^*,\alpha_i^*\alpha_i\mid i\in Q_0\}, \]
  where $\alpha_i$ (resp. $\alpha_i^*$) is the arrow from $i$ to $i+1$ (resp. from $i+1$ to $i$) and $l_i$ is the loop on $i$
 (in the case that $m=1$, we regard $Q=Q^{\circ}$ and $I_1$ is generated by 
  $\{\alpha_i\alpha_i^*,\alpha_i^*\alpha_i\mid i\in Q_0\}$).
 Then $\Lambda_m:=kQ/I_m$ satisfies the Condition\;\ref{nscd}.
 
 We remark that for any algebra $\Lambda$ satisfying the Condition\;\ref{nscd}, there is a
 surjective algebra homomorphism from $\Lambda$ to $\Lambda_1$.
 \item Let $Q$ be the following quiver.
 \[
{\unitlength 0.1in%
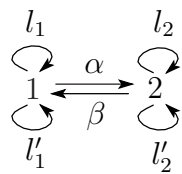
\begin{picture}(  8.7000,  7.1000)( 17.3000,-12.5000)%
\put(18.0000,-10.0000){\makebox(0,0)[lb]{$1$}}%
%
\special{pn 8}%
\special{pa 1960 920}%
\special{pa 2360 920}%
\special{fp}%
\special{sh 1}%
\special{pa 2360 920}%
\special{pa 2293 900}%
\special{pa 2307 920}%
\special{pa 2293 940}%
\special{pa 2360 920}%
\special{fp}%
%
\special{pn 8}%
\special{pa 2340 970}%
\special{pa 1940 970}%
\special{fp}%
\special{sh 1}%
\special{pa 1940 970}%
\special{pa 2007 990}%
\special{pa 1993 970}%
\special{pa 2007 950}%
\special{pa 1940 970}%
\special{fp}%
\put(21.1000,-8.6000){\makebox(0,0)[lb]{$\alpha$}}%
\put(21.2000,-11.6000){\makebox(0,0)[lb]{$\beta$}}%
\put(24.4000,-9.9000){\makebox(0,0)[lb]{$2$}}%
%
\special{pn 8}%
\special{pa 1810 850}%
\special{pa 1774 832}%
\special{pa 1745 814}%
\special{pa 1730 793}%
\special{pa 1736 770}%
\special{pa 1758 747}%
\special{pa 1792 729}%
\special{pa 1833 720}%
\special{pa 1873 722}%
\special{pa 1909 735}%
\special{pa 1933 754}%
\special{pa 1940 777}%
\special{pa 1929 802}%
\special{pa 1905 828}%
\special{pa 1880 850}%
\special{fp}%
%
\special{pn 8}%
\special{pa 1905 828}%
\special{pa 1880 850}%
\special{fp}%
\special{sh 1}%
\special{pa 1880 850}%
\special{pa 1943 821}%
\special{pa 1920 815}%
\special{pa 1917 791}%
\special{pa 1880 850}%
\special{fp}%
%
\special{pn 8}%
\special{pa 2470 860}%
\special{pa 2434 842}%
\special{pa 2405 824}%
\special{pa 2390 803}%
\special{pa 2396 780}%
\special{pa 2418 757}%
\special{pa 2452 739}%
\special{pa 2493 730}%
\special{pa 2533 732}%
\special{pa 2569 745}%
\special{pa 2593 764}%
\special{pa 2600 787}%
\special{pa 2589 812}%
\special{pa 2565 838}%
\special{pa 2540 860}%
\special{fp}%
%
\special{pn 8}%
\special{pa 2565 838}%
\special{pa 2540 860}%
\special{fp}%
\special{sh 1}%
\special{pa 2540 860}%
\special{pa 2603 831}%
\special{pa 2580 825}%
\special{pa 2577 801}%
\special{pa 2540 860}%
\special{fp}%
%
\special{pn 8}%
\special{pa 1810 1030}%
\special{pa 1774 1048}%
\special{pa 1745 1066}%
\special{pa 1730 1087}%
\special{pa 1736 1110}%
\special{pa 1758 1133}%
\special{pa 1792 1151}%
\special{pa 1833 1160}%
\special{pa 1873 1158}%
\special{pa 1909 1145}%
\special{pa 1933 1126}%
\special{pa 1940 1103}%
\special{pa 1929 1078}%
\special{pa 1905 1052}%
\special{pa 1880 1030}%
\special{fp}%
%
\special{pn 8}%
\special{pa 1905 1052}%
\special{pa 1880 1030}%
\special{fp}%
\special{sh 1}%
\special{pa 1880 1030}%
\special{pa 1917 1089}%
\special{pa 1920 1065}%
\special{pa 1943 1059}%
\special{pa 1880 1030}%
\special{fp}%
%
\special{pn 8}%
\special{pa 2470 1030}%
\special{pa 2434 1048}%
\special{pa 2405 1066}%
\special{pa 2390 1087}%
\special{pa 2396 1110}%
\special{pa 2418 1133}%
\special{pa 2452 1151}%
\special{pa 2493 1160}%
\special{pa 2533 1158}%
\special{pa 2569 1145}%
\special{pa 2593 1126}%
\special{pa 2600 1103}%
\special{pa 2589 1078}%
\special{pa 2565 1052}%
\special{pa 2540 1030}%
\special{fp}%
%
\special{pn 8}%
\special{pa 2565 1052}%
\special{pa 2540 1030}%
\special{fp}%
\special{sh 1}%
\special{pa 2540 1030}%
\special{pa 2577 1089}%
\special{pa 2580 1065}%
\special{pa 2603 1059}%
\special{pa 2540 1030}%
\special{fp}%
\put(18.0000,-6.7000){\makebox(0,0)[lb]{$l_1$}}%
\put(18.0000,-13.6000){\makebox(0,0)[lb]{$l'_1$}}%
\put(24.7000,-6.7000){\makebox(0,0)[lb]{$l_2$}}%
\put(24.6000,-13.8000){\makebox(0,0)[lb]{$l'_2$}}%
\end{picture}}%
 \vspace{5pt}
 \]
 Let $I$ be an admissible ideal of $kQ$ generated by 
 \[\{(\alpha\beta)^2,(\beta\alpha)^2,l_1\alpha-\alpha(l_2+l'_2),
 l'_1\alpha-\alpha l'_2, l_2\beta-\beta l'_1, l_2'\beta-\beta l_1,
  l_i^2,{l'_{i}}^{2}, l_il'_i,l'_il_i\ (i=1,2)
 \}\]
Then $\Gamma=kQ/I$ satisfies the Condition\;\ref{nscd}. 
\end{enumerate}
\end{example}
Main result of this paper is the following.
\begin{theorem}
\label{mainresult}
Let $\Lambda=kQ/I$ is a finite dimensional algebra with $I$ being an admissible ideal of $kQ$.
Then $\sttilt \Lambda\simeq (\Sym_{n+1},\leq)$ if and only if $\Lambda$ satisfies the Condition\;\ref{nscd}.
\end{theorem}
\begin{remark}
By using \cite[Theorem\;4.1]{EJR} (and Theorem\;\ref{mizuno}), 
we can construct infinitely many algebras whose support $\tau$-tilting posets
are isomorphic to $(\Sym_{n+1},\leq)$. In fact, let $\Lambda_m$ be the algebra considered in
Example\;\ref{examplenscd} (3). Then $z_m:=l_1^{m-1}+l_2^{m-1}+\cdots+l_n^{m-1}$ is in $\Rad \Lambda\cap Z(\Lambda_m)$,
where $Z(\Lambda_m)$ is the center of $\Lambda_m$. It is easy to check that $\Lambda_m/(z_m)=\Lambda_{m-1}$.
Hence \cite[Theorem\;4.1]{EJR} implies that $\sttilt \Lambda_m\simeq \sttilt \Lambda_1$ for any $m\geq 1$.
Also by using \cite[Theorem\;4.1]{EJR} and Theorem\;\ref{mizuno}, we have that 
\[\sttilt \Lambda_m\simeq \sttilt \Lambda_1\simeq \sttilt \Pi_n\simeq (\Sym_{n+1},\leq).\]
Therefore \cite[Theorem\;4.1]{EJR} is very powerful. But unfortunately, there exists
an algebra $\Lambda$ with $\sttilt \Lambda\simeq (\Sym_{n+1},\leq)$ such that
we can't prove $\sttilt \Lambda\simeq (\Sym_{n+1},\leq)$ by using \cite[Theorem\;4.1]{EJR} and Theorem\;\ref{mizuno}.

We denote by $\Lambda\leadsto \Lambda'$ if there exists $z\in \Rad \Lambda\cap Z(\Lambda)$ such that $\Lambda/(z)\cong \Lambda'$.
Let $\sim$ be the  equivalence relation on the set of (isomorphism classes of) basic finite dimensional algebras generated
by $\leadsto$.  Then we can show that
 \[\Gamma\not\sim \Pi_2,\]
where $\Gamma$ is the algebra considered in Example\;\ref{examplenscd} (4). Note that if 
$\Lambda \leadsto \Lambda^{(1)}, \Lambda^{(2)},$
then there is $\Lambda' $ such that $\Lambda^{(1)},\Lambda^{(2)}\leadsto \Lambda'$. In particular,
$\Lambda^{(1)}\sim \Lambda^{(2)}$ if and only if there exists $\Lambda$ such that
\[\Lambda^{(1)}\leadsto \cdots\leadsto \Lambda\ \mathrm{and}\ \Lambda^{(2)}\leadsto\cdots\leadsto \Lambda.\]
Now suppose that $\Gamma\sim \Pi_2(\sim \Lambda_1)$. Since $\Lambda_1$ has no non-zero element
in $\Rad \Lambda_1\cap Z(\Lambda_1)$, there is a path
\[\Gamma\leadsto\cdots\leadsto \Lambda_1.\]
Let $z\in \Rad \Gamma\cap Z(\Gamma)$. Since $\Rad \Gamma\cap Z(\Gamma)\subset e_1(\Rad \Gamma)e_1+e_2(\Rad \Gamma)e_2$ and
\[l_1,l_1' \alpha\beta,l_1\alpha\beta,l_1'\alpha\beta,l_2,l_2',\beta\alpha,l_2\beta\alpha,l_2'\beta\alpha\] form a basis
of $e_1(\Rad \Gamma)e_1+e_2(\Rad \Gamma)e_2$, we can write
\[z=al_1+bl_1'+c\alpha\beta+dl_1\alpha\beta+el_1'\alpha\beta+a'l_2+b'l_2'+c'\beta\alpha+d'l_2\beta\alpha+e'l_2'\beta\alpha.\]
By $z\alpha=\alpha z$ and $z\beta=\beta z$, we obtain that $a=a'=b=b'=d=d'=e=e'=0$ and $c=c'$.
Then $c=c'=0$ follows from $l_1z=zl_1$. This implies that $\Rad \Gamma\cap Z(\Gamma)=0$.
Therefore we have that $\Gamma\simeq \Lambda_1$ and reach a contradiction.
 $\Gamma\not\sim \Pi_2$ says that we can't prove $\sttilt\Gamma\simeq (\Sym_{3},\leq)$
by using \cite[Theorem\;4.1]{EJR} and Theorem\;\ref{mizuno}.    
\end{remark}
\section{Proof of Theorem\;\ref{mainresult}}
In this section, we give a proof of Theorem\;\ref{mainresult}. From now on, we put 
$P_i=e_i \Lambda$ the indecomposable projective module of $\Lambda$ associated with $i\in Q_0$.
Also we put $X_i:=e_i\Lambda/e_i\Lambda(1-e_i)\Lambda\simeq \Lambda/(1-e_i)$. Note that
$X_i$ is in $\sttilt \Lambda$ with $\supp(X_i)=\{i\}$. Therefore we have that
$\dip(0)=\{X_i\mid i\in Q_0\}.$

\subsection{Case  $n=2$}
\label{subsec:n=2}
In this subsection, we see that Theorem\;\ref{mainresult} hold for the case that $n=2$.
The following results are proved in \cite{AK}. We give  proofs for reader's convenience.
\begin{lemma}\label{1to2}\cite{AK}
Let $Q$ be a quiver with precisely two vertices, say  $1,2$,
 and $I$  an admissible ideal of $kQ$. Suppose that there is an arrow $\alpha$ from $1$ to $2$.
Put $\Lambda:=kQ/I$.
Then $X_1\oplus P_2 $ is not $\tau$-rigid. Moreover, we have
$\Hom_{\Lambda}(P_1,\tau X_1)=0$ if and only if 
$\alpha$ is a unique arrow from $1$ to $2$ and 
$\alpha\Lambda e_2=e_1\Lambda e_2=e_1\Lambda\alpha$.  
\end{lemma}
\begin{proof}
Let $P_{X_1}=[P_2^r\stackrel{d}{\to} P_1]$ be a two term presilting complex associated with $X_1$.
We can easily  see that $\Hom_{\Lambda}(P_2,\tau X_1)\neq 0$. In fact,
we have that $\Hom_{\Kb(\proj \Lambda)}(P_{X_1},P_2[1])\neq 0$.

We assume that $\Hom_{\Kb(\proj \Lambda)}(P_{X_1},P_1[1])= 0(\Leftrightarrow \Hom_{\Lambda}(P_1,\tau X_1)=0)$.
Then any homomorphism $f\in \Hom_{\Lambda}(P_2^r, P_1)$ factors through $d$.
By considering $f=(\overbrace{\alpha,0,\dots,0}^r):P_2^r\to P_1$, we obtain that $r=1$.
In particular, we may regard $d$ as an element of $e_1\Lambda e_2$ and get that
\[(e_1\Lambda e_1)d=e_1\Lambda e_2.\]
Hence there exists 
 $x\in e_1 \Lambda e_1\setminus e_1 (\Rad \Lambda) e_1$ such that $\alpha=xd$.
We conclude that
\[e_1\Lambda \alpha=e_1\Lambda xd=e_1 \Lambda d=e_1\Lambda e_2. \]
Note that $\Im d=d \Lambda=e_1\Lambda e_2 \Lambda$. Therefore
there also exists 
 $y\in e_2 \Lambda e_2\setminus e_2 (\Rad \Lambda) e_2$ such that $\alpha=dy$.
This implies that
\[\alpha \Lambda=dy\Lambda=d\Lambda=e_1 \Lambda e_2 \Lambda.\] 
Hence, $\alpha$ is a unique arrow from 1 to 2 and we have 
\[\alpha\Lambda e_2=e_1\Lambda e_2=e_1\Lambda \alpha.\]

If $\alpha$ is a unique arrow from 1 to 2 and  
$\alpha\Lambda e_2=e_1\Lambda e_2=e_1\Lambda \alpha$ holds,
then it is easy to check that 
\[P_{X_1}=[P_2\stackrel{\alpha}{\to}P_1]\]
gives a minimal projective presentation of $X_1$ 
and $\Hom_{\Kb(\proj\Lambda)}(P_{X_1},P_1[1])=0$.
\end{proof}


The following proposition is an immediate consequence of Lemma \ref{1to2}.

\begin{proposition}\label{mr}
Let $Q$ and $\Lambda$ be as in Lemma \ref{1to2}.
\item $\sttilt\Lambda\simeq (\Sym_{3},\leq)$ if and only if the following hold.
\begin{enumerate}[{\rm (a)}]
\item There exists a unique arrow $\alpha$ of $Q$ from $1$ to $2$ and $\alpha\Lambda e_2=e_1\Lambda e_2=e_1\Lambda \alpha$.
\item There exists a unique arrow $\beta$ of $Q$ from $2$ to $1$ and $\beta\Lambda e_1=e_2\Lambda e_1=e_2\Lambda \beta$.
\end{enumerate} 
\end{proposition}

\begin{proof}

Since $\dip(0)=\{X_1,X_2\}$, $\sttilt \Lambda\simeq (\Sym_3,\leq)$ if and only if
there are a path from $\Lambda$ to $X_1$ with length 2 and a path from $\Lambda$ to $X_2$ with length 2.
This is equivalent to that $X_1$, $X_2$ are not projective and $X_1\oplus P_i, X_2\oplus P_j$ are
$\tau$-rigid with $\{i,j\}=\{1,2\}$.
Then the assertion  follows from Lemma\;\ref{1to2}.
\end{proof}

\subsection{`if' part}
\label{subsec:ifpart}
We assume that $\Lambda$ satisfies the Condition\;\ref{nscd}. 

For vertices $i,j\in Q_0$, we denote by 
$f_i^j$ the homomorphism from $P_i$ to $P_j$ given by the path $w_i^j$ i.e. $f_i^j(e_i \lambda )=w_i^j \lambda\in P_j.$
\begin{lemma}
\label{pathcondition}
\begin{enumerate}[{\rm (1)}]
\item Let $\alpha:i\to j$ be an arrow of $Q^{\circ}$. Then we have 
\[ e_i (\Rad \Lambda)  \alpha = \alpha (\Rad \Lambda) e_{j}\]
\item  Let $w\in e_i \Lambda e_j $ 
 Then there is a unique $a\in k$ and
(not necessary unique) $l\in e_i(\Rad \Lambda) e_i$ such that
$w=(a+l)w_j^i$ in $\Lambda$. Also there is a unique $a'\in k$ and
(not necessary unique) $l'\in e_j(\Rad \Lambda) e_j$ such that
$w=w_j^i(a'+l')$ in $\Lambda$. Furthermore, we have $a=a'$. 
\item Let $f\in \Hom_{\Lambda}(P_j,P_i)$ and $V\subset Q_0$. Assume that $f$ is given by  $w=(ae_i+l)w_j^i$ with $a\neq 0$.
Then $f$ is factors through $P=\bigoplus_{t\in V} P_t$ if and only if $w_i^j$ factors through
some $t\in V$. 
\end{enumerate}
\end{lemma}
\begin{proof}
We show (1). Let $l\in e_i(\Rad \Lambda) e_i$. Then 
\[l\alpha \in e_i \Lambda e_j = \alpha \Lambda e_j. \]
Accordingly, there exists $a\in k$ and $l'\in e_j(\Rad \Lambda) e_j$ such that
\[l\alpha=\alpha (a+l'). \]
If $a\neq 0$, then $a+l'$ is an invertible element in $\Lambda=kQ/I$.
Therefore we have that $\alpha- l\alpha l''\equiv 0 \mod I$
for some $l''\in kQ$. Since $I$ is admissible, we reach a contradiction.
Hence we conclude that
\[e_i(\Rad \Lambda)\alpha=e_i (\Rad \Lambda) e_i \alpha \subset \alpha (\Rad \Lambda) e_{j}.\]
Similarly, one can check that
\[e_i (\Rad \Lambda) e_i \alpha \supset \alpha e_j(\Rad \Lambda) e_{j}.\] 

Next we prove (2). Existence of $a\in k$ and $l\in e_i(\Rad \Lambda) e_i$
 directly follows from Condition\;\ref{nscd}\;(b). Suppose that
 \[w=(a_1+l_1)w_j^i=(a_2+l_2)w_j^i.\]
 It is sufficient to show that $a_1=a_2$.
 If $a_1\neq a_2$, then $(a_1-a_2+l_1-l_2)w_j^i=0$ and $a_1-a_2+l_1-l_2$ is invertible.
 In particular, we obtain $w_j^i=0$. This  contradicts to Condition\;\ref{nscd}\;(c).  
 Same argument gives that there are unique $a'\in k$ and
 $l'\in e_j(\Rad \Lambda) e_j$ such that
$w=w_j^i(a'+l')$ in $\Lambda$. Then $a=a'$ follows from (1).
 The assertion (3) follows from (1) and (2). 
\end{proof}
\begin{theorem}\cite{M}
\label{Mizunosresult} Let $\Lambda$ be the preprojective algebra of type $A_n$.
For $i\in Q_0$, we let $I_i=(1-e_i)$. 
\begin{enumerate}[{\rm (1)}]
\item $($See also \cite[I\hspace{-.1em}I\hspace{-.1em}I]{BIRS}$)$. Let $w\in \Sym_{n+1}$. If $s_{i_{\ell}}\cdots s_{i_1}$ and $s_{j_{\ell}}\cdots s_{j_1}$ are reduced 
expression of $w$, then we have 
\[I_{i_{\ell}}\cdots I_{i_1}=I_{j_{\ell}}\cdots I_{j_1}.\]   
In this case, we denote $I_{i_{\ell}}\cdots I_{i_1}$ by $I_w$. 
\item $I_w\in \sttilt \Lambda$.
\item The map  $w\mapsto I_{ww_0}$ gives a poset isomorphism
\[(\Sym_{n+1},\leq)\stackrel{\sim}{\to} \sttilt \Lambda.\]
\end{enumerate}
\end{theorem}
We let $\Xi:=\{\mathbf{i}=\{i_{0}<i_1<\cdots<i_{2m}\}\subset \{0,1,\dots,n+1\}\mid m\geq 0,
\ I\neq \{0\},\{n+1\}\}$. For $\Xi\ni \mathbf{i}=\{i_{0}<i_1<\cdots<i_{2m}\}$, we set 
$m_{\mathbf{i}}:=m$.
\begin{lemma}
\label{trigidofpreproj}
 If $\Lambda=kQ/I$ is a preprojective algebra of type $A_n$, then
\[\#\trigid \Lambda \leq \# \{ \mathbf{i}\in \Xi\mid m_{\mathbf{i}}>0 \}.\]
\end{lemma}
\begin{proof}
Since any indecomposable $\tau$-rigid module $X$ is in $\add I_w$ for some $w\in \Sym_{n+1}$ and 
$I_w=e_1 I_w\oplus\cdots \oplus e_n I_w$, one sees that
\[\trigid \Lambda=\{e_i I_w\mid 1\leq i\leq n,\;w\in \Sym_{n+1}\}\setminus\{0\}.\]

Note that $e_i \Lambda$ has following form (loewy  series).
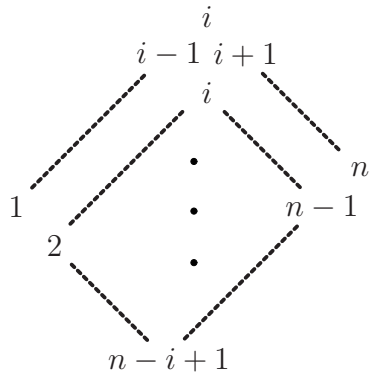
\[
{\unitlength 0.1in%
\begin{picture}( 18.0000, 18.0000)( 27.1000,-23.3500)%
\put(38.0000,-6.0000){\makebox(0,0){$i$}}%
\put(36.0000,-8.0000){\makebox(0,0){$i-1$}}%
\put(40.0000,-8.0000){\makebox(0,0){$i+1$}}%
\put(28.0000,-16.0000){\makebox(0,0){$1$}}%
\put(30.0000,-18.0000){\makebox(0,0){$2$}}%
\put(36.0000,-24.0000){\makebox(0,0){$n-i+1$}}%
\put(38.0000,-10.0000){\makebox(0,0){$i$}}%
\put(46.0000,-14.0000){\makebox(0,0){$n$}}%
\put(44.0000,-16.0000){\makebox(0,0){$n-1$}}%
%
\special{pn 20}%
\special{pa 4090 900}%
\special{pa 4490 1300}%
\special{dt 0.045}%
%
\special{pn 20}%
\special{pa 3470 910}%
\special{pa 2870 1510}%
\special{dt 0.045}%
%
\special{pn 20}%
\special{pa 3090 1900}%
\special{pa 3490 2300}%
\special{dt 0.045}%
%
\special{pn 20}%
\special{pa 4270 1700}%
\special{pa 3670 2300}%
\special{dt 0.045}%
%
\special{pn 20}%
\special{pa 3670 1100}%
\special{pa 3070 1700}%
\special{dt 0.045}%
%
\special{pn 20}%
\special{pa 3890 1100}%
\special{pa 4290 1500}%
\special{dt 0.045}%
%
\special{pn 4}%
\special{sh 1}%
\special{ar 3730 1360 16 16 0  6.28318530717959E+0000}%
\special{sh 1}%
\special{ar 3730 1620 16 16 0  6.28318530717959E+0000}%
%
\special{pn 4}%
\special{sh 1}%
\special{ar 3730 1890 16 16 0  6.28318530717959E+0000}%
\end{picture}}%
\vspace{5pt}
\]
Let $w\in \Sym_{n+1}$ and $j\in \{1,\dots,n\}$. We show that $e_iI_w/e_i I_w(1-e_j)\Lambda$
is either a simple module or $0$.
Assume that \[e_iI_w/e_i I_w(1-e_j)\Lambda\neq 0.\]
Let $\lambda=e_i\lambda \in e_i I_w$ such that $\overline{\lambda}:=(\lambda+e_i I_w(1-e_j)\Lambda)\neq 0.$
Since $\overline{\lambda} (1-e_j)=0$, we may assume that
$\lambda\in e_i I_w e_j$. By the relation of $\Lambda$, one sees that
\[\lambda=w_j^i(a+(\alpha_j\alpha_j^*)f(\alpha_j\alpha_j^*))(\alpha_j\alpha_j^*)^N,\]
fore some $N\in\Z_{\geq 0}$, $a\in k\setminus \{0\}$ and $f(X)\in k[X]$.
 Hence we also may assume that $\lambda=w_j^i(\alpha_j\alpha_j^*)^N$.
Now we let $\lambda'\in e_i I_w$. Above argument implies that
there are $N'\in \Z_{\geq 0}$, $b\in k\setminus\{0\}$ and $g(X)\in k[X]$
such that $\lambda'=w_j^i(b+(\alpha_j\alpha_j^*)g(\alpha_j\alpha_j^*))(\alpha_j \alpha_j^*)^{N'}$.
Then $N>N'$ implies that $\overline{\lambda}=\overline{w_j^i(\alpha_j\alpha_j^*)^N}=0$. Therefore
one has that $N\leq N'$. In particular, $\overline{\lambda'}\in \overline{\lambda}\Lambda$.
Thus $e_iI_w/e_i I_w(1-e_j)\Lambda$ is a simple module associated with $j$.

Hence one obtains that $e_iI_w$ has following form. 
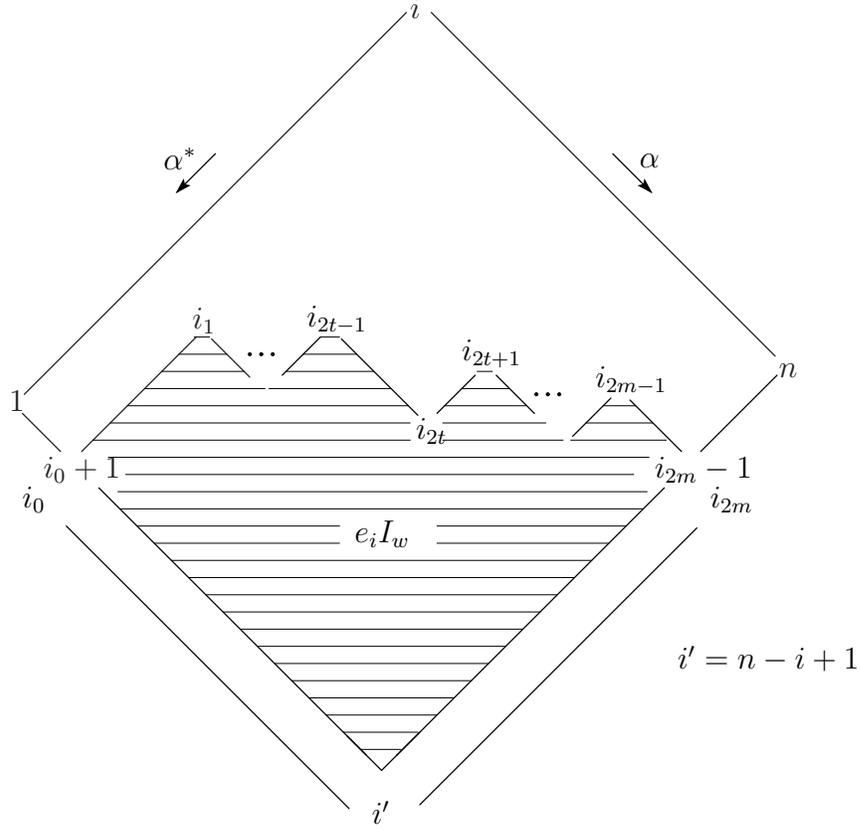
\[
{\unitlength 0.1in%
\begin{picture}( 40.3000, 42.7000)( 15.2000,-51.7000)%
%
\special{pn 8}%
\special{pa 1600 3000}%
\special{pa 3600 1000}%
\special{fp}%
%
\special{pn 8}%
\special{pa 3730 1000}%
\special{pa 5530 2800}%
\special{fp}%
\put(15.2000,-30.9000){\makebox(0,0)[lb]{$1$}}%
\put(36.2000,-10.3000){\makebox(0,0)[lb]{$i$}}%
%
\special{pn 8}%
\special{pa 1590 3100}%
\special{pa 1790 3300}%
\special{fp}%
\put(55.5000,-29.1000){\makebox(0,0)[lb]{$n$}}%
\put(17.0000,-34.6000){\makebox(0,0)[lb]{$i_0+1$}}%
%
\special{pn 8}%
\special{pa 1900 3300}%
\special{pa 2500 2700}%
\special{fp}%
\put(24.8000,-26.8000){\makebox(0,0)[lb]{$i_1$}}%
%
\special{pn 8}%
\special{pa 2580 2710}%
\special{pa 2780 2910}%
\special{fp}%
%
\special{pn 4}%
\special{sh 1}%
\special{ar 2780 2790 8 8 0  6.28318530717959E+0000}%
\special{sh 1}%
\special{ar 2840 2790 8 8 0  6.28318530717959E+0000}%
\special{sh 1}%
\special{ar 2900 2790 8 8 0  6.28318530717959E+0000}%
%
\special{pn 8}%
\special{pa 2950 2900}%
\special{pa 3150 2700}%
\special{fp}%
%
\special{pn 8}%
\special{pa 4270 3100}%
\special{pa 4070 2900}%
\special{fp}%
%
\special{pn 8}%
\special{pa 3260 2710}%
\special{pa 3660 3110}%
\special{fp}%
%
\special{pn 8}%
\special{pa 3760 3100}%
\special{pa 3960 2900}%
\special{fp}%
\put(30.8000,-26.8000){\makebox(0,0)[lb]{$i_{2t-1}$}}%
\put(36.5000,-32.6000){\makebox(0,0)[lb]{$i_{2t}$}}%
\put(38.9000,-28.6000){\makebox(0,0)[lb]{$i_{2t+1}$}}%
%
\special{pn 4}%
\special{sh 1}%
\special{ar 4280 3000 8 8 0  6.28318530717959E+0000}%
\special{sh 1}%
\special{ar 4340 3000 8 8 0  6.28318530717959E+0000}%
\special{sh 1}%
\special{ar 4400 3000 8 8 0  6.28318530717959E+0000}%
%
\special{pn 8}%
\special{pa 4470 3220}%
\special{pa 4670 3020}%
\special{fp}%
\put(45.9000,-30.0000){\makebox(0,0)[lb]{$i_{2m-1}$}}%
%
\special{pn 8}%
\special{pa 4770 3030}%
\special{pa 5040 3300}%
\special{fp}%
\put(49.0000,-34.6000){\makebox(0,0)[lb]{$i_{2m}-1$}}%
%
\special{pn 8}%
\special{pa 5540 2900}%
\special{pa 5140 3300}%
\special{fp}%
%
\special{pn 8}%
\special{pa 1990 3490}%
\special{pa 3390 4890}%
\special{fp}%
%
\special{pn 8}%
\special{pa 4950 3490}%
\special{pa 3550 4890}%
\special{fp}%
\put(33.3000,-38.0000){\makebox(0,0)[lb]{$e_iI_w$}}%
%
\special{pn 8}%
\special{pa 1860 3370}%
\special{pa 3460 4970}%
\special{pa 5100 3360}%
\special{pa 4710 2970}%
\special{pa 4440 3260}%
\special{pa 4010 2840}%
\special{pa 3710 3160}%
\special{pa 3200 2650}%
\special{pa 2870 2980}%
\special{pa 2530 2660}%
\special{pa 1840 3360}%
\special{pa 2200 3690}%
\special{ip}%
%
\special{pn 4}%
\special{pa 1950 3240}%
\special{pa 2150 3440}%
\special{pa 2040 3560}%
\special{ip}%
%
\special{pn 4}%
\special{pa 3250 3630}%
\special{pa 3760 3630}%
\special{pa 3760 3830}%
\special{pa 3250 3830}%
\special{pa 3250 3630}%
\special{ip}%
%
\special{pn 4}%
\special{pa 3620 3080}%
\special{pa 3620 3280}%
\special{pa 3790 3290}%
\special{pa 3790 3090}%
\special{ip}%
%
\special{pn 8}%
\special{pa 4970 3230}%
\special{pa 4780 3430}%
\special{pa 4980 3630}%
\special{ip}%
%
\special{pn 8}%
\special{pa 2600 1740}%
\special{pa 2400 1940}%
\special{fp}%
\special{sh 1}%
\special{pa 2400 1940}%
\special{pa 2461 1907}%
\special{pa 2438 1902}%
\special{pa 2433 1879}%
\special{pa 2400 1940}%
\special{fp}%
%
\special{pn 8}%
\special{pa 4680 1740}%
\special{pa 4880 1940}%
\special{fp}%
\special{sh 1}%
\special{pa 4880 1940}%
\special{pa 4847 1879}%
\special{pa 4842 1902}%
\special{pa 4819 1907}%
\special{pa 4880 1940}%
\special{fp}%
\put(23.3000,-18.2000){\makebox(0,0)[lb]{$\alpha^*$}}%
\put(48.2000,-18.1000){\makebox(0,0)[lb]{$\alpha$}}%
\put(15.9000,-36.2000){\makebox(0,0)[lb]{$i_0$}}%
\put(34.2000,-52.5000){\makebox(0,0)[lb]{$i'$}}%
\put(51.9000,-36.2000){\makebox(0,0)[lb]{$i_{2m}$}}%
%
\special{pn 8}%
\special{pa 1820 3690}%
\special{pa 3300 5170}%
\special{fp}%
%
\special{pn 8}%
\special{pa 5120 3700}%
\special{pa 3670 5150}%
\special{fp}%
\put(50.2000,-44.5000){\makebox(0,0)[lb]{$i'=n-i+1$}}%
%
\special{pn 8}%
\special{pa 3290 4790}%
\special{pa 3470 4970}%
\special{fp}%
%
\special{pn 8}%
\special{pa 3470 4970}%
\special{pa 3670 4770}%
\special{fp}%
%
\special{pn 4}%
\special{pa 2750 2880}%
\special{pa 2320 2880}%
\special{fp}%
\special{pa 2860 2970}%
\special{pa 2230 2970}%
\special{fp}%
\special{pa 3610 3060}%
\special{pa 2140 3060}%
\special{fp}%
\special{pa 3620 3150}%
\special{pa 2050 3150}%
\special{fp}%
\special{pa 3620 3240}%
\special{pa 1960 3240}%
\special{fp}%
\special{pa 4870 3330}%
\special{pa 2040 3330}%
\special{fp}%
\special{pa 4960 3240}%
\special{pa 4460 3240}%
\special{fp}%
\special{pa 4890 3150}%
\special{pa 4540 3150}%
\special{fp}%
\special{pa 4800 3060}%
\special{pa 4630 3060}%
\special{fp}%
\special{pa 4790 3420}%
\special{pa 2130 3420}%
\special{fp}%
\special{pa 4860 3510}%
\special{pa 2090 3510}%
\special{fp}%
\special{pa 4840 3600}%
\special{pa 2100 3600}%
\special{fp}%
\special{pa 3250 3690}%
\special{pa 2190 3690}%
\special{fp}%
\special{pa 3250 3780}%
\special{pa 2280 3780}%
\special{fp}%
\special{pa 4570 3870}%
\special{pa 2370 3870}%
\special{fp}%
\special{pa 4660 3780}%
\special{pa 3760 3780}%
\special{fp}%
\special{pa 4750 3690}%
\special{pa 3760 3690}%
\special{fp}%
\special{pa 4480 3960}%
\special{pa 2460 3960}%
\special{fp}%
\special{pa 4390 4050}%
\special{pa 2550 4050}%
\special{fp}%
\special{pa 4300 4140}%
\special{pa 2640 4140}%
\special{fp}%
\special{pa 4210 4230}%
\special{pa 2730 4230}%
\special{fp}%
\special{pa 4120 4320}%
\special{pa 2820 4320}%
\special{fp}%
\special{pa 4030 4410}%
\special{pa 2910 4410}%
\special{fp}%
\special{pa 3940 4500}%
\special{pa 3000 4500}%
\special{fp}%
\special{pa 3850 4590}%
\special{pa 3090 4590}%
\special{fp}%
\special{pa 3760 4680}%
\special{pa 3180 4680}%
\special{fp}%
\special{pa 3660 4770}%
\special{pa 3270 4770}%
\special{fp}%
\special{pa 3570 4860}%
\special{pa 3360 4860}%
\special{fp}%
\special{pa 4420 3240}%
\special{pa 3790 3240}%
\special{fp}%
\special{pa 4330 3150}%
\special{pa 3790 3150}%
\special{fp}%
\special{pa 4230 3060}%
\special{pa 3800 3060}%
\special{fp}%
\special{pa 4140 2970}%
\special{pa 3890 2970}%
\special{fp}%
\special{pa 4050 2880}%
\special{pa 3970 2880}%
\special{fp}%
\special{pa 3520 2970}%
\special{pa 2880 2970}%
\special{fp}%
\special{pa 3430 2880}%
\special{pa 2970 2880}%
\special{fp}%
\special{pa 3340 2790}%
\special{pa 3060 2790}%
\special{fp}%
\special{pa 3250 2700}%
\special{pa 3150 2700}%
\special{fp}%
\special{pa 2660 2790}%
\special{pa 2410 2790}%
\special{fp}%
\special{pa 2570 2700}%
\special{pa 2490 2700}%
\special{fp}%
\end{picture}}%
\vspace{5pt}\]

Accordingly, $e_iI_w$ determines a polygon $\mathcal{P}$ whose vertices are $v_0,v_1,v_2,\dots,v_{2m-1},v_{2m}$ and $v$
corresponding to $i_0,i_1,i_2,\dots,i_{2m-1},i_{2m}$ and $n-i+1$. We note that $i_0\neq i_{2m}$.
Now we input $\mathcal{P}$ in $\R^2 $ by following correspondence:
\begin{itemize}
\item $v_0=(0,0)$.
\item $\alpha=(1,-1)$.
\item $\alpha^*=(-1,-1)$.
\end{itemize} 
Then we have 
\[\begin{array}{ccl}
v_{2t-1}&=&(i_{2t-1}-i_0,i_{2t-1}-i_0+2\sum_{s=1}^{t-1}i_{2s-1}-2\sum_{s=1}^{t-1}i_{2s}),\\\\
v_{2t}&=&(i_{2t}-i_0,-i_{2t}-i_0+2\sum_{s=1}^t i_{2s-1}-2\sum_{s=1}^{t-1}i_{2s}).\\
\end{array} \]
Let 
$v_{2m}=(i_{2t}-i_0,-i_{2t}-i_0+2\sum_{s=1}^m i_{2s-1}-2\sum_{s=1}^{m-1}i_{2s})=:(x,y)$.
Since $v$ is the intersection of $\{a(1,-1)\mid a\in \R \}$ and $\{b(1,1)+v_{2m}\mid b\in\R \}$, we conclude that  
\[(n-i-i_0+1,-(n-i-i_0+1))=v=(\frac{x-y}{2},\frac{-x+y}{2}).\]
Therefore $e_i I_w$ is uniquely determined by $\mathbf{i}=(i_0<i_1<\cdots<i_{2m})\in \Xi $. 
In particular, $\trigid \Lambda$ is parametrized by a subset of
 \[ \{\mathbf{i}\in \Xi\mid m_{\mathbf{i}}>0\}.\]
\end{proof} 
For  $\mathbf{i}\in \Xi$,  
we set a two-term objects $X_{\mathbf{i}}(\Lambda)=X_{\mathbf{i}}=[X_{\mathbf{i}}^{-1}\stackrel{d_{\mathbf{i}}}{\to} 
X_{\mathbf{i}}^0]$  as follows:\vspace{5pt}
\[
{\unitlength 0.1in%
\begin{picture}( 16.5000, 36.7000)( 22.7000,-44.6000)%
\put(24.0000,-14.0000){\makebox(0,0)[lb]{$P_{i_0}$}}%
\put(24.0000,-23.4000){\makebox(0,0)[lb]{$P_{i_2}$}}%
\put(39.2000,-18.8000){\makebox(0,0)[lb]{$P_{i_1}$}}%
\put(23.3000,-35.8000){\makebox(0,0)[lb]{$P_{i_{2m_{\mathbf{i}}-2}}$}}%
\put(24.0000,-45.4000){\makebox(0,0)[lb]{$P_{i_{2m_{\mathbf{i}}}}$}}%
\put(39.2000,-40.8000){\makebox(0,0)[lb]{$P_{i_{2m_{\mathbf{i}}-1}}$}}%
\put(39.0000,-31.6000){\makebox(0,0)[lb]{$P_{i_{2m_{\mathbf{i}}-3}}$}}%
%
\special{pn 4}%
\special{sh 1}%
\special{ar 3320 2540 8 8 0  6.28318530717959E+0000}%
\special{sh 1}%
\special{ar 3320 2740 8 8 0  6.28318530717959E+0000}%
\special{sh 1}%
\special{ar 3320 2940 8 8 0  6.28318530717959E+0000}%
\put(32.2000,-16.1000){\makebox(0,0)[lb]{$f_{i_0}^{i_1}$}}%
\put(32.2000,-21.4000){\makebox(0,0)[lb]{$f_{i_2}^{i_1}$}}%
\put(22.7000,-9.2000){\makebox(0,0)[lb]{$(-1\mathrm{th})$}}%
\put(38.5000,-9.2000){\makebox(0,0)[lb]{$(0\mathrm{th})$}}%
%
\special{pn 8}%
\special{pa 2700 1380}%
\special{pa 3080 1490}%
\special{fp}%
%
\special{pn 8}%
\special{pa 3600 1660}%
\special{pa 3860 1750}%
\special{fp}%
\special{sh 1}%
\special{pa 3860 1750}%
\special{pa 3804 1709}%
\special{pa 3810 1733}%
\special{pa 3790 1747}%
\special{pa 3860 1750}%
\special{fp}%
%
\special{pn 8}%
\special{pa 2720 2260}%
\special{pa 3100 2150}%
\special{fp}%
%
\special{pn 8}%
\special{pa 3600 1990}%
\special{pa 3860 1910}%
\special{fp}%
\special{sh 1}%
\special{pa 3860 1910}%
\special{pa 3790 1910}%
\special{pa 3809 1926}%
\special{pa 3802 1949}%
\special{pa 3860 1910}%
\special{fp}%
\special{pa 3790 1700}%
\special{pa 3790 1700}%
\special{fp}%
\put(31.2000,-38.5000){\makebox(0,0)[lb]{$f_{i_{2m_{\mathbf{i}}-2}}^{i_{2m_{\mathbf{i}}-1}}$}}%
%
\special{pn 8}%
\special{pa 2700 3580}%
\special{pa 3080 3690}%
\special{fp}%
%
\special{pn 8}%
\special{pa 3600 3860}%
\special{pa 3860 3950}%
\special{fp}%
\special{sh 1}%
\special{pa 3860 3950}%
\special{pa 3804 3909}%
\special{pa 3810 3933}%
\special{pa 3790 3947}%
\special{pa 3860 3950}%
\special{fp}%
%
\special{pn 8}%
\special{pa 2720 4460}%
\special{pa 3100 4350}%
\special{fp}%
%
\special{pn 8}%
\special{pa 3600 4190}%
\special{pa 3860 4110}%
\special{fp}%
\special{sh 1}%
\special{pa 3860 4110}%
\special{pa 3790 4110}%
\special{pa 3809 4126}%
\special{pa 3802 4149}%
\special{pa 3860 4110}%
\special{fp}%
\special{pa 3790 3900}%
\special{pa 3790 3900}%
\special{fp}%
\put(31.3000,-44.1000){\makebox(0,0)[lb]{$f_{i_{2m_{\mathbf{i}}}}^{i_{2m_{\mathbf{i}}-1}}$}}%
%
\special{pn 8}%
\special{pa 2720 3460}%
\special{pa 3100 3350}%
\special{fp}%
%
\special{pn 8}%
\special{pa 3600 3190}%
\special{pa 3860 3110}%
\special{fp}%
\special{sh 1}%
\special{pa 3860 3110}%
\special{pa 3790 3110}%
\special{pa 3809 3126}%
\special{pa 3802 3149}%
\special{pa 3860 3110}%
\special{fp}%
\special{pa 3790 2900}%
\special{pa 3790 2900}%
\special{fp}%
\put(31.3000,-34.2000){\makebox(0,0)[lb]{$f_{i_{2m_{\mathbf{i}}-2}}^{i_{2m_{\mathbf{i}}-3}}$}}%
\end{picture}}%
\vspace{5pt}\]
where we assume $P_0=P_{n+1}=0$.
\begin{lemma}
\label{indecomposability}
$X_{\mathbf{i}}$ is indecomposable.
\end{lemma}
\begin{proof}
It is sufficient to show that $\End_{\Kb(\proj \Lambda) }(X_{\mathbf{i}})$ is local.
Let $\varphi=(u,v)\in \End_{\Kb(\proj \Lambda) }(X_{\mathbf{i}})$, where
$u\in \End_{\Lambda}(\bigoplus_{t} P_{i_{2t}})$ and $v\in \End_{\Lambda}(\bigoplus_{t} P_{i_{2t+1}})$.
Denote by $u_t^s:P_{i_{2t}} \to P_{i_{2s}}$ and $v_t^s:P_{i_{2t+1}} \to P_{i_{2s+1}}$ given by
$u$ and $v$ respectively. If $t\neq s$, then $u_t^s$ and $v_t^s$ are in radical of $\mod \Lambda$.
Hence $u$ (resp. $v$) is an isomorphism if and only if $u_t^t$ (resp. $v_t^t$) is isomorphism
for any $t$. 

By Lemma\;\ref{pathcondition}\;(2), we can easily check  that if $u_t^t$ (resp. $v_t^t$) is an isomorphism, 
then $v_t^t $ and $v_{t-1}^{t-1}$ (resp. $u_t^t $ and $u_{t+1}^{t+1}$) are isomorphisms.
In this case, we have that $\varphi$ is an isomorphism. Therefore,
 $\varphi$ is not an isomorphism only if $u_t^s$ and $v_t^s$ are in radical of $\mod \Lambda$
for any $t,s$. Conversely, if $u_t^s$ and $v_t^s$ are in radical of $\mod \Lambda$
for any $t,s$, then $\varphi$ is not an isomorphism. In particular, the set of non-isomorphisms of
 $\End_{\Kb(\proj \Lambda) }(X_{\mathbf{i}})$ form an ideal. This gives the assertion.
\end{proof}

\begin{lemma} 
\label{rigidcondition} Let $\mathbf{i}, 
\mathbf{j}\in \Xi$. 
For a pair $(t,s)$ such that $0<i_{2t}<j_{2s+1}<n+1$,
 we define two sequences $\mathbf{t}^{+}(t,s):=(t=t_{0}\leq t_{1}\leq t_{2} \leq \cdots)$ and 
$\mathbf{s}^{+}(t,s):=(s=s_{0}\leq s_{1}\leq s_{2}\cdots)$
by following rule$:$
\begin{itemize}
\item[{\rm (i)}] $t_r:=\left\{\begin{array}{cl}
\mathrm{max}\{t\leq m_\mathbf{i}+1\mid i_{2t-2}<j_{2s_{r-1}+1}\} & \mathrm{if\ } s_{r-1}\leq m_{\mathbf{j}}-1\\
m_{\mathbf{i}}+1 & \mathrm{if\ } s_{r-1}\geq m_{\mathbf{j}}, t_{r-1}\leq m_{\mathbf{i}}\\
m_{\mathbf{i}}+2 & \mathrm{if\ } t_{r-1}\geq m_{\mathbf{i}}+1
\end{array}\right..$

\item[{\rm (ii)}] $s_r:=\left\{\begin{array}{cl}
\mathrm{max}\{s\leq m_\mathbf{j}\mid j_{2s-1}<i_{2t_r}\} & \mathrm{if\ } t_r\leq m_{\mathbf{i}}\\
m_{\mathbf{j}} & \mathrm{if\ } s_{r-1}< m_{\mathbf{j}}, t_r= m_{\mathbf{i}}+1\\
m_{\mathbf{j}}+1 & \mathrm{if\ }s_{r-1}\geq m_{\mathbf{j}}.
\end{array}\right.
$

\end{itemize}
Also we define two sequences $\mathbf{t}^{-}(t,s):=(t=t_{0}\geq t_{-1}\geq t_{-2}\geq \cdots)$ and 
$\mathbf{s}^{-}(t,s):=(s=s_{0}\geq s_{-1}\geq s_{-2}\geq \cdots)$
by following rule$:$
\begin{itemize}
\item[{\rm (iii)}] 
$s_r:=\left\{\begin{array}{cl}
\mathrm{min}\{s\geq -1 \mid j_{2s+3}>i_{2t_{r+1}}\} & \mathrm{if\ } t_{r+1}\geq 0\\
-1& \mathrm{if\ } t_{r+1}\leq -1, s_{r+1}\geq 0 \\
-2 & \mathrm{if\ } s_{r+1}\leq -1.\\
\end{array}\right.$

\item[{\rm (iv)}] 
$t_r:=\left\{\begin{array}{cl}
\mathrm{min}\{t\geq -1  \mid i_{2t+2}>j_{2s_r+1}\} & \mathrm{if\ } s_r\geq 0 \\
-1 & \mathrm{if\ }s_r=-1, t_{r+1}\geq 0\\
-2 & \mathrm{if\ } t_{r+1}\leq -1\\
\end{array}\right.
$

\end{itemize}
If $\Hom_{\Kb(\proj \Lambda)}(X_{\mathbf{i}},X_{\mathbf{j}}[1])=0$, then one of the following holds.
\begin{itemize}
\item[(1)] We have {\rm (a),\ (b),\ (c)} and {\rm (d)}.
\begin{enumerate}[{\rm (a)}]
\item $i_{2t_r}<i_{2t_{r}+1}\leq j_{2s_{r}+1}$ for any $r\geq 0$ such that $t_r\leq m_{\mathbf{i}}$.
\item $i_{2t_{r}-2}<i_{2t_{r}-1}\leq j_{2s_{r-1}+1}$ for any $r\geq 1$ such that $t_r\leq m_{\mathbf{i}}+1$.
\item $j_{2s_{r-1}+1}<j_{2s_{r-1}+2}\leq i_{2t_r}$ for any $r\geq 1$ such that $s_r\leq m_{\mathbf{j}}$.
\item $j_{2s_{r}-1}<j_{2s_{r}}\leq i_{2t_r}$ for any $r\geq 1$ such that $s_r\leq m_{\mathbf{j}}$.

\end{enumerate} 
Where we put $i_{2m_{\mathbf{i}}+1}=j_{2m_{\mathbf{j}}+1}=n+2$ and $i_{2m_{\mathbf{i}}+1}=j_{2m_{\mathbf{j}}+1}=n+3$. 
\item[(2)] We have {\rm (a),\ (b),\ (c)} and {\rm (d)}.
\begin{enumerate}[{\rm (a')}]
\item $j_{2s_{r}+1}>j_{2s_{r}}\geq i_{2t_{r}}$ for any $r\leq 0$ such that $s_{r}\geq 0$.
\item $j_{2s_{r}+3}>j_{2s_{r}+2}\geq i_{2t_{r+1}}$ for any $r\leq -1$ such that $s_{r}\geq -1$.
\item $i_{2t_{r+1}}>i_{2t_{r+1}-1}\geq j_{2s_{r}+1}$ for any $r\leq -1$ such that $t_{r}\geq -1$.
\item $i_{2t_{r}+2}>i_{2t_{r}+1}\geq j_{2s_{r}+1}$ for any $r\leq -1$ such that $t_{r}\geq -1$.
\end{enumerate}
Where we put $i_{-1}=j_{-1}=-1$ and $i_{-2}=j_{-2}=-2$
\end{itemize}
\end{lemma}
\begin{proof}
We note that if $t_r\leq m_{\mathbf{i}}$ (resp. $s_r\leq m_{\mathbf{j}}$) for $r\geq 0$, then $t_r<t_{r+1}$ (resp. $s_r<s_{r+1}$).
and if $t_r\geq -1$ (resp. $s_r\geq -1$) for $r\leq 0$, then $t_r>t_{r-1}$ (resp. $s_r>s_{r-1}$).
We also note that $\Hom_{\Kb(\proj \Lambda)}(X_{\mathbf{i}},X_{\mathbf{j}}[1])=0$ implies that $i_{2p}\neq j_{2q+1}$ for any $p,q$.
We first show the assertion in the case that $t_1=t_0+1$. 
Let \[\begin{array}{l}
 \begin{array}{lrl}
 \xi^{+}
 &:=& \{(i_{2t_0+1},j_{2s_0+1})\mid i_{2t_0+1}>j_{2s_0+1}\}(=:\xi^+_0)\\
& \sqcup &\{(i_{2t_r-1},j_{2s_{r-1}+1})\mid r\geq 1, i_{2t_r-1}>j_{2s_{r-1}+1}, t_r\leq m_{\mathbf{i}}\}(=:\xi^+_1)\\
&\sqcup& \{(i_{2t_r},j_{2s_{r-1}+2})\mid r\geq 1, i_{2t_r}<j_{2s_{r-1}+2}, s_r\leq m_{\mathbf{j}}\}(=:\xi^+_2)\\
 &\sqcup & \{(i_{2t_r}, j_{2s_r})\mid r\geq 1, i_{2t_r}<j_{2s_r}, s_r\leq m_{\mathbf{j}}\}(=:\xi^+_3)\\
&\sqcup & \{(i_{2t_r+1},j_{2s_r+1})\mid r\geq 1, i_{2t_r+1}>j_{2s_r+1}, t_r\leq m_{\mathbf{i}}\}(=:\xi^+_4).
\end{array}\\\\
\begin{array}{lrl}
\xi^{-}
 &:=&\{(i_{2t_0},j_{2s_0})\mid i_{2t_0}>j_{2s_0}\}(=:\xi^-_0) \\
 &\sqcup&\{(i_{2t_{r+1}},j_{2s_r+2})\mid r\leq -1, i_{2t_{r+1}}>j_{2s_r+2}, s_r\geq -1\}(=:\xi^-_1)\\
&\sqcup & \{(i_{2t_{r+1}-1},j_{2s_r+1})\mid r\leq -1, i_{2t_{r+1}-1}<j_{2s_r+1}, t_r\geq -1\}(=:\xi^-_2)\\
 &\sqcup & \{(i_{2t_r+1},j_{2s_r+1})\mid r\leq -1, i_{2t_r+1}<j_{2s_r+1}, t_r \geq -1\}(=:\xi^-_3)\\
&\sqcup & \{(i_{2t_r},j_{2s_r})\mid r\leq -1, i_{2t_r}>j_{2s_r}, s_r\geq 0\}(=:\xi^-_4).
\end{array}
\end{array} \]
It is sufficient to show that either $\xi^{+}=\emptyset$ or $\xi^{-}=\emptyset$ holds.
We assume that $\xi^{+}\neq \emptyset$.
We consider lexicographical order $\preceq$ on $\Z^2$ 
(i.e. $(x,y)\preceq (x',y')$ if either (i) $x<x',$ or (ii) $x=x',\;y\leq y'$ hold)
 and take a minimum element $(x,y)\in \xi^{+}$. 
\begin{claim}
\label{rigidcdadclaim}
Suppose that $\xi^-\neq\emptyset$. Let $(x',y')$ be a maximum element of $\xi^-$.
\begin{enumerate}[{\rm (1)}]
\item Assume that $(x',y')=(i_{2t_{r'+1}},j_{2s_{r'}+2})\in \xi_1^-$. If $r'=-1$, then we have 
\[i_{2t_0+1}\leq j_{2s_{-1}+3}.\] If $r'<-1$, then the following holds. 
\[ \left\{\begin{array}{rcccccccccccccclll}
  &           & &            & &          & &j_{2s_{r'}+3}&=&j_{2s_{r'+1}+1}&=&i_{2t_{r'+1}+1}&<\cdots \\
< & i_{2t_p+2}&=&i_{2t_{p+1}}&=&j_{2s_p+2}&<&j_{2s_p+3}&=&j_{2s_{p+1}+1}&=&i_{2t_{p+1}+1}&<\cdots\\
< & i_{2t_{-2}+2}&=&i_{2t_{-1}}&=&j_{2s_{-2}+2}&<&j_{2s_{-2}+3}&=&j_{2s_{-1}+1}&=&i_{2t_{-1}+1}\\
< & i_{2t_{-1}+2}&=& i_{2t_0}&=& j_{2s_{-1}+2}&<&j_{2s_{-1}+3}&\geq &i_{2t_0+1}
\end{array}\right..\] 
\item Assume that $(x',y')=(i_{2t_{r'+1}-1},j_{2s_{r'}+1})\in \xi_2^-$,. If $r'=-1$, then we have
\[i_{2t_0}=j_{2s_{-1}+2}<j_{2s_{-1}+3}\geq i_{2t_0+1}.\] If $r'<-1$, then the following holds.
\[ \left\{\begin{array}{rcccccccccclll}
& & &i_{2t_{r'+1}}&=&j_{2s_{r'}+2}&<&j_{2s_{r'}+3}&=&j_{2s_{r'+1}+1}&=&i_{2t_{r'+1}+1}&<\cdots\\
< & i_{2t_p+2}&=&i_{2t_{p+1}}&=&j_{2s_p+2}&<&j_{2s_p+3}&=&j_{2s_{p+1}+1}&=&i_{2t_{p+1}+1}&<\cdots\\
< & i_{2t_{-2}+2}&=&i_{2t_{-1}}&=&j_{2s_{-2}+2}&<&j_{2s_{-2}+3}&=&j_{2s_{-1}+1}&=&i_{2t_{-1}+1}\\
< & i_{2t_{-1}+2}&=& i_{2t_0}&=& j_{2s_{-1}+2}&<&j_{2s_{-1}+3}&\geq &i_{2t_0+1}
\end{array}\right..\] 
\item Assume that $(x',y')=(i_{2t_{r'}+1},j_{2s_{r'}+1})\in \xi_3^-$. Then we obtain $i_{2t_{r'}+2}>j_{2s_{r'}+1}>0$.
 If $r'=-1$, then we have 
\[i_{2t_{-1}+2}=i_{2t_0}=j_{2s_{-1}+2}<j_{2s_{-1}+3}\geq i_{2t_0+1}.\] If $r'<-1$, then the following holds. 
\[ \left\{\begin{array}{rcccccccccclll}
&i_{2t_{r'}+2}&=&i_{2t_{r'+1}}&=&j_{2s_{r'}+2}&<&j_{2s_{r'}+3}&=&j_{2s_{r'+1}+1}&=&i_{2t_{r'+1}+1}&<\cdots\\
< & i_{2t_p+2}&=&i_{2t_{p+1}}&=&j_{2s_p+2}&<&j_{2s_p+3}&=&j_{2s_{p+1}+1}&=&i_{2t_{p+1}+1}&<\cdots\\
< & i_{2t_{-2}+2}&=&i_{2t_{-1}}&=&j_{2s_{-2}+2}&<&j_{2s_{-2}+3}&=&j_{2s_{-1}+1}&=&i_{2t_{-1}+1}\\
< & i_{2t_{-1}+2}&=& i_{2t_0}&=& j_{2s_{-1}+2}&<&j_{2s_{-1}+3}&\geq &i_{2t_0+1}
\end{array}\right..\] 
\item Assume that $(x',y')=(i_{2t_{r'}},j_{2s_{r'}})\in \xi_4^-$. Then we obtain 
 $0<i_{2t_{r'}}<i_{2t_{r'}+1}=j_{2s_{r'}+1}$.
 If $r'=-1$, then we have 
\[i_{2t_{-1}+2}=i_{2t_0}=j_{2s_{-1}+2}<j_{2s_{-1}+3}\geq i_{2t_0+1}.\] If $r'<-1$, then the following holds.
\[ \left\{\begin{array}{rcccccccccclll}
&i_{2t_{r'}+2}&=&i_{2t_{r'+1}}&=&j_{2s_{r'}+2}&<&j_{2s_{r'}+3}&=&j_{2s_{r'+1}+1}&=&i_{2t_{r'+1}+1}&<\cdots\\
< & i_{2t_p+2}&=&i_{2t_{p+1}}&=&j_{2s_p+2}&<&j_{2s_p+3}&=&j_{2s_{p+1}+1}&=&i_{2t_{p+1}+1}&<\cdots\\
< & i_{2t_{-2}+2}&=&i_{2t_{-1}}&=&j_{2s_{-2}+2}&<&j_{2s_{-2}+3}&=&j_{2s_{-1}+1}&=&i_{2t_{-1}+1}\\
< & i_{2t_{-1}+2}&=& i_{2t_0}&=& j_{2s_{-1}+2}&<&j_{2s_{-1}+3}&\geq &i_{2t_0+1}
\end{array}\right..\] 
\end{enumerate}
\end{claim}
\begin{pfclaim}
We treat the case that $(x',y')=(i_{2t_{r'}},j_{2s_{r'}})\in \xi^-_4$. By definition we obtain that
 $i_{2t}\geq i_{2t_{r'}+2}>j_{2s_{r'}+1}>i_{2t_{r'}}>j_{2s_{r'}}\geq 0.$ Accordingly, $P_{i_{2t_{r'}}},P_{j_{2s_{r'}+1}}\neq 0$. 
 Let $f_{i_{2t_{r'}}}^{j_{2s_{r'}+1}}\in \Hom_{\Lambda}(P_{i_{2t_{r'}}},P_{j_{2s_{r'}+1}})$ 
 given by $w_{i_{2t_{r'}}}^{j_{2s_{r'}+1}}.$ 
We consider $\varphi=(\varphi_{i_{2t}}^{j_{2s+1}}:P_{i_{2t}}^{j_{2s+1}}\to P_{j_{2s+1}})\in \Hom_{\Lambda}(X_{\mathbf{i}}^{-1},X_{\mathbf{j}}^{0})
=\Hom_{\Lambda}(\oplus P_{i_{2t}},\oplus P_{j_{2s+1}})$
such that \[\varphi_{i_{2t}}^{j_{2s+1}}=\left\{\begin{array}{cl}
f_{i_{2t_{r'}}}^{j_{2s_{r'}+1}} & (t,s)=(t_{r'},s_{r'})\\\\
0 & \mathrm{otherwise}.
\end{array}\right.\]
We may regard $\varphi$ as a morphism in $\Hom_{\Kb(\proj \Lambda)}(X_{\mathbf{i}},X_{\mathbf{j}}[1])$ by natural way.
Note that \[\Hom_{\Kb(\proj \Lambda)}(X_{\mathbf{i}},X_{\mathbf{j}}[1])=0.\]
 Therefore there are $h=(h_{i_{2t+1}}^{j_{2s+1}})\in \Hom_{\Lambda}(X_{\mathbf{i}}^0,Y_J^0)$ and
 $h'=(h_{i_{2t}}^{j_{2s}})\in \Hom_{\Lambda}(X_{\mathbf{i}}^{-1},Y_J^{-1})$ such that
 \[\varphi=h\circ d_{X_{\mathbf{i}}}+d_{X_{\mathbf{j}}}\circ h'.\]
 In particular, one sees the following equation. (See figure\;\ref{fig:one}.)
 \[f_{i_{2t_{r'}}}^{j_{2s_{r'}+1}}=h_{i_{2t_{r'}+1}}^{j_{2s_{r'}+1}}\circ f_{i_{2t_{r'}}}^{i_{2t_{r'}+1}}+
 h_{i_{2t_{r'}-1}}^{j_{2s_{r'}+1}}\circ f_{i_{2t_{r'}}}^{i_{2t_{r'}-1}}+
 f_{j_{2s_{r'}}}^{j_{2s_{r'}+1}}\circ h_{i_{2t_{r'}}}^{j_{2s_{r'}}}+
 f_{j_{2s_{r'}+2}}^{j_{2s_{r'}+1}}\circ h_{i_{2t_{r'}}}^{j_{2s_{r'}+2}}.\]
 \vspace{5pt}
 \begin{figure}[h]
\begin{center}
{\unitlength 0.1in%
\begin{picture}( 42.0000, 31.3200)(  3.8000,-35.6200)%
\put(24.0000,-10.0000){\makebox(0,0)[lb]{$P_{i_{2t_{r'}}}$}}%
\put(24.0000,-19.4000){\makebox(0,0)[lb]{$P_{i_{2t_{r'}+2}}$}}%
%
\special{pn 8}%
\special{pa 2990 1810}%
\special{pa 3870 1440}%
\special{fp}%
\special{sh 1}%
\special{pa 3870 1440}%
\special{pa 3801 1447}%
\special{pa 3821 1461}%
\special{pa 3816 1484}%
\special{pa 3870 1440}%
\special{fp}%
\put(39.2000,-14.8000){\makebox(0,0)[lb]{$P_{i_{2t_{r'}+1}}$}}%
\put(9.3000,-27.1000){\makebox(0,0)[lb]{$P_{j_{2s_{r'}}}$}}%
\put(9.3000,-36.5000){\makebox(0,0)[lb]{$P_{j_{2s_{r'}+2}}$}}%
\put(24.5000,-32.2000){\makebox(0,0)[lb]{$P_{j_{2s_{r'}+1}}$}}%
\put(39.0000,-5.6000){\makebox(0,0)[lb]{$P_{i_{2t_{r'}-1}}$}}%
\put(31.3000,-24.5000){\makebox(0,0)[lb]{$h_{i_{2t_{r'}+1}}^{j_{2s_{r'}+1}}$}}%
\put(44.0000,-14.4000){\makebox(0,0)[lb]{$h_{i_{2t_{r'}-1}}^{j_{2s_{r'}+1}}$}}%
\put(15.6000,-19.2000){\makebox(0,0)[lb]{$h_{i_{2t_{r'}}}^{j_{2s_{r'}}}$}}%
\put(21.6000,-23.8000){\makebox(0,0)[lb]{$f_{i_{2t_{r'}}}^{j_{2s_{r'}+1}}$}}%
%
\special{pn 8}%
\special{pn 8}%
\special{pa 2320 990}%
\special{pa 2291 1001}%
\special{pa 2261 1012}%
\special{pa 2260 1012}%
\special{fp}%
\special{pa 2198 1036}%
\special{pa 2174 1045}%
\special{pa 2144 1057}%
\special{pa 2139 1059}%
\special{fp}%
\special{pa 2077 1082}%
\special{pa 2056 1090}%
\special{pa 2027 1101}%
\special{pa 2017 1105}%
\special{fp}%
\special{pa 1956 1129}%
\special{pa 1938 1135}%
\special{pa 1909 1147}%
\special{pa 1896 1152}%
\special{fp}%
\special{pa 1835 1176}%
\special{pa 1820 1182}%
\special{pa 1791 1193}%
\special{pa 1775 1199}%
\special{fp}%
\special{pa 1714 1224}%
\special{pa 1702 1229}%
\special{pa 1654 1248}%
\special{fp}%
\special{pa 1594 1273}%
\special{pa 1552 1290}%
\special{pa 1534 1298}%
\special{fp}%
\special{pa 1474 1323}%
\special{pa 1462 1328}%
\special{pa 1431 1341}%
\special{pa 1415 1349}%
\special{fp}%
\special{pa 1355 1375}%
\special{pa 1340 1381}%
\special{pa 1310 1395}%
\special{pa 1296 1401}%
\special{fp}%
\special{pa 1236 1428}%
\special{pa 1218 1436}%
\special{pa 1187 1451}%
\special{pa 1178 1455}%
\special{fp}%
\special{pa 1119 1484}%
\special{pa 1097 1495}%
\special{pa 1062 1514}%
\special{fp}%
\special{pa 1005 1546}%
\special{pa 982 1560}%
\special{pa 951 1581}%
\special{fp}%
\special{pa 897 1618}%
\special{pa 876 1633}%
\special{pa 846 1658}%
\special{fp}%
\special{pa 797 1702}%
\special{pa 783 1716}%
\special{pa 762 1738}%
\special{pa 753 1748}%
\special{fp}%
\special{pa 713 1800}%
\special{pa 706 1810}%
\special{pa 679 1854}%
\special{fp}%
\special{pa 649 1913}%
\special{pa 647 1917}%
\special{pa 635 1945}%
\special{pa 624 1972}%
\special{fp}%
\special{pa 603 2034}%
\special{pa 585 2096}%
\special{fp}%
\special{pa 571 2160}%
\special{pa 564 2192}%
\special{pa 558 2223}%
\special{fp}%
\special{pa 547 2288}%
\special{pa 547 2290}%
\special{pa 541 2324}%
\special{pa 538 2351}%
\special{fp}%
\special{pa 528 2416}%
\special{pa 527 2424}%
\special{pa 523 2458}%
\special{pa 520 2480}%
\special{fp}%
\put(3.8000,-28.1000){\makebox(0,0)[lb]{$h_{i_{2t_{r'}}}^{j_{2s_{r'}+2}}$}}%
%
\special{pn 8}%
\special{pn 8}%
\special{pa 500 2890}%
\special{pa 502 2953}%
\special{fp}%
\special{pa 506 3018}%
\special{pa 506 3022}%
\special{pa 509 3054}%
\special{pa 513 3081}%
\special{fp}%
\special{pa 526 3145}%
\special{pa 526 3147}%
\special{pa 535 3177}%
\special{pa 546 3205}%
\special{fp}%
\special{pa 576 3263}%
\special{pa 590 3287}%
\special{pa 608 3312}%
\special{pa 611 3315}%
\special{fp}%
\special{pa 652 3365}%
\special{pa 670 3385}%
\special{pa 693 3408}%
\special{pa 696 3411}%
\special{fp}%
\special{pa 743 3455}%
\special{pa 767 3477}%
\special{pa 791 3498}%
\special{fp}%
\special{pa 840 3540}%
\special{pa 840 3540}%
\special{fp}%
%
\special{pn 8}%
\special{pa 818 3521}%
\special{pa 840 3540}%
\special{da 0.070}%
\special{sh 1}%
\special{pa 840 3540}%
\special{pa 803 3481}%
\special{pa 800 3505}%
\special{pa 776 3512}%
\special{pa 840 3540}%
\special{fp}%
%
\special{pn 8}%
\special{pn 8}%
\special{pa 2780 3060}%
\special{pa 2809 3049}%
\special{pa 2839 3038}%
\special{pa 2840 3038}%
\special{fp}%
\special{pa 2901 3015}%
\special{pa 2927 3005}%
\special{pa 2960 2992}%
\special{fp}%
\special{pa 3021 2969}%
\special{pa 3073 2949}%
\special{pa 3081 2946}%
\special{fp}%
\special{pa 3141 2923}%
\special{pa 3162 2915}%
\special{pa 3191 2904}%
\special{pa 3201 2900}%
\special{fp}%
\special{pa 3262 2876}%
\special{pa 3280 2869}%
\special{pa 3309 2857}%
\special{pa 3321 2852}%
\special{fp}%
\special{pa 3382 2829}%
\special{pa 3399 2822}%
\special{pa 3428 2809}%
\special{pa 3440 2804}%
\special{fp}%
\special{pa 3501 2780}%
\special{pa 3518 2773}%
\special{pa 3560 2755}%
\special{fp}%
\special{pa 3620 2730}%
\special{pa 3639 2722}%
\special{pa 3678 2705}%
\special{fp}%
\special{pa 3738 2678}%
\special{pa 3760 2669}%
\special{pa 3790 2655}%
\special{pa 3796 2652}%
\special{fp}%
\special{pa 3856 2626}%
\special{pa 3882 2613}%
\special{pa 3913 2599}%
\special{pa 3913 2599}%
\special{fp}%
\special{pa 3972 2570}%
\special{pa 4003 2554}%
\special{pa 4028 2540}%
\special{fp}%
\special{pa 4085 2508}%
\special{pa 4089 2506}%
\special{pa 4117 2489}%
\special{pa 4139 2475}%
\special{fp}%
\special{pa 4193 2438}%
\special{pa 4198 2435}%
\special{pa 4223 2416}%
\special{pa 4244 2399}%
\special{fp}%
\special{pa 4293 2357}%
\special{pa 4295 2355}%
\special{pa 4317 2334}%
\special{pa 4338 2311}%
\special{fp}%
\special{pa 4379 2261}%
\special{pa 4411 2215}%
\special{pa 4415 2208}%
\special{fp}%
\special{pa 4446 2151}%
\special{pa 4454 2135}%
\special{pa 4466 2107}%
\special{pa 4471 2093}%
\special{fp}%
\special{pa 4493 2031}%
\special{pa 4506 1989}%
\special{pa 4511 1970}%
\special{fp}%
\special{pa 4526 1907}%
\special{pa 4529 1895}%
\special{pa 4538 1844}%
\special{fp}%
\special{pa 4549 1780}%
\special{pa 4551 1765}%
\special{pa 4555 1732}%
\special{pa 4557 1717}%
\special{fp}%
\special{pa 4564 1652}%
\special{pa 4570 1588}%
\special{fp}%
\special{pa 4575 1524}%
\special{pa 4577 1493}%
\special{pa 4580 1460}%
\special{fp}%
%
\special{pn 8}%
\special{pa 2809 3049}%
\special{pa 2780 3060}%
\special{da 0.070}%
\special{sh 1}%
\special{pa 2780 3060}%
\special{pa 2849 3055}%
\special{pa 2830 3041}%
\special{pa 2835 3018}%
\special{pa 2780 3060}%
\special{fp}%
%
\special{pn 8}%
\special{pn 8}%
\special{pa 4100 600}%
\special{pa 4124 622}%
\special{pa 4148 643}%
\special{pa 4150 645}%
\special{fp}%
\special{pa 4200 692}%
\special{pa 4218 709}%
\special{pa 4240 732}%
\special{pa 4247 740}%
\special{fp}%
\special{pa 4293 791}%
\special{pa 4303 804}%
\special{pa 4322 829}%
\special{pa 4334 845}%
\special{fp}%
\special{pa 4371 903}%
\special{pa 4392 936}%
\special{pa 4406 960}%
\special{fp}%
\special{pa 4439 1020}%
\special{pa 4439 1021}%
\special{pa 4454 1049}%
\special{pa 4469 1078}%
\special{pa 4470 1080}%
\special{fp}%
%
\special{pn 8}%
\special{pa 2410 1050}%
\special{pa 1890 1590}%
\special{da 0.070}%
%
\special{pn 8}%
\special{pa 1570 1960}%
\special{pa 1070 2480}%
\special{da 0.070}%
\special{sh 1}%
\special{pa 1070 2480}%
\special{pa 1131 2446}%
\special{pa 1107 2442}%
\special{pa 1102 2418}%
\special{pa 1070 2480}%
\special{fp}%
%
\special{pn 8}%
\special{pa 4020 1560}%
\special{pa 3500 2100}%
\special{da 0.070}%
%
\special{pn 8}%
\special{pa 3180 2510}%
\special{pa 2680 3030}%
\special{da 0.070}%
\special{sh 1}%
\special{pa 2680 3030}%
\special{pa 2741 2996}%
\special{pa 2717 2992}%
\special{pa 2712 2968}%
\special{pa 2680 3030}%
\special{fp}%
%
\special{pn 8}%
\special{pa 2530 1060}%
\special{pa 2518 1090}%
\special{pa 2505 1119}%
\special{pa 2481 1179}%
\special{pa 2469 1208}%
\special{pa 2457 1238}%
\special{pa 2424 1328}%
\special{pa 2414 1358}%
\special{pa 2404 1389}%
\special{pa 2394 1419}%
\special{pa 2385 1450}%
\special{pa 2376 1480}%
\special{pa 2360 1542}%
\special{pa 2339 1635}%
\special{pa 2333 1667}%
\special{pa 2326 1698}%
\special{pa 2321 1730}%
\special{pa 2315 1761}%
\special{pa 2309 1793}%
\special{pa 2304 1825}%
\special{pa 2299 1856}%
\special{pa 2284 1952}%
\special{pa 2280 1983}%
\special{pa 2270 2047}%
\special{pa 2270 2050}%
\special{fp}%
%
\special{pn 8}%
\special{pa 2260 2460}%
\special{pa 2266 2556}%
\special{pa 2269 2588}%
\special{pa 2273 2620}%
\special{pa 2278 2651}%
\special{pa 2285 2682}%
\special{pa 2293 2713}%
\special{pa 2302 2743}%
\special{pa 2312 2773}%
\special{pa 2323 2803}%
\special{pa 2336 2833}%
\special{pa 2362 2891}%
\special{pa 2377 2920}%
\special{pa 2391 2949}%
\special{pa 2407 2978}%
\special{pa 2422 3007}%
\special{pa 2437 3035}%
\special{pa 2440 3040}%
\special{fp}%
%
\special{pn 8}%
\special{pa 2437 3035}%
\special{pa 2440 3040}%
\special{fp}%
\special{sh 1}%
\special{pa 2440 3040}%
\special{pa 2423 2973}%
\special{pa 2413 2994}%
\special{pa 2389 2993}%
\special{pa 2440 3040}%
\special{fp}%
%
\special{pn 8}%
\special{pa 2740 870}%
\special{pa 3040 760}%
\special{fp}%
\put(30.7000,-8.1000){\makebox(0,0)[lb]{$f_{i_{2t_{r'}}}^{i_{2t_{r'}-1}}$}}%
%
\special{pn 8}%
\special{pa 3600 590}%
\special{pa 3820 510}%
\special{fp}%
\special{sh 1}%
\special{pa 3820 510}%
\special{pa 3751 514}%
\special{pa 3770 528}%
\special{pa 3764 552}%
\special{pa 3820 510}%
\special{fp}%
%
\special{pn 8}%
\special{pa 1320 3562}%
\special{pa 1620 3452}%
\special{fp}%
\put(16.4000,-35.0200){\makebox(0,0)[lb]{$f_{j_{2s_{r'}+2}}^{j_{2s_{r'}+1}}$}}%
%
\special{pn 8}%
\special{pa 2180 3282}%
\special{pa 2400 3202}%
\special{fp}%
\special{sh 1}%
\special{pa 2400 3202}%
\special{pa 2331 3206}%
\special{pa 2350 3220}%
\special{pa 2344 3244}%
\special{pa 2400 3202}%
\special{fp}%
%
\special{pn 8}%
\special{pa 2730 970}%
\special{pa 3020 1070}%
\special{fp}%
\put(30.9000,-12.5000){\makebox(0,0)[lb]{$f_{i_{2t_{r'}}}^{i_{2t_{r'}+1}}$}}%
%
\special{pn 8}%
\special{pa 3590 1250}%
\special{pa 3820 1340}%
\special{fp}%
\special{sh 1}%
\special{pa 3820 1340}%
\special{pa 3765 1297}%
\special{pa 3770 1321}%
\special{pa 3751 1334}%
\special{pa 3820 1340}%
\special{fp}%
%
\special{pn 8}%
\special{pa 1320 2740}%
\special{pa 1610 2840}%
\special{fp}%
\put(16.8000,-30.2000){\makebox(0,0)[lb]{$f_{j_{2s_{r'}}}^{j_{2s_{r'}+1}}$}}%
%
\special{pn 8}%
\special{pa 2180 3020}%
\special{pa 2410 3110}%
\special{fp}%
\special{sh 1}%
\special{pa 2410 3110}%
\special{pa 2355 3067}%
\special{pa 2360 3091}%
\special{pa 2341 3104}%
\special{pa 2410 3110}%
\special{fp}%
\end{picture}}%
  \end{center}
\caption{}
\label{fig:one}
\end{figure} 
 Note that $i_{2t_{r'}-1},j_{2s_{r'}}<i_{2t_{r'}}<j_{2s_{r'}+1}<j_{2s_{r'}+2}$.
  Lemma\;\ref{pathcondition} implies that
 \[i_{2t_{r'}+1}\leq j_{2s_{r'}+1}.\]
 Since $(i_{2t_{r'}},j_{2s_{r'}})$ is maximum in $\xi^-$ and 
 $(i_{2t_{r'}},j_{2s_{r'}})\prec (i_{2t_{r'}+1},j_{2s_{r'}+1})$,
   we conclude that
  \[j_{2s_{r'}+1}\leq i_{2t_{r'}+1}.\]
In particular, we get  
\[i_{2t_{r'}+1}= j_{2s_{r'}+1}.\] Lemma\;\ref{pathcondition} also says that
\[h_{i_{2t_{r'}+1}}^{j_{2s_{r'}+1}}= 1+l,\]
for some $l\in  \Rad\End_{\Lambda}(P_{j_{2s_{r'}+1}}).$  

Next we consider $\varphi_{i_{2t_{r'}+2}}^{j_{2s_{r'}+1}}=0$. 
We obtain the following equation. 
\[0=h_{i_{2t_{r'}+1}}^{j_{2s_{r'}+1}}\circ f_{i_{2t_{r'}+2}}^{i_{2t_{r'}+1}}+
 h_{i_{2t_{r'}+3}}^{j_{2s_{r'}+1}}\circ f_{i_{2t_{r'}+2}}^{i_{2t_{r'}+3}}+
 f_{j_{2s_{r'}}}^{j_{2s_{r'}+1}}\circ h_{i_{2t_{r'}+2}}^{j_{2s_{r'}}}+
 f_{j_{2s_{r'}+2}}^{j_{2s_{r'}+1}}\circ h_{i_{2t_{r'}+2}}^{j_{2s_{r'}+2}}.\]
 \vspace{5pt}
 By Lemma\;\ref{pathcondition}, we have that
 \[\begin{array}{ccl}
  h_{i_{2t_{r'}+3}}^{j_{2s_{r'}+1}}\circ f_{i_{2t_{r'}+2}}^{i_{2t_{r'}+3}}+
 f_{j_{2s_{r'}}}^{j_{2s_{r'}+1}}\circ h_{i_{2t_{r'}+2}}^{j_{2s_{r'}}}+
 f_{j_{2s_{r'}+2}}^{j_{2s_{r'}+1}}\circ h_{i_{2t_{r'}+2}}^{j_{2s_{r'}+2}}
 &=&(-1+l)\circ f_{i_{2t_{r'}+2}}^{j_{2s_{r'}+1}}\\
 &=&f_{i_{2t_{r'}+2}}^{j_{2s_{r'}+1}}\circ(-1+l'),
 \end{array} \]
 for some $l'\in \Rad\End_{\Lambda}(P_{j_{2s_{r'}+2}}).$
 Then Lemma\;\ref{pathcondition} implies that
 \[j_{2s_{r'}+2}\leq i_{2t_{r'}+2}.\] Since $(i_{2t_{r'}},j_{2s_{r'}})$ is maximum in $\xi^-$
  and $(i_{2t_{r'}},j_{2s_{r'}})\prec (i_{2t_{r'+1}},j_{2s_{r'}+2})$,
   we conclude that
  \[i_{2t_{r'}+2}\leq i_{2t_{r'+1}}\leq j_{2s_{r'}+2}.\]
  In particular, we have that
  \[i_{2t_{r'}+2}= i_{2t_{r'+1}}= j_{2s_{r'}+2}\ \mathrm{and}\ 
  h_{i_{2t_{r'+1}}}^{j_{2s_{r'}+2}}\equiv -1 \mod \Rad\End_{\Lambda}(P_{j_{2s_{r'}+2}}).\]\vspace{5pt}
If $r'<-1$, then  one can check that
 \[(\sharp) \left\{\begin{array}{rcccccccccclll}
&i_{2t_{r'}+2}&=&i_{2t_{r'+1}}&=&j_{2s_{r'}+2}&<&j_{2s_{r'}+3}&=&j_{2s_{r'+1}+1}&=&i_{2t_{r'+1}+1}&<\cdots\\
< & i_{2t_p+2}&=&i_{2t_{p+1}}&=&j_{2s_p+2}&<&j_{2s_p+3}&=&j_{2s_{p+1}+1}&=&i_{2t_{p+1}+1}&<\cdots\\
< & i_{2t_{-2}+2}&=&i_{2t_{-1}}&=&j_{2s_{-2}+2}&<&j_{2s_{-2}+3}&=&j_{2s_{-1}+1}&=&i_{2t_{-1}+1}\\
< & i_{2t_{-1}+2}&=& i_{2t_0}&=& j_{2s_{-1}+2}&<&j_{2s_{-1}+3}&\geq &i_{2t_0+1}
\end{array}\right.,\] 
and
\[(\natural) \left\{\begin{array}{cccl}
h_{i_{2t_p+1}}^{j_{2s_p+1}}&\equiv & 1 & \mod  \Rad\End_{\Lambda}(P_{j_{2s_p+1}})\\
h_{i_{2t_{p+1}}}^{j_{2s_p+2}}&\equiv & -1 & \mod  \Rad\End_{\Lambda}(P_{j_{2s_p+2}}),
\end{array}\right.\] 
for any $p\in\{r',r'+1,\dots, -1\}$ inductively. In the case that $r'=-1$, we have
 that
\[i_{2t_0+1}\leq j_{2s_{-1}+3}.\] In fact, we have an equation
\[f_{j_{2s_{-1}+2}}^{j_{2s_{-1}+3}}\circ h_{i_{2t_0}}^{j_{2s_{-1}+2}}
+ f_{j_{2s_{-1}+4}}^{j_{2s_{-1}+3}}\circ h_{i_{2t_0}}^{j_{2s_{-1}+4}}
+h_{i_{2t_0-1}}^{j_{2s_{-1}+3}}\circ f_{i_{2t_0}}^{i_{2t_0-1}}
+h_{i_{2t_0+1}}^{j_{2s_{-1}+3}}\circ f_{i_{2t_0}}^{i_{2t_0+1}}=0
\] and that $h_{i_{2t_0}}^{j_{2s_{-1}+2}}\equiv  -1\;\mod 
\Rad\End_{\Lambda}(P_{j_{2s_{-1}+2}})$. Accordingly, Lemma\;\ref{pathcondition}
implies the assertion.
Similar argument gives remaining assertions. 
(We only  note that if $(x',y')=(i_{2t_{r'}+1},j_{2s_{r'}+1})\in \xi_3^-$, 
then $-1\leq i_{2t_{r'}+1}<j_{2s_{r'}+1}$ implies that $s_{r'}\geq 0$. Therefore, we obtain that
$0<j_{2s_{r'}+1}<i_{2t_{r'}+2}$. In particular, we have $P_{i_{2t_{r'}+2}}\neq 0\neq P_{j_{2s_{r'}+1}}.$)

\end{pfclaim}
\begin{claim}
 \label{rigidcdclaim2}
\begin{enumerate}[{\rm (1)}]
\item If $(x,y)=(i_{2t_0+1},j_{2s_0+1})\in \xi^+_0$, then $\xi^{-}=\emptyset$.
\item If $(x,y)=(i_{2t_r},j_{2s_{r-1}+2})\in \xi^+_2$, then we have $i_{2t_r-1}=j_{2s_{r-1}+1}$ and 
\[(\ast)\left\{\begin{array}{rccccccccccll}
&i_{2t_r-2}&=&i_{2t_{r-1}}&=&j_{2s_{r-1}}&>&j_{2s_{r-1}-1}&=&j_{2s_{r-2}+1}&=&i_{2t_{r-1}-1}>\cdots\\
> & i_{2t_p-2}&=&i_{2t_{p-1}}&=&j_{2s_{p-1}}&>&j_{2s_{p-1}-1}&=&j_{2s_{p-2}+1}&=&i_{2t_{p-1}-1}>\cdots\\
> & i_{2t_2-2}&=&i_{2t_{1}}&=&j_{2s_{1}}&>&j_{2s_{1}-1}&=&j_{2s_{0}+1}&=&i_{2t_{1}-1}\\
> & j_{2s_0}&\geq& i_{2t_1-2}&=&i_{2t_0}.
\end{array}\right.\]
\item If $(x,y)=(i_{2t_r},j_{2s_r})\in \xi^+_3$, then we have $s_r=s_{r-1}+1$ and so
 \[(x,y)=(i_{2t_r},j_{2s_r})=(i_{2t_r},j_{2s_{r-1}+2}).\]
\item Assume that $(x,y)=(i_{2t_r+1},j_{2s_r+1})\in \xi^+_{4}$.
If $t_r<m_{\mathbf{i}}$, then $i_{2t_r}=j_{2s_r}>j_{2s_r-1}=j_{2s_{r-1}+1}=i_{2t_r-1}$
 and $(\ast)$ hold. If $t_r=m_{\mathbf{i}}$, then $i_{2t_r}\neq n+1$. Furthermore, we obtain that
 $i_{2t_r}=j_{2s_r}>j_{2s_r-1}=j_{2s_{r-1}+1}=i_{2t_r-1}$
 and $(\ast)$.
\item If $(x,y)=(i_{2t_{r+1}-1},j_{2s_r+1})\in \xi^+_1$ $(r\geq 0)$, then we have that $t_{r+1}=t_r+1$ and so 
\[(x,y)=(i_{2t_{r+1}-1},j_{2s_r+1})=(i_{2t_r+1},j_{2s_r+1}).\]

\end{enumerate}
 \end{claim} 
 \begin{pfclaim}
 We show (1). Suppose that $\xi^-\neq \emptyset$ and take a maximum element $(x',y')$ of $\xi^-$.
 If $(x',y')=(i_{2t_0},j_{2s_0})$, then $i_{2t_0}>j_{2s_0}$.
 Now we consider $f=f_{i_{2t_0}}^{j_{2s_0+1}}\in \Hom_{\Lambda}(P_{i_{2t_0}},P_{j_{2s_0+1}})$ given by the path
 $ w_{i_{2t_0}}^{j_{2s_0+1}}$. Since $\Hom_{\Kb(\proj \Lambda)}(X_{\mathbf{i}},X_{\mathbf{j}}[1])=0$,
 we have that $f$ factors through $P_{i_{2t_0-1}}\oplus P_{i_{2t_0+1}}\oplus P_{j_{2s_0}}\oplus P_{j_{2s_0+2}}$.
 Note that $w_{i_{2t_0}}^{j_{2s_0+1}}$ does not factor through
 $i_{2t_0-1}$, $j_{2s_0+2}$ and $j_{2s_0}$. 
 Therefore Lemma\;\ref{pathcondition} implies that 
 $i_{2t_0}<i_{2t_0+1}\leq  j_{2s_0+1}$.
 This contradicts to $i_{2t_0+1}>j_{2s_0+1}$.
   Hence we may assume that $(x',y')\in \xi^-_b\ (b=1,2,3,4)$. 
Then  Claim\;\ref{rigidcdadclaim} implies that 
\[i_{2t_0+1}\leq j_{2s_{-1}+3}\leq j_{2s_0+1}.\]
This contradicts to the hypothesis of $(1)$.



Next we show (2). We consider 
$\varphi\in \Hom_{\Lambda}(X_{\mathbf{i}}^{-1},Y_J^{0})=\Hom_{\Lambda}(\oplus P_{i_{2p}},\oplus P_{j_{2q+1}})$
given by $(\varphi_{i_{2p}}^{j_{2q+1}}:P_{i_{2p}}\to P_{j_{2q+1}})$ where  \[\varphi_{i_{2p}}^{j_{2q+1}}=\left\{\begin{array}{cl}
f_{i_{2t_r}}^{j_{2s_{r-1}+1}} & (p,q)=(t_r,s_{r-1})\\\\
0 & \mathrm{otherwise}.
\end{array}\right.\]
Then one can apply similar argument we used in the proof of Claim\;\ref{rigidcdadclaim} for $\varphi$ and
obtain $(\ast)$. Likewise, we have $(\ast)$ in the case of $(3)$, $(4)$ and $(5)$.
(For the assertion (4), we remark  that $(x,y)=(i_{2t_r+1},j_{2s_r+1})$ implies that $s_r<m_{\mathbf{j}}$. 
By definition of $s_r$, we have $i_{2t_r}<j_{2s_r+1}<j_{2m_{\mathbf{j}}+1}=n+2$ and so $i_{2t_r}\neq n+1$.)
 \end{pfclaim}

We continue a proof of Lemma\;\ref{rigidcondition}. (Remark: We now assume that $\xi^+\neq \emptyset$ and 
consider the case that $t_1=t_0+1$.)
 Suppose that 
$\xi^-\neq\emptyset$ and take a maximum element $(x',y')\in \xi^-.$
We will give a contradiction in the case  that $(x,y)\in \xi^+_4$ and $(x',y')\in \xi^-_4$.

Let $(\varphi,h,h')$ be a triple considered in the proof of Claim\;\ref{rigidcdadclaim}.
Then one has $(\sharp)$ and $(\natural)$.
 Therefore by Claim\;\ref{rigidcdclaim2}\;(4) and $(\sharp)$,
we conclude that
\[i_{2t_1-1}=i_{2t_0+1}\leq j_{2s_{-1}+3}\leq j_{2s_0+1}=i_{2t_1-1}. \]
In particular, we have that
\[j_{2s_{-1}+3}=j_{2s_0+1}=i_{2t_1-1}.\] Note that $\varphi_{i_{2t_0}}^{j_{2s_0+1}}=0$
and $h_{i_{2t_0}}^{j_{2_{s_0}}}\equiv -1  \mod  \Rad\End_{\Lambda}(P_{j_{2s_0}})$ 
(by $(\natural)$ and $s_{0}=s_{-1}+1$). Hence Lemma\;\ref{pathcondition} and $i_{2t_1-1}=j_{2s_0+1}$
 imply that
\[h_{i_{2t_1-1}}^{j_{2s_0+1}}\equiv  1  \mod  \Rad\End_{\Lambda}(P_{j_{2s_0+1}}).\]
Then by using Lemma\;\ref{pathcondition} and the condition $(\ast)$, one can check that 
\[\left\{\begin{array}{cccl}
h_{i_{2t_p-1}}^{j_{2s_{p-1}+1}}&\equiv & 1 & \mod  \Rad\End_{\Lambda}(P_{j_{2s_{p-1}+1}})\\
h_{i_{2t_p}}^{j_{2s_p}}&\equiv & -1 & \mod  \Rad\End_{\Lambda}(P_{j_{2s_p}}),
\end{array}\right.\]
for any $p\in\{1,\dots,r \}$ inductively. 
Now Claim\;\ref{rigidcdclaim2}\;(4) implies that
$i_{2t_r}\neq n+1(\Leftrightarrow P_{i_{2t_r}}\neq 0)$. Then $i_{2t_r+1}>j_{2s_r+1}$ gives that 
$s_r<m_{\mathbf{j}}(\Leftrightarrow P_{j_{2s_r+1}}\neq 0)$.
Note that $\varphi_{i_{2t_r}}^{j_{2s_r}+1}=0$ and
 \[h_{i_{2t_r}}^{j_{2s_r}}\equiv  -1 \mod  \Rad\End_{\Lambda}(P_{j_{2s_r}}).\]
If $t_r=m_{\mathbf{i}}$, then by Lemma\;\ref{pathcondition}, 
$w_{i_{2t_r}}^{j_{2s_r+1}}$ factors through either $j_{2s_r+2}$ or $i_{2t_r-1}$.
This is a contradiction.
 If $t_r<m_{\mathbf{i}}$, then Lemma\;\ref{pathcondition} implies that $w_{i_{2t_r}}^{j_{2s_r}+1}$ have to through $i_{2t_r+1}$.
 In particular, we obtain that  
 \[i_{2t_r}<i_{2t_r+1}\leq j_{2s_r+1}.\]
 This contradicts to that $(x,y)=(i_{2t_r+1},j_{2s_r+1})\in \xi_4^+$.
 
 If $(x,y)\in \xi^+_{a}$ and $(x',y')\in \xi^-_{b}$, then similar argument gives a contradiction.
 Therefore we have the assertion in the case that $t_1=t_0+1$. 
 
 

Suppose that $t_0<m_{\mathbf{i}}$ and $2t_1-2>2t$. Let $t':=t_{1}-1>t$ and 
$\mathbf{t}^{+}(t',s)=(t'_{0}<t'_{1}<\cdots)$, $\mathbf{s}^{-}(t',s)=(s'_{0}<s'_{1}<\cdots)$, 
$\mathbf{t}^{-}(t',s)=(t'_{0}>t'_{-1}>\cdots)$, $\mathbf{s}^{-}(t',s)=(s'_{0}>s'_{1}>\cdots)$. 

By definition we have the following:
\[\left\{ \begin{array}{cl}
 t'_{r}=t_{r} & r\geq 1\\
 s'_{r}=s_{r} & r\geq 1 \\
 \end{array}\right.\]
We also obtain that $t'_1=t'_0+1$. Hence the assertion holds for $(t',s)$.
If $(t',s)$ satisfies the condition (1), then the assertion is obvious. Therefore we assume that
 $(t',s)$ satisfies the condition (2).
If $s_0=0$, then we have $s_{-1}=-1=t_{-1}$ and
\[j_{2s_0}=j_{2s'_0}\geq i_{2t'_0}\geq i_{2t_0}.\]
Then it is easy to check the condition (2) of this lemma. 
Hence we may assume that $s_0>0$.
\begin{claim}
\label{rigidcdclaim1} We have the following.
\begin{enumerate}[{\rm (1)}]
\item There exists $\ell\geq 0$ such that 
\[\left\{ \begin{array}{cl}
 t'_{r-\ell}=t_{r} & r\leq -1\\
 s'_{r-\ell}=s_{r} & r\leq -1 \\
 \end{array}\right.\]
 \item If $t_0\neq t'_{-\ell}$, then we have  
 \[j_{2s_0}\geq j_{2s_{-1}+2}\geq i_{2t_{0}}>i_{2t_0-1} \geq j_{2s_{-1}+1}.\]
 In particular, $(t,s)$ satisfies the condition (2).
 \end{enumerate}
\end{claim}
\begin{pfclaim}
Note that $s'_{-1}\geq s_{-1}$.
First we assume that $s'_{-1}=s_{-1}$. If $s'_{-1}=s_{-1}=-1$, then $t'_{-1}=t_{-1}=-1$. If 
$s'_{-1}=s_{-1}\geq 0$, then we have that
\[i_{2t'_0}>\cdots>i_{2t_0}\geq i_{2t_{-1}+2}>j_{2s'_{-1}+1}=j_{2s_{-1}+1}>i_{2t_{-1}}.\]
Therefore we obtain $t'_{-1}=t_{-1}$. Hence the assertion is obvious.
Next we assume that $s'_{-1}>s_{-1}$. Let $\ell$ be a positive integer such that
$j_{2s'_{-\ell}+1}\geq j_{2s_{-1}+3} \geq j_{2s'_{-\ell-1}+3}$.
Thus we have that 
\[i_{2t'_{-\ell}+2}>j_{2s'_{-\ell}+1}\geq j_{2s_{-1}+3} \geq j_{2s'_{-\ell-1}+3}> i_{2t'_{-\ell}}.\]
Note that $j_{2s_{-1}+3}>i_{2t_0}>j_{2s_{-1}+1}$. If $j_{2s_{-1}+3}>j_{2s'_{-\ell-1}+3}$, then 
$j_{2s_{-1}+1}\geq j_{2s'_{-\ell-1}+3}$ and so we have 
\[i_{2t'_{-\ell}+2}>j_{2s'_{-\ell}+1}\geq j_{2s_{-1}+3}>i_{2t_0}>
j_{2s_{-1}+1}\geq j_{2s'_{-\ell-1}+3}> i_{2t'_{-\ell}}.\]
This gives that
\[t'_{-\ell}<t_0<t'_{-\ell}+1.\]
This is a contradiction. Hence we conclude that
\[s_{-1}=s'_{-\ell-1}.\]
If $i_{2t_0}>i_{2t'_{-\ell}}$, then we get that
\[i_{2t'_{-\ell}+2}>j_{2s'_{-\ell}+1}\geq j_{2s_{-1}+3}>i_{2t_0}>i_{2t'_{-\ell}}.\]
This is a contradiction.
Therefore, we obtain that 
\[i_{2t_{0}}\leq i_{2t'_{-\ell}}.\]
Then we have that
\[i_{2t'_{-\ell}}\geq i_{2t_0}\geq i_{2t_{-1}+2}>j_{2s'_{-\ell-1}+1}=j_{2s_{-1}+1}>i_{2t_{-1}}.\]
This implies  $t'_{-\ell-1}=t_{-1}$. (Remark: If $s_{-1}=s'_{\ell-1}=-1$, then $t'_{-\ell-1}=t_{-1}=-1$.) 
In particular, we conclude that
\[\left\{ \begin{array}{cl}
 t'_{r-\ell}=t_{r} & r\leq -1\\
 s'_{r-\ell}=s_{r} & r\leq -1 \\
 \end{array}\right.\]
 Suppose that $t_0\neq t'_{-\ell}$. In this case, we have $i_{2t'_{-\ell}}> i_{2t_{0}}$. 
 Since $(t',s)$ satisfies the condition (2), we conclude that
\[j_{2s_0}\geq j_{2s_{-1}+2}=j_{2s'_{-\ell-1}+2}\geq i_{2t'_{-\ell}}> i_{2t_0}>j_{2s'_{-\ell-1}+1}=j_{2s_{-1}+1}.\]
By applying Lemma\;\ref{pathcondition}\;(3) for $f_{i_{2t_0}}^{j_{2s'_{-\ell-1}+1}}$, we obtain that
\[i_{2t_{0}-1}\geq j_{2s'_{-\ell-1}+1}=j_{2s_{-1}+1}.\]
\end{pfclaim} 
Therefore by Claim\;\ref{rigidcdclaim1}, the assertion also holds for the case that
 $2t_1-2>2t$. 
\end{proof}
For $\mathbf{i},\mathbf{j}\in \Xi$ and $f:P_{i_{2t}}\to P_{j_{2s+1}}$, we set
$\varphi(\mathbf{i},\mathbf{j},f)=\varphi(f):=
(\varphi_{i_{2p}}^{j_{2q+1}}:P_{i_{2p}}\to P_{j_{2q+1}})\in \Hom_{\Lambda}(X_{\mathbf{i}}^{-1},X_{\mathbf{j}}^0)$
 where
\[\varphi_{i_{2p}}^{j_{2q+1}}=\left\{\begin{array}{cl}
f & (p,q)=(t,s)\\\\
0 & \mathrm{otherwise}.
\end{array}\right.\]
We regard $\varphi(t,s,f)$ as a morphism in $\Hom_{\Kb(\proj \Lambda)}(X_{\mathbf{i}},X_{\mathbf{j}}[1])$.
\begin{lemma}
\label{rigidcondition3} Let $\mathbf{i},\mathbf{j}\in \Xi$.
Let $(t,s)$ be a pair such that $0<i_{2t}<j_{2s+1}<n+1$.
If either (1) or (2) of Lemma\;\ref{rigidcondition} holds, then $\varphi(f)=0$
 for any $f\in \Hom_{\Lambda}(P_{i_{2t}},P_{j_{2s+1}}).$.
\end{lemma}
\begin{proof}
We first consider the case that the assertion (1) of Lemma\;\ref{rigidcondition} holds.
We note that $i_{2t+1}=i_{2t_0+1}\leq j_{2s_0+1}=j_{2s+1}\leq n$. In particular, we have
$t<m_{\mathbf{i}}$ and $s<m_{\mathbf{j}}$.
\begin{claim}
\label{rigidcdforxclaim}
We have the following.
\begin{enumerate}[{\rm (1)}]
\item $t_1\leq m_{\mathbf{i}}$ and $s_1\leq m_{\mathbf{j}}$.
\item If $t_1\leq m_{\mathbf{i}}$ and $s_1<m_{\mathbf{j}}$, then there exists 
$h_{i_{2p+1}}^{j_{2s+1}}:P_{i_{2p+1}}\to P_{j_{2s+1}}$ 
and $h_{i_{2t_1}}^{j_{2q+2}}:P_{i_{2t_1}}\to P_{j_{2q+1}}$ for any $p\in \{t,\dots, t_1-1\}$ and $q\in \{s,\dots,s_1-1\}$
such that 
\[\varphi(f)-h\circ d_{\mathbf{i}}-d_{\mathbf{j}}\circ h'=\varphi(g)\]
for some $g:P_{2t_1}\to P_{2s_1+1}$, where $h\in \Hom_{\Lambda}(X_{\mathbf{i}}^0,X_{\mathbf{j}}^0) $
and $h'\in \Hom_{\Lambda}(X_{\mathbf{i}}^{-1},X_{\mathbf{j}}^{-1})$ 
are morphisms given by $\{h_{i_{2p+1}}^{j_{2s+1}}\mid t\leq p\leq t_1-1\}$
and $\{h_{i_{2t_1}}^{j_{2q+2}}\mid s\leq q\leq s_1-1\}$ respectively.
\item If $i_{2t_1}=n+1$, then $\varphi(f)=0$ in $\Kb(\proj \Lambda)$.
\item If $s_1=m_{\mathbf{j}}$, then $\varphi(f)=0$ in $\Kb(\proj \Lambda)$.
\item If $t_1=m_{\mathbf{i}}$, then $s_1=m_{\mathbf{j}}$.
\end{enumerate}
\end{claim}
\begin{pfclaim}
We show (1). By the condition (1)-(b) of Lemma\;\ref{rigidcondition} 
(note that $t_1\leq m_{\mathbf{i}}+1$ by definition of $t_1$), we get that
\[i_{2t_1-1}\leq j_{2s+1}\leq n.\]
This implies that $t_1\leq m_{\mathbf{i}}.$ Then $s_1\leq m_{\mathbf{j}}$ follows from the definition of
$s_1$.
We prove (2) and (3). By Lemma\;\ref{pathcondition}, we can write
$f=l \circ f_{i_{2t}}^{j_{2s+1}}$ for some $l\in \End_{\Lambda}(P_{j_{2s+1}})$.  
By the condition (1)-(a) of Lemma\;\ref{rigidcondition}, we conclude that
  $w_{i_{2t}}^{j_{2s+1}}$ factors through $i_{2t_0+1}$.
 In particular, we obtain that
 \[\varphi(f)-h_1\circ d_{\mathbf{i}}=\varphi(f_1),\]
 for some $f_1:P_{i_{2t+2}}\to P_{j_{2s+1}}$,
 where $h_1$ is a morphism in $\Hom_{\Kb(\proj \Lambda)}(X_{\mathbf{i}}^0,X_{\mathbf{j}}^0)$ given by
 $h_{i_{2t+1}}^{j_{2s+1}}:=l\circ f_{i_{2t+1}}^{j_{2s+1}}$ and $f_1:P_{i_{2t+2}}\to P_{j_{2s+1}}$.
 Inductively, one can constructs $h_{i_{2p+1}}^{j_{2s+1}}:P_{i_{2p+1}}\to P_{j_{2s+1}}$
   for any
 $p\in \{t,t+1,\dots,t_1-2\}$ and $f':P_{i_{2t_1-2}}\to P_{j_{2s+1}}$ such that
 \[\varphi(t,s,f)-h'\circ d_{\mathbf{i}}=\varphi(f'),\]
 where $h'$ is a morphism in $\Hom_{\Kb(\proj \Lambda)}(X_{\mathbf{i}}^0,X_{\mathbf{j}}^0)$ given by
 $\{h_{i_{2p+1}}^{j_{2s+1}}\mid t\leq p\leq t_1-2\}$.
  By using the condition (1)-(b) of Lemma\;\ref{rigidcondition},  
 there are $h_{i_{2t_1-1}}^{j_{2s+1}}:P_{i_{2t_1-1}}\to P_{j_{2s+1}}$ and $f^{(1)}:P_{i_{2t_1}}\to P_{j_{2s+1}}$
 such that
 \[\varphi(f)-h\circ d_{\mathbf{i}}=\varphi(f^{(1)}),\]
 where $h$ is a morphism in $\Hom_{\Kb(\proj \Lambda)}(X_{\mathbf{i}}^0,X_{\mathbf{j}}^0)$ given by
 $\{h_{i_{2p+1}}^{j_{2s+1}}\mid t\leq p\leq t_1-1\}$. 
Thus we have  $\varphi(f)-h\circ d_{\mathbf{i}}=0$ in tha case that $i_{2t_1}=n+1$ and get the assertion (3).
Assume that $i_{2t_1}<n+1$.
 Lemma\;\ref{pathcondition} implies that there exists $l'\in \End_{\Lambda}(P_{i_{2t_1}})$
 such that $f^{(1)}=f_{i_{2t_1}}^{j_{2s+1}}\circ l'$. Now by the condition (1)-(c) of
  Lemma\;\ref{rigidcondition}, we obtain that
 $w_{i_{2t_1}}^{j_{2s+1}}$ factors through $j_{2s+2}$. Therefore, we conclude that
 \[\varphi(f^{(1)})-d_{\mathbf{j}}\circ h'_1=\varphi(f^{(1)}_1),\]
 for some $f^{(1)}_1:P_{i_{2t+2}}\to P_{j_{2s+1}}$,
 where $h'_1$ is a morphism in $\Hom_{\Kb(\proj \Lambda)}(X_{\mathbf{i}}^{-1},X_{\mathbf{j}}^{-1})$ given by
 $h_{i_{2t_1}}^{j_{2s+2}}:= f_{i_{2t_1}}^{j_{2s+2}}\circ l'$.
Inductively, (and by using the condition (1)-(d) of Lemma\;\ref{rigidcondition}),
we can construct $h_{i_{2t_1}}^{j_{2q+2}}:P_{i_{2t_1}}\to P_{j_{2q+2}}$ for any
 $q\in \{s,s+1,\dots,s_1-1\}$ and $g:P_{i_{2t_1}}\to P_{j_{2s+1}}$ such that
 \[\varphi(f^{(1)})-d_{\mathbf{j}}\circ h=\varphi(g),\]
 where $h'$ is a morphism in $\Hom_{\Kb(\proj \Lambda)}(X_{\mathbf{i}}^{-1},X_{\mathbf{j}}^{-1})$ given by
 $\{h_{i_{2t_1}}^{j_{2q+2}}\mid s\leq q\leq s_1-1\}$.
 
 
We prove (4). By the assertion (3), we may assume that $i_{2t_1}\leq n$.
Then $j_{2s_1}\leq i_{2t_1}\leq n$ follows from the condition (1)-(d) of Lemma\;\ref{rigidcondition}.
Therefore, for a morphisms $h\in \Hom_{\Lambda}(X_{\mathbf{i}}^0,X_{\mathbf{j}}^0) $
and $h'\in \Hom_{\Lambda}(X_{\mathbf{i}}^{-1},X_{\mathbf{j}}^{-1})$ constructed in the proof of (2), we have that
\[\varphi(f)-h\circ d_{\mathbf{i}}-d_{\mathbf{j}}\circ h'=0.\] 

Finally we show the assertion (5).  The condition (1)-(a) of Lemma\;\ref{rigidcondition}
implies that 
\[n+2=i_{2t_1+1}\leq j_{2s_1+1}. \]
Hence $s_1=m_{\mathbf{j}}$.
\end{pfclaim}
Therefore if $i_{2t}<j_{2s+1}$ and the assertion (1) of 
Lemma\;\ref{rigidcondition} hold, then $\varphi(f)=0$ follows from Claim\;\ref{rigidcdforxclaim}.

Next we assume that the condition (2) of Lemma\;\ref{rigidcondition} holds.
We note that $j_{2s}=j_{2s_0}\geq i_{2t_0}=i_{2t}\geq 1$. 
\begin{claim}
\label{rigidcdforxclaim2}
We have the following.
\begin{enumerate}[{\rm (1)}]
\item $s_{-1},t_{-1}\geq -1$.
\item $j_{2s_{-1}+2}>0(\Leftrightarrow P_{j_{2s_{-1}+2}}\neq 0).$
\item If $s_{-1}\geq 0$ and $t_{-1}\geq 0$, then there exists 
$h_{i_{2t}}^{j_{2q}}:P_{i_{2t}}\to P_{j_{2q}}$ 
and $h_{i_{2p+1}}^{j_{2s_{-1}+1}}:P_{i_{2p-1}}\to P_{j_{2s_{-1}+1}}$
 for any $q\in \{s,\dots,s_{-1}+1\}$ and $p\in \{t,\dots, t_{-1}+1\}$
such that 
\[\varphi(f)-d_{\mathbf{j}}\circ h'-h\circ d_{\mathbf{i}}=\varphi(g)\]
for some $g:P_{2t_{-1}}\to P_{2s_{-1}+1}$, where $h'\in \Hom_{\Lambda}(X_{\mathbf{i}}^{-1},X_{\mathbf{j}}^{-1})$ and 
$h\in \Hom_{\Lambda}(X_{\mathbf{i}}^0,X_{\mathbf{j}}^0) $
are morphisms given by $\{h_{i_{2t}}^{j_{2q}}\mid s\geq q\geq s_{-1}+1\}$
and $\{h_{i_{2p-1}}^{j_{2s_{-1}+1}}\mid t\geq p\geq t_{-1}+1\}$ respectively.
\item If $s_{-1}=-1$, then $\varphi(f)=0$ in $\Kb(\proj \Lambda)$.
\item If $t_{-1}=-1$, then $s_{-1}=-1$.
\end{enumerate}
\end{claim}
\begin{pfclaim}
The assertion (1) follows from the definitions of $s_{-1}$ and $t_{-1}$.
The condition (2)-(b') of Lemma\;\ref{rigidcondition} gives that $j_{2s_{-1}+2}\geq i_{2t_0}>0$. Hence we obtain the assertion (2).
The assetion (5) follows from (2)-(d') of Lemma\;\ref{rigidcondition}. In fact, if $t_{-1}=-1$, then
$-1=i_{2t_{-1}+1}\geq j_{2s_{-1}+1}$.
One can apply similar argument used in the proof of Claim\;\ref{rigidcdforxclaim} (2),(3)
and get the assertion (3) and (4). 
\end{pfclaim}
Then $\varphi(f)=0$ directly follows from Claim\;\ref{rigidcdforxclaim2}.
\end{proof}

By considering labeling-change $i\leftrightarrow n+1-i$ on $Q_0\sqcup \{0,n+1\}=\{0,1,\dots,n,n+1\}$, 
we also obtain the following. 
\begin{lemma}
\label{rigidcondition2}
Let $\mathbf{i},\mathbf{j}\in \Xi$.
For a pair $(t,s)$ such that $n+1>i_{2t}>j_{2s-1}>0$,
 we define two sequences $\mathbf{t}^{+}(t,s):=(t=t_{0}\geq t_{1}\geq t_{2}
 \geq \cdots)$ and 
$\mathbf{s}^{+}(t,s):=(s=s_{0}\geq s_{1}\geq  s_{2}\cdots)$
as follows$:$
\begin{itemize}
\item[{\rm (i)}] $t_r:=\left\{\begin{array}{cl}
\mathrm{min}\{t\geq -1\mid i_{2t+2}>j_{2s_{r-1}-1}\}& \mathrm{if\ } s_{r-1}\geq 1\\
-1 & \mathrm{if\ } s_{r-1}\leq 0, t_{r-1}\geq 0\\
-2 & \mathrm{if\ } t_{r-1}\leq -1\\
\end{array}\right..
 $
 \item[{\rm (ii)}] $s_r:=\left\{\begin{array}{cl}
\mathrm{min}\{s\geq 0 \mid j_{2s+1}>i_{2t_r}\}& \mathrm{if\ } t_{r}\geq 0\\
0 & \mathrm{if\ } t_{r}=-1, s_{r-1}\geq 1\\
-1 & \mathrm{if\ } s_{r-1}\leq 0\\
\end{array}\right..
 $
\end{itemize}
Also we define two sequences $\mathbf{t}^{-}(t,s):=(t=t_{0}\leq t_{-1}\leq t_{-2}\leq \cdots)$ and 
$\mathbf{s}^{-}(t,s):=(s=s_{0}\leq s_{-1}\leq s_{-2}\leq \cdots)$
as follows$:$
\begin{itemize}
\item[{\rm (iii)}] $s_r:=\left\{\begin{array}{cl}
\mathrm{max}\{s\leq m_{\mathbf{j}}+1 \mid j_{2s-3}<i_{2t_{r+1}}\}& \mathrm{if\ } t_{r+1}\leq m_{\mathbf{i}}\\
m_{\mathbf{j}}+1 & \mathrm{if\ } t_{r+1}=m_{\mathbf{i}}+1, s_{r+1}\leq m_{\mathrm{j}}\\
m_{\mathbf{j}}+2 & \mathrm{if\ } s_{r-1}\geq m_{\mathrm{j}}+1\\
\end{array}\right..
 $
\item[{\rm (iv)}] $t_r:=\left\{\begin{array}{cl}
\mathrm{max}\{t\leq m_{\mathbf{i}+1} \mid i_{2t-2}<j_{2s_r-1}\}& \mathrm{if\ } s_r\leq m_{\mathbf{j}}\\
m_{\mathbf{i}}+1 & \mathrm{if\ } s_r\geq m_{\mathbf{j}}+1, t_{r+1}\leq m_{\mathbf{i}}\\
m_{\mathbf{i}}+2 & \mathrm{if\ } t_{r+1}\geq m_{\mathbf{i}}+1\\
\end{array}\right..
$
\end{itemize}
If $\Hom_{\Kb(\proj \Lambda)}(X_{\mathbf{i}},X_{\mathbf{j}}[1])=0$, then one of the following holds.
\begin{itemize}
\item[(1)] We have {\rm (a),\ (b),\ (c)} and {\rm (d)}.
\begin{enumerate}[{\rm (a)}]
\item $i_{2t_{r}+2}>i_{2t_{r}+1}\geq
 j_{2s_{r-1}-1}$ for any $r\geq 1$ such that $t_r\geq -1$.
\item $j_{2s_{r-1}-1}>j_{2s_{r-1}-2}\geq
 i_{2t_r}$ for any $r\geq 1$ such that $s_r\geq 0$.
\item $j_{2s_{r}+1}>j_{2s_{r}}\geq
 i_{2t_r}$ for any $r\geq 1$ such that $s_r\geq 0$.
\item $i_{2t_r}>i_{2t_{r}-1}\geq
 j_{2s_{r}-1}$ for any $r\geq 0$ such that $t_r\geq 0$.
\end{enumerate} 
Where we put $i_{-1}=j_{-1}=-1$ and $i_{-2}=j_{-2}=-2$. 
\item[(2)] We have {\rm (a'),\ (b'),\ (c')} and {\rm (d')}.
\begin{enumerate}[{\rm (a')}]
\item $j_{2s_{r}-3}<j_{2s_{r}-2}\leq
 i_{2t_{r+1}}$ for any $r\leq -1$ such that $s_{r}\leq m_{\mathbf{j}}+1$.
\item $i_{2t_{r+1}}<i_{2t_{r+1}+1}\leq
 j_{2s_{r}-1}$ for any $r\leq -1$ such that $t_{r}\leq m_{\mathbf{i}}+1$.
\item $i_{2t_{r}-2}<i_{2t_{r}-1}\leq
 j_{2s_{r}-1}$ for any $r\leq -1$ such that $t_{r}\leq m_{\mathbf{i}}+1$.
\item $j_{2s_{r}-1}<j_{2s_{r}}\leq
 i_{2t_{r}}$ for any $r\leq -1$ such that $s_{r}\leq m_{\mathbf{j}}$.
\end{enumerate}
Where we put $i_{2m_{\mathbf{i}}+1}=j_{2m_{\mathbf{j}}+1}=n+2$ and $i_{2m_{\mathbf{i}}+2}=j_{2m_{\mathbf{j}}+2}=n+3$.
\end{itemize}
\end{lemma}
\begin{lemma}
\label{rigidcondition4} Let $\mathbf{i},\mathbf{j}\in \Xi$.
Let $(t,s)$ be a pair such that $n+1>i_{2t}>j_{2s-1}>0$.
If either (1) or (2) of Lemma\;\ref{rigidcondition2} holds, then $\varphi(f)=0$
 for any $f\in \Hom_{\Lambda}(P_{i_{2t}},P_{j_{2s-1}})$.
\end{lemma}
By Lemma\;\ref{rigidcondition}, Lemma\;\ref{rigidcondition3}, Lemma\;\ref{rigidcondition2} and Lemma\;\ref{rigidcondition4},
we obtain a combinatorial description of $\Hom_{\Kb(\proj \Lambda)}(X_{\mathbf{i}},X_{\mathbf{j}}[1])=0$.
\begin{proposition}
\label{rigidconditionforx}
Let $\mathbf{i},\mathbf{j}\in Xi$. Then $\Hom_{\Kb(\proj \Lambda)}(X_{\mathbf{i}},X_{\mathbf{j}}[1])=0$ if and only if
the following two conditions hold.
\begin{enumerate}[{\rm (a)}]
\item For any pair $(t,s)$ with $0<i_{2t}<j_{2s+1}<n+1$, either (1) or (2) of Lemma\;\ref{rigidcondition} holds.
\item For any pair $(t,s)$ with $0>i_{2t}>j_{2s-1}>n+1$, either (1) or (2) of Lemma\;\ref{rigidcondition2} holds.
\end{enumerate}
\end{proposition} 
Let $\mathbb{P}_{\Lambda}:=\sttilt \Lambda\cap \add \bigoplus_{\mathbf{i}\in \Xi} X_{\mathbf{i}}$. 
\begin{lemma}
\label{casepreproj}
Let $\Lambda$ be the preprojective algebra of type $A_n$. Then 
$\mathbb{P}_{\Lambda}=\sttilt\Lambda$.
\end{lemma}
 Note that the number of isomorphism classes of
  indecomposable 2-term presilting object of $\Kb(\proj \Lambda)$
is equal to $\#\trigid \Lambda+n$. Then Lemma\;\ref{trigidofpreproj} implies
 that
\[\#\trigid \Lambda+n\leq \#\{\mathbf{i}\in \Xi\mid m_{\mathbf{i}}>1\}+n=\#\Xi.\]
By Lemma\;\ref{indecomposability}, Proposition\;\ref{rigidconditionforx}, we have that
$X_{\mathbf{i}}$ is indecomposable 2-term silting object of $\Kb(\proj \Lambda)$ for any $\mathbf{i}\in \Xi$.
Hence $X$ is an indecomposable 2-term presilting object of $\Kb(\proj \Lambda)$ if and only if
$X$ is isomorphic to $X_{\mathbf{i}}$ for some $\mathbf{i}\in \Xi$. 
\begin{proposition}
\label{ifpart}
Assume that $\Lambda$ satisfies the Condition\;\ref{nscd}. Then we have
\[\sttilt \Lambda\simeq (\Sym_{n+1},\leq).\]
\end{proposition}
\begin{proof}
By Lemma\;\ref{indecomposability}, Proposition\;\ref{rigidconditionforx} and Lemma\;\ref{casepreproj}, we have that
$\mathbb{P}_{\Lambda}$ is a full subposet of $\sttilt \Lambda$ and 
\[\mathbb{P}_{\Lambda}\simeq (\Sym_{n+1},\leq).\]
This gives the assertion. In fact, $\mathbb{P}_{\Lambda}$ is $n$-regular and so finite connected component of $\sttilt \Lambda$.
Hence we have $\mathbb{P}_{\Lambda}=\sttilt \Lambda$.
\end{proof}

\subsection{`only if' part}
In this subsection, we prove that $\sttilt \Lambda\simeq (\Sym_{n+1},\leq)$ induces the Condition\;\ref{nscd}.
For a proof, we use an induction on $n$. In subsection\;\ref{subsec:n=2}, we treated the case $n=2$ 
and in subsection\;\ref{subsec:ifpart}, we showed  that $\sttilt \Lambda\simeq (\Sym_{n+1},\leq)$ if 
$\Lambda$ satisfies the Condition\;\ref{nscd}.
 Hence we may assume that Theorem\;\ref{mainresult} holds for any $\Lambda$ such that $\#Q_0<n$ and consider the case
that $\#Q_0=n$.

Let $\Lambda=kQ/I$ be an algebra such that
\[\sttilt \Lambda\simeq (\Sym_{n+1},\leq).\]
Let $\rho: (\Sym_{n+1},\leq)\stackrel{\sim}{\to} \sttilt \Lambda $
and denote by $T_w:=\rho (w)$. 
For any $a\in Q_0$, we set $X_a:=e_a\Lambda/e_a\Lambda (1-e_a)\Lambda$.
Then $\dip(0)=\{X_a\mid a\in Q_0\}=\{T_{s_i}\mid 1\leq i\leq n\}$. Therefore we may assume that
$Q_{0}=\{1,\dots,n\}$ and $T_{s_i}=X_i$.
\begin{lemma}
\label{determiningarrow}
For $i\neq j\in Q_0$, we set $X_{i,j}:=X_i\bigvee X_j$.
\begin{enumerate}[{\rm (1)}]
\item $[0, X_{i,j}]$
has one of the following form.  
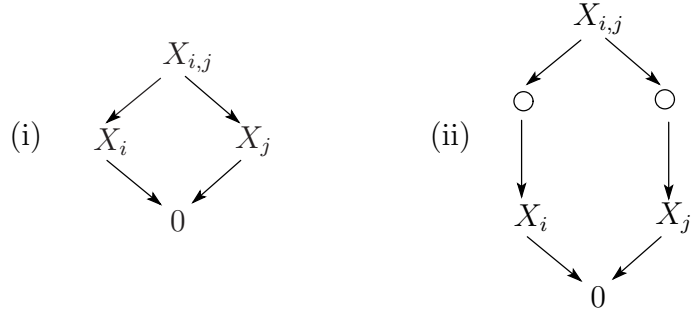
\[
{\unitlength 0.1in%
\begin{picture}( 34.8000, 14.3000)(  9.6000,-22.7000)%
\put(18.0000,-20.0000){\makebox(0,0)[lb]{$0$}}%
\put(14.0000,-16.0000){\makebox(0,0)[lb]{$X_{i}$}}%
\put(21.4000,-16.1000){\makebox(0,0)[lb]{$X_{j}$}}%
\put(17.6000,-11.7000){\makebox(0,0)[lb]{$X_{i,j}$}}%
%
\special{pn 8}%
\special{pa 1750 1200}%
\special{pa 1470 1430}%
\special{fp}%
\special{sh 1}%
\special{pa 1470 1430}%
\special{pa 1534 1403}%
\special{pa 1511 1396}%
\special{pa 1509 1372}%
\special{pa 1470 1430}%
\special{fp}%
%
\special{pn 8}%
\special{pa 1470 1630}%
\special{pa 1750 1860}%
\special{fp}%
\special{sh 1}%
\special{pa 1750 1860}%
\special{pa 1711 1802}%
\special{pa 1709 1826}%
\special{pa 1686 1833}%
\special{pa 1750 1860}%
\special{fp}%
%
\special{pn 8}%
\special{pa 1880 1190}%
\special{pa 2160 1430}%
\special{fp}%
\special{sh 1}%
\special{pa 2160 1430}%
\special{pa 2122 1371}%
\special{pa 2120 1395}%
\special{pa 2096 1402}%
\special{pa 2160 1430}%
\special{fp}%
%
\special{pn 8}%
\special{pa 2170 1630}%
\special{pa 1920 1850}%
\special{fp}%
\special{sh 1}%
\special{pa 1920 1850}%
\special{pa 1983 1821}%
\special{pa 1960 1815}%
\special{pa 1957 1791}%
\special{pa 1920 1850}%
\special{fp}%
\put(9.6000,-15.9000){\makebox(0,0)[lb]{$\mathrm{(i)}$}}%
\put(40.0000,-24.0000){\makebox(0,0)[lb]{$0$}}%
\put(36.0000,-20.0000){\makebox(0,0)[lb]{$X_{i}$}}%
\put(43.4000,-20.1000){\makebox(0,0)[lb]{$X_{j}$}}%
\put(39.1000,-9.7000){\makebox(0,0)[lb]{$X_{i,j}$}}%
%
\special{pn 8}%
\special{pa 3950 1000}%
\special{pa 3670 1230}%
\special{fp}%
\special{sh 1}%
\special{pa 3670 1230}%
\special{pa 3734 1203}%
\special{pa 3711 1196}%
\special{pa 3709 1172}%
\special{pa 3670 1230}%
\special{fp}%
%
\special{pn 8}%
\special{pa 3670 2030}%
\special{pa 3950 2260}%
\special{fp}%
\special{sh 1}%
\special{pa 3950 2260}%
\special{pa 3911 2202}%
\special{pa 3909 2226}%
\special{pa 3886 2233}%
\special{pa 3950 2260}%
\special{fp}%
%
\special{pn 8}%
\special{pa 4080 990}%
\special{pa 4360 1230}%
\special{fp}%
\special{sh 1}%
\special{pa 4360 1230}%
\special{pa 4322 1171}%
\special{pa 4320 1195}%
\special{pa 4296 1202}%
\special{pa 4360 1230}%
\special{fp}%
%
\special{pn 8}%
\special{pa 4370 2030}%
\special{pa 4120 2250}%
\special{fp}%
\special{sh 1}%
\special{pa 4120 2250}%
\special{pa 4183 2221}%
\special{pa 4160 2215}%
\special{pa 4157 2191}%
\special{pa 4120 2250}%
\special{fp}%
\put(31.6000,-15.9000){\makebox(0,0)[lb]{$\mathrm{(ii)}$}}%
%
%
\special{pn 8}%
\special{ar 3650 1320 50 50  0.0000000  6.2831853}%
%
%
\special{pn 8}%
\special{ar 4390 1310 50 50  0.0000000  6.2831853}%
%
\special{pn 8}%
\special{pa 3640 1420}%
\special{pa 3640 1820}%
\special{fp}%
\special{sh 1}%
\special{pa 3640 1820}%
\special{pa 3660 1753}%
\special{pa 3640 1767}%
\special{pa 3620 1753}%
\special{pa 3640 1820}%
\special{fp}%
%
\special{pn 8}%
\special{pa 4410 1430}%
\special{pa 4410 1830}%
\special{fp}%
\special{sh 1}%
\special{pa 4410 1830}%
\special{pa 4430 1763}%
\special{pa 4410 1777}%
\special{pa 4390 1763}%
\special{pa 4410 1830}%
\special{fp}%
\end{picture}}%
\vspace{5pt}\]
\item In the case of {\rm (i)}, there is no edge between $i$ and $j$ in $Q$.
\item In the case of {\rm (ii)}, then we have 
\[
{\unitlength 0.1in%
\begin{picture}(  3.5000,  1.2700)(  4.0000, -3.9700)%
\put(4.0000,-4.0000){\makebox(0,0)[lb]{$i$}}%
%
\special{pn 8}%
\special{pa 540 320}%
\special{pa 740 320}%
\special{fp}%
\special{sh 1}%
\special{pa 740 320}%
\special{pa 673 300}%
\special{pa 687 320}%
\special{pa 673 340}%
\special{pa 740 320}%
\special{fp}%
%
\special{pn 8}%
\special{pa 710 370}%
\special{pa 510 370}%
\special{fp}%
\special{sh 1}%
\special{pa 510 370}%
\special{pa 577 390}%
\special{pa 563 370}%
\special{pa 577 350}%
\special{pa 510 370}%
\special{fp}%
\put(7.5000,-4.1000){\makebox(0,0)[lb]{$j$}}%
\end{picture}}%
\vspace{5pt}\]
\item $Q^{\circ}$ is a double quiver of type $A_n$, i.e. 
\[
{\unitlength 0.1in%
\begin{picture}( 20.5000,  1.3400)( 18.0000, -7.5400)%
\put(22.2000,-7.5000){\makebox(0,0)[lb]{$1$}}%
%
\special{pn 8}%
\special{pa 2360 670}%
\special{pa 2560 670}%
\special{fp}%
\special{sh 1}%
\special{pa 2560 670}%
\special{pa 2493 650}%
\special{pa 2507 670}%
\special{pa 2493 690}%
\special{pa 2560 670}%
\special{fp}%
%
\special{pn 8}%
\special{pa 2530 720}%
\special{pa 2330 720}%
\special{fp}%
\special{sh 1}%
\special{pa 2330 720}%
\special{pa 2397 740}%
\special{pa 2383 720}%
\special{pa 2397 700}%
\special{pa 2330 720}%
\special{fp}%
\put(25.7000,-7.6000){\makebox(0,0)[lb]{$2$}}%
%
\special{pn 8}%
\special{pa 2720 670}%
\special{pa 2920 670}%
\special{fp}%
\special{sh 1}%
\special{pa 2920 670}%
\special{pa 2853 650}%
\special{pa 2867 670}%
\special{pa 2853 690}%
\special{pa 2920 670}%
\special{fp}%
%
\special{pn 8}%
\special{pa 2890 720}%
\special{pa 2690 720}%
\special{fp}%
\special{sh 1}%
\special{pa 2690 720}%
\special{pa 2757 740}%
\special{pa 2743 720}%
\special{pa 2757 700}%
\special{pa 2690 720}%
\special{fp}%
\put(29.3000,-7.6000){\makebox(0,0)[lb]{$3$}}%
%
\special{pn 8}%
\special{pa 3080 677}%
\special{pa 3280 677}%
\special{fp}%
\special{sh 1}%
\special{pa 3280 677}%
\special{pa 3213 657}%
\special{pa 3227 677}%
\special{pa 3213 697}%
\special{pa 3280 677}%
\special{fp}%
%
\special{pn 8}%
\special{pa 3250 727}%
\special{pa 3050 727}%
\special{fp}%
\special{sh 1}%
\special{pa 3050 727}%
\special{pa 3117 747}%
\special{pa 3103 727}%
\special{pa 3117 707}%
\special{pa 3050 727}%
\special{fp}%
%
\special{pn 8}%
\special{pa 3630 677}%
\special{pa 3830 677}%
\special{fp}%
\special{sh 1}%
\special{pa 3830 677}%
\special{pa 3763 657}%
\special{pa 3777 677}%
\special{pa 3763 697}%
\special{pa 3830 677}%
\special{fp}%
%
\special{pn 8}%
\special{pa 3800 727}%
\special{pa 3600 727}%
\special{fp}%
\special{sh 1}%
\special{pa 3600 727}%
\special{pa 3667 747}%
\special{pa 3653 727}%
\special{pa 3667 707}%
\special{pa 3600 727}%
\special{fp}%
%
\special{pn 4}%
\special{sh 1}%
\special{ar 3320 700 8 8 0  6.28318530717959E+0000}%
\special{sh 1}%
\special{ar 3520 700 8 8 0  6.28318530717959E+0000}%
%
\special{pn 4}%
\special{sh 1}%
\special{ar 3420 700 8 8 0  6.28318530717959E+0000}%
\put(38.5000,-7.6000){\makebox(0,0)[lb]{$n$}}%
\put(18.0000,-7.8000){\makebox(0,0)[lb]{$Q^{\circ}=$}}%
\end{picture}}%
\]
\end{enumerate}
\end{lemma} 
\begin{proof}
Note that $[1,s_i\vee s_j]$ has following form.
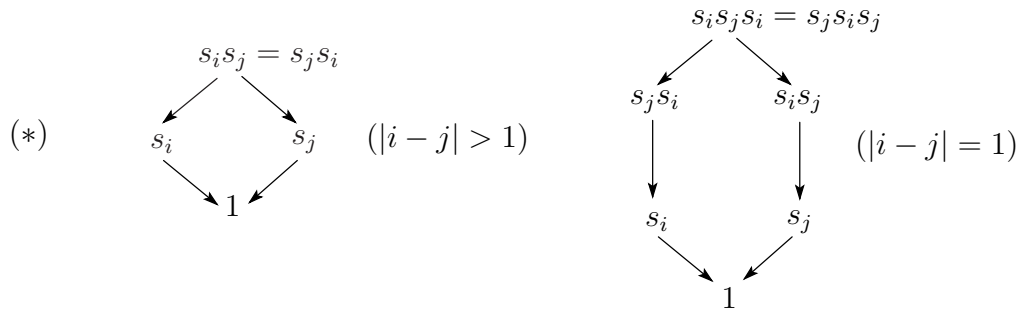
\[
{\unitlength 0.1in%
\begin{picture}( 44.3000, 14.4000)(  6.6000,-22.7000)%
\put(14.0000,-15.9000){\makebox(0,0)[lb]{$s_i$}}%
\put(21.4000,-15.9000){\makebox(0,0)[lb]{$s_j$}}%
\put(16.5000,-11.6000){\makebox(0,0)[lb]{$s_is_j=s_js_i$}}%
%
\special{pn 8}%
\special{pa 1750 1200}%
\special{pa 1470 1430}%
\special{fp}%
\special{sh 1}%
\special{pa 1470 1430}%
\special{pa 1534 1403}%
\special{pa 1511 1396}%
\special{pa 1509 1372}%
\special{pa 1470 1430}%
\special{fp}%
%
\special{pn 8}%
\special{pa 1470 1630}%
\special{pa 1750 1860}%
\special{fp}%
\special{sh 1}%
\special{pa 1750 1860}%
\special{pa 1711 1802}%
\special{pa 1709 1826}%
\special{pa 1686 1833}%
\special{pa 1750 1860}%
\special{fp}%
%
\special{pn 8}%
\special{pa 1880 1190}%
\special{pa 2160 1430}%
\special{fp}%
\special{sh 1}%
\special{pa 2160 1430}%
\special{pa 2122 1371}%
\special{pa 2120 1395}%
\special{pa 2096 1402}%
\special{pa 2160 1430}%
\special{fp}%
%
\special{pn 8}%
\special{pa 2170 1630}%
\special{pa 1920 1850}%
\special{fp}%
\special{sh 1}%
\special{pa 1920 1850}%
\special{pa 1983 1821}%
\special{pa 1960 1815}%
\special{pa 1957 1791}%
\special{pa 1920 1850}%
\special{fp}%
\put(43.9000,-24.0000){\makebox(0,0)[lb]{$1$}}%
\put(39.9000,-20.0000){\makebox(0,0)[lb]{$s_i$}}%
\put(47.3000,-20.1000){\makebox(0,0)[lb]{$s_j$}}%
\put(42.3000,-9.6000){\makebox(0,0)[lb]{$s_is_js_i=s_js_is_j$}}%
%
\special{pn 8}%
\special{pa 4340 1000}%
\special{pa 4060 1230}%
\special{fp}%
\special{sh 1}%
\special{pa 4060 1230}%
\special{pa 4124 1203}%
\special{pa 4101 1196}%
\special{pa 4099 1172}%
\special{pa 4060 1230}%
\special{fp}%
%
\special{pn 8}%
\special{pa 4060 2030}%
\special{pa 4340 2260}%
\special{fp}%
\special{sh 1}%
\special{pa 4340 2260}%
\special{pa 4301 2202}%
\special{pa 4299 2226}%
\special{pa 4276 2233}%
\special{pa 4340 2260}%
\special{fp}%
%
\special{pn 8}%
\special{pa 4470 990}%
\special{pa 4750 1230}%
\special{fp}%
\special{sh 1}%
\special{pa 4750 1230}%
\special{pa 4712 1171}%
\special{pa 4710 1195}%
\special{pa 4686 1202}%
\special{pa 4750 1230}%
\special{fp}%
%
\special{pn 8}%
\special{pa 4760 2030}%
\special{pa 4510 2250}%
\special{fp}%
\special{sh 1}%
\special{pa 4510 2250}%
\special{pa 4573 2221}%
\special{pa 4550 2215}%
\special{pa 4547 2191}%
\special{pa 4510 2250}%
\special{fp}%
%
%
%
\special{pn 8}%
\special{pa 4030 1420}%
\special{pa 4030 1820}%
\special{fp}%
\special{sh 1}%
\special{pa 4030 1820}%
\special{pa 4050 1753}%
\special{pa 4030 1767}%
\special{pa 4010 1753}%
\special{pa 4030 1820}%
\special{fp}%
%
\special{pn 8}%
\special{pa 4800 1430}%
\special{pa 4800 1830}%
\special{fp}%
\special{sh 1}%
\special{pa 4800 1830}%
\special{pa 4820 1763}%
\special{pa 4800 1777}%
\special{pa 4780 1763}%
\special{pa 4800 1830}%
\special{fp}%
\put(17.9000,-19.2000){\makebox(0,0)[lb]{$1$}}%
\put(39.1000,-13.8000){\makebox(0,0)[lb]{$s_js_i$}}%
\put(46.6000,-13.8000){\makebox(0,0)[lb]{$s_is_j$}}%
\put(25.3000,-16.0000){\makebox(0,0)[lb]{$(|i-j|>1)$}}%
\put(50.9000,-16.4000){\makebox(0,0)[lb]{$(|i-j|=1)$}}%
\put(6.6000,-15.8000){\makebox(0,0)[lb]{$(\ast)$}}%
\end{picture}}%
\vspace{5pt}\]
Accordingly, we have (1).

We prove remaining assertions. If $[0,X_{i,j}]$ has form (i), then $X_i$ and $X_j$ are projective  $\Lambda/(1-e_i-e_j)$-modules. Therefore,
we obtain the assertion (2). 
Let $M$ be a maximum element of $\sttilt_{(1-e_i-e_j)\Lambda^-} \Lambda$.
Then $ X_i,X_j\leq M$ implies that $X_{i,j}\leq M$. In particular, $[0,X_{i,j}]$ is a full subposet
of $\sttilt_{(1-e_i-e_j)\Lambda^-} \Lambda\simeq \sttilt \Lambda/(1-e_i-e_j)$. Since $[0,X_{i,j}]$ is 2-regular and so finite connected component
of $\sttilt_{(1-e_i-e_j)\Lambda^-} \Lambda$. 
This implies that \[[0,X_{i,j}]=\sttilt_{(1-e_i-e_j)\Lambda^-} \Lambda\simeq \sttilt \Lambda/(1-e_i-e_j).\]
Then  the assertion (3) follows from Proposition\;\ref{mr}. 
The assertion (4) is a direct consequence of (1), (2), (3) and ($*$).
\end{proof}

\begin{lemma}
\label{localstroflambda}
Let $Y_i\in \dis(\Lambda)$ such that $P_i\not\in \add Y_i$.
We define $\sigma\in \Sym_{n} $ as follows:
\[T_{s_i w_0}= Y_{\sigma(i)}.\]
\begin{enumerate}[{\rm (1)}]
\item Let $i\neq j\in Q_0$ and $Y_{i,j}:=Y_i\wedge Y_j$. Then $[Y_{i,j},\Lambda]$
has one of the following form.  
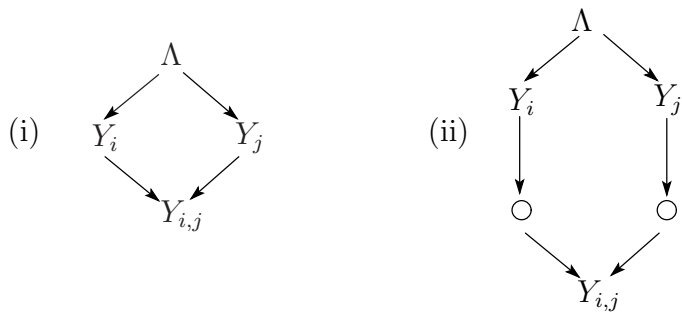
\[
{\unitlength 0.1in%
\begin{picture}( 35.0000, 14.7000)(  9.6000,-23.1000)%
\put(17.6000,-20.2000){\makebox(0,0)[lb]{$Y_{i,j}$}}%
\put(14.0000,-16.0000){\makebox(0,0)[lb]{$Y_{i}$}}%
\put(21.5000,-16.1000){\makebox(0,0)[lb]{$Y_{j}$}}%
\put(17.6000,-11.5000){\makebox(0,0)[lb]{$\Lambda$}}%
%
\special{pn 8}%
\special{pa 1750 1200}%
\special{pa 1470 1430}%
\special{fp}%
\special{sh 1}%
\special{pa 1470 1430}%
\special{pa 1534 1403}%
\special{pa 1511 1396}%
\special{pa 1509 1372}%
\special{pa 1470 1430}%
\special{fp}%
%
\special{pn 8}%
\special{pa 1470 1630}%
\special{pa 1750 1860}%
\special{fp}%
\special{sh 1}%
\special{pa 1750 1860}%
\special{pa 1711 1802}%
\special{pa 1709 1826}%
\special{pa 1686 1833}%
\special{pa 1750 1860}%
\special{fp}%
%
\special{pn 8}%
\special{pa 1880 1190}%
\special{pa 2160 1430}%
\special{fp}%
\special{sh 1}%
\special{pa 2160 1430}%
\special{pa 2122 1371}%
\special{pa 2120 1395}%
\special{pa 2096 1402}%
\special{pa 2160 1430}%
\special{fp}%
%
\special{pn 8}%
\special{pa 2170 1630}%
\special{pa 1920 1850}%
\special{fp}%
\special{sh 1}%
\special{pa 1920 1850}%
\special{pa 1983 1821}%
\special{pa 1960 1815}%
\special{pa 1957 1791}%
\special{pa 1920 1850}%
\special{fp}%
\put(9.6000,-15.9000){\makebox(0,0)[lb]{$\mathrm{(i)}$}}%
\put(39.4000,-24.4000){\makebox(0,0)[lb]{$Y_{i,j}$}}%
\put(39.1000,-9.7000){\makebox(0,0)[lb]{$\Lambda$}}%
%
\special{pn 8}%
\special{pa 3950 1000}%
\special{pa 3670 1230}%
\special{fp}%
\special{sh 1}%
\special{pa 3670 1230}%
\special{pa 3734 1203}%
\special{pa 3711 1196}%
\special{pa 3709 1172}%
\special{pa 3670 1230}%
\special{fp}%
%
\special{pn 8}%
\special{pa 3670 2030}%
\special{pa 3950 2260}%
\special{fp}%
\special{sh 1}%
\special{pa 3950 2260}%
\special{pa 3911 2202}%
\special{pa 3909 2226}%
\special{pa 3886 2233}%
\special{pa 3950 2260}%
\special{fp}%
%
\special{pn 8}%
\special{pa 4080 990}%
\special{pa 4360 1230}%
\special{fp}%
\special{sh 1}%
\special{pa 4360 1230}%
\special{pa 4322 1171}%
\special{pa 4320 1195}%
\special{pa 4296 1202}%
\special{pa 4360 1230}%
\special{fp}%
%
\special{pn 8}%
\special{pa 4370 2030}%
\special{pa 4120 2250}%
\special{fp}%
\special{sh 1}%
\special{pa 4120 2250}%
\special{pa 4183 2221}%
\special{pa 4160 2215}%
\special{pa 4157 2191}%
\special{pa 4120 2250}%
\special{fp}%
\put(31.6000,-15.9000){\makebox(0,0)[lb]{$\mathrm{(ii)}$}}%
%
%
%
\special{pn 8}%
\special{pa 3640 1420}%
\special{pa 3640 1820}%
\special{fp}%
\special{sh 1}%
\special{pa 3640 1820}%
\special{pa 3660 1753}%
\special{pa 3640 1767}%
\special{pa 3620 1753}%
\special{pa 3640 1820}%
\special{fp}%
%
\special{pn 8}%
\special{pa 4410 1430}%
\special{pa 4410 1830}%
\special{fp}%
\special{sh 1}%
\special{pa 4410 1830}%
\special{pa 4430 1763}%
\special{pa 4410 1777}%
\special{pa 4390 1763}%
\special{pa 4410 1830}%
\special{fp}%
\put(35.8000,-14.1000){\makebox(0,0)[lb]{$Y_{i}$}}%
\put(43.4000,-14.1000){\makebox(0,0)[lb]{$Y_{j}$}}%
%
\special{pn 8}%
\special{ar 3650 1910 50 50  0.0000000  6.2831853}%
%
\special{pn 8}%
\special{ar 4410 1910 50 50  0.0000000  6.2831853}%
\end{picture}}%
\vspace{5pt}\]
\item In the case of {\rm (i)}, there is no edge between $i$ and $j$ in $Q$.
\item In the case of {\rm (ii)}, then we have 
\[
{\unitlength 0.1in%
\begin{picture}(  3.5000,  1.2700)(  4.0000, -3.9700)%
\put(4.0000,-4.0000){\makebox(0,0)[lb]{$i$}}%
%
\special{pn 8}%
\special{pa 540 320}%
\special{pa 740 320}%
\special{fp}%
\special{sh 1}%
\special{pa 740 320}%
\special{pa 673 300}%
\special{pa 687 320}%
\special{pa 673 340}%
\special{pa 740 320}%
\special{fp}%
%
\special{pn 8}%
\special{pa 710 370}%
\special{pa 510 370}%
\special{fp}%
\special{sh 1}%
\special{pa 510 370}%
\special{pa 577 390}%
\special{pa 563 370}%
\special{pa 577 350}%
\special{pa 510 370}%
\special{fp}%
\put(7.5000,-4.1000){\makebox(0,0)[lb]{$j$}}%
\end{picture}}%
\vspace{5pt}\]
\item $\sigma$ induces quiver automorphism 
\[Q^{\circ}\stackrel{\sim}{\to} Q^{\circ}.\]
\end{enumerate}

\end{lemma}
\begin{proof}
Similar to the proof of Lemma\;\ref{determiningarrow} (1), one can easily check (1).
Note that $Y_i,Y_j\in \sttilt_{\Lambda/(e_i \Lambda\oplus e_{j}\Lambda)} \Lambda \simeq \sttilt \Lambda/(1-e_i-e_j)$.
By using same argument used in the proof of Lemma\;\ref{determiningarrow}, one has 
$[Y_i\bigwedge Y_j,\Lambda]=\sttilt_{\Lambda/(e_i \Lambda\oplus e_{j}\Lambda)} \Lambda$.
Then the assertions (2) and (3) follow from Proposition\;\ref{mr}.
 We prove (4). $i$ and $j$ shear an edge in $Q$ if and only if 
 $|i-j|=1$. By definition of $\sigma$, we have that $|i-j|=1$ if and only if
 $Y_{\sigma(i)}$ and $Y_{\sigma(j)}$ satisfies (ii). Hence the assertion
 follows from (1), (2) and (3).  
   \end{proof}
 
 \begin{lemma}
 \label{reduction}
 Let $i<j$ and $e=e_i+e_{i+1}+\cdots+e_j$. 
 \begin{enumerate}[{\rm (1)}]
  \item $\sttilt (\Lambda/(1-e))\simeq (\Sym_{j-i+2},\leq)$.
  \item If $j-i<n-1$, then we have a path
  \[e_i\Lambda/e_i \Lambda (1-e_i)\Lambda\leftarrow
   e_i\Lambda/e_i \Lambda (1-e_i)\Lambda\oplus e_i\Lambda/e_i \Lambda (1-e_i-e_{i+1})\Lambda
   \leftarrow\cdots\leftarrow \bigoplus_{k=i}^j e_i\Lambda/e_i \Lambda (1-e_i-\cdots-e_k)\Lambda  \]
   in $\sttilt_{(1-e)\Lambda^-} \Lambda. $
   \end{enumerate}
 \end{lemma}
\begin{proof}
We prove (1). Since $s_i\vee s_{i+1}\vee \cdots \vee s_j$ be the longest element in 
$\langle s_i,\dots,s_j\rangle\simeq \Sym_{j-i+2},$ we have that
\[[0,T_{s_i}\vee\cdots\vee T_{s_j}]=\rho( [1, s_{i}\vee \cdots \vee s_j]) \simeq (\Sym_{j-i+2},\leq). \]
Note that $T_{s_i},\dots,T_{s_j}\in \sttilt_{(1-e)\Lambda^-} \Lambda $. Therefore for a maximum element
$M$ of $\sttilt_{(1-e)\Lambda^-} \Lambda $, we obtain that
\[T_{s_i}\vee\cdots\vee T_{s_j}\leq M.\] In particular, 
\[[0,T_{s_i}\vee\cdots\vee T_{s_j}]\subset \sttilt_{(1-e)\Lambda^-} \Lambda.\]
Since $\sttilt_{(1-e)\Lambda^-} \Lambda\simeq \sttilt (\Lambda/(1-e))$ is $(j-i+1)$-regular
poset, we conclude that 
\[\sttilt (\Lambda/(1-e))\simeq [0,T_{s_i} \vee \cdots \vee T_{s_j} ]\simeq (\Sym_{j-i+2},\leq).\]

Next we show (2). By (1) and the hypothesis of induction, $\Lambda_e :=\Lambda/(1-e)$ satisfies (a), (b) and (c)
of Condition\;\ref{nscd}.
Hence one can check that
\[\bigoplus_{k=i}^{\ell-1} e_{i}\Lambda_e/e_i\Lambda_e (1-e_i-\cdots-e_{k})\Lambda_e\leftarrow 
\bigoplus_{k=i}^{\ell} e_{i}\Lambda_e/e_i\Lambda_e (1-e_i-\cdots-e_k)\Lambda_e. \]
In fact, the 2-term presilting object in $\Kb(\proj \Lambda_e)$ corresponding to
 $e_{i}\Lambda_e/e_i\Lambda_e (1-e_{i,k})\Lambda$
is $X_{\mathbf{i}_{i,k}}$, where $e_{i,k}:=e_i+\cdots+e_k$ and $\mathbf{i}_{i,k}:=\{i-1<i<k+1\}$. 
Then $X_{\mathbf{i}_{i,k}}\oplus X_{\mathbf{i}_{i,k'}}$ is a presilting object 
of $\sttilt \Lambda_e=\sttilt_{(1-e)\Lambda^-} \Lambda$ by Proposition\;\ref{rigidconditionforx} (or direct calculation).
 Then the assertion follows from 
 \[e_{i}\Lambda_e/e_i\Lambda_e (1-e_i-\cdots-e_{k})\Lambda_e\simeq e_i \Lambda/e_i\Lambda (1-e_i-\cdots-e_{k})\Lambda.\]
\end{proof} 
\begin{lemma}
\label{titoj} We have the following.
\begin{enumerate}[{\rm(1)}]
\item If $i<j$. Then $T_{s_j s_{j-1}\cdots s_i}$ is a unique element of
\[ \dip(T_{s_{j-1}\cdots s_i})\cap [T_{s_{j-2}\cdots s_i},T_{s_{j-1}\cdots s_i}\vee T_{s_j}].\] 
Moreover, if $(i,j)\neq (1,n)$, then 
\[T_{s_js_{j-1}\cdots s_i}=\bigoplus_{k=i}^j e_{i}\Lambda /e_i\Lambda (1-e_i-\cdots-e_k)\Lambda.\]
\item If $i>j$. Then $T_{s_j s_{j+1}\cdots s_i}$ is a unique element of
\[\dip(T_{s_{j+1}\cdots s_i})\cap [T_{s_{j+2}\cdots s_i},T_{s_{j+1}\cdots s_i}
\vee T_{s_j}].\] 
Moreover, if $(i,j)\neq (n,1)$, then 
\[T_{s_js_{j+1}\cdots s_i}=\bigoplus_{k=i}^j e_{i}\Lambda /e_i\Lambda (1-e_i-\cdots-e_k)\Lambda.\]
\end{enumerate}
\end{lemma}
\begin{proof}
We consider the case that $i<j$.
We claim that 
\[T_{s_j s_{j-1}\cdots s_i}\in \dip(T_{s_{j-1}\cdots s_i})\cap [T_{s_{j-2}\cdots s_i},T_{s_{j-1}\cdots s_i}\vee T_{s_j}].\] 
It is obvious that $T_{s_js_{j-1}\cdots s_i}\in \dip(T_{s_{j-1}\cdots s_i})$. Put $w=s_{j-2}\cdots s_i$. 
Then we obtain that
\[
{\unitlength 0.1in%
\begin{picture}( 11.4700, 14.2000)( 32.9000,-22.6000)%
\put(39.6000,-23.7000){\makebox(0,0)[lb]{$w$}}%
\put(32.9000,-20.1000){\makebox(0,0)[lb]{$ws_j=s_j w$}}%
\put(41.4000,-20.1000){\makebox(0,0)[lb]{$s_{j-1}w$}}%
\put(37.1000,-9.7000){\makebox(0,0)[lb]{$s_js_{j-1}s_j w$}}%
%
\special{pn 8}%
\special{pa 3950 1000}%
\special{pa 3670 1230}%
\special{fp}%
\special{sh 1}%
\special{pa 3670 1230}%
\special{pa 3734 1203}%
\special{pa 3711 1196}%
\special{pa 3709 1172}%
\special{pa 3670 1230}%
\special{fp}%
%
\special{pn 8}%
\special{pa 3670 2030}%
\special{pa 3950 2260}%
\special{fp}%
\special{sh 1}%
\special{pa 3950 2260}%
\special{pa 3911 2202}%
\special{pa 3909 2226}%
\special{pa 3886 2233}%
\special{pa 3950 2260}%
\special{fp}%
%
\special{pn 8}%
\special{pa 4080 990}%
\special{pa 4360 1230}%
\special{fp}%
\special{sh 1}%
\special{pa 4360 1230}%
\special{pa 4322 1171}%
\special{pa 4320 1195}%
\special{pa 4296 1202}%
\special{pa 4360 1230}%
\special{fp}%
%
\special{pn 8}%
\special{pa 4370 2030}%
\special{pa 4120 2250}%
\special{fp}%
\special{sh 1}%
\special{pa 4120 2250}%
\special{pa 4183 2221}%
\special{pa 4160 2215}%
\special{pa 4157 2191}%
\special{pa 4120 2250}%
\special{fp}%
%
%
\special{pn 8}%
\special{pa 3640 1420}%
\special{pa 3640 1820}%
\special{fp}%
\special{sh 1}%
\special{pa 3640 1820}%
\special{pa 3660 1753}%
\special{pa 3640 1767}%
\special{pa 3620 1753}%
\special{pa 3640 1820}%
\special{fp}%
%
\special{pn 8}%
\special{pa 4410 1430}%
\special{pa 4410 1830}%
\special{fp}%
\special{sh 1}%
\special{pa 4410 1830}%
\special{pa 4430 1763}%
\special{pa 4410 1777}%
\special{pa 4390 1763}%
\special{pa 4410 1830}%
\special{fp}%
\put(34.0000,-14.0000){\makebox(0,0)[lb]{$s_{j-1}s_j w$}}%
\put(41.4000,-14.0000){\makebox(0,0)[lb]{$s_j s_{j-1}w$}}%
\end{picture}}%
\vspace{8pt}\]
Note that $w s_j \geq s_j$ and $w\not\geq s_j$. 
Hence we conclude that $w s_j=s_j\vee w$
and
\[s_{j-1}w \vee s_j= s_{j-1}w\vee w \vee s_j= s_js_{j-1}s_j w=s_{j-1}s_j s_{j-1}w. \]
In particular, we conclude that
\[T_{s_j s_{j-1} w}\in [T_w, T_{s_{j-1} w}\vee T_{s_j w}]= [T_w, T_{s_{j-1} w}\vee T_{s_j}].\]
Then the uniqueness follows from the fact that $\dis(s_js_{j-1}s_j w)=\{s_{j-1}s_j w,s_js_{j-1}w\}$.  

We note that $T_{s_i}=X_i=e_i\Lambda/e_i\Lambda(1-e_i)\Lambda$.
Now we assume that 
\[T_{s_{j'}s_{j'-1}\cdots s_i}=\bigoplus_{k=i}^{j'}e_i\Lambda/e_i\Lambda(1-e_i-\cdots-e_{j'})\Lambda\]
 holds for any $j'\in \{i,\cdots, j-1 \}$.
Then $T_{s_j w}=T_w\vee T_{s_j}= T_w \oplus X_j $ and 
$T_{s_{j-1}w}= T_{w}\oplus e_{i}\Lambda /e_i\Lambda (1-e_i-\cdots-e_{j-1})\Lambda.$
Therefore $T_{s_j w}, T_{s_{j-1}w}\in \sttilt_{T_w\oplus(1-e)\Lambda^-} \Lambda$, where $e:= e_i+\cdots+ e_j $.
This shows that $T_{s_j w}\vee T_{s_{j-1}w}\leq M$, where $M$ is a maximum element 
of $\sttilt_{T_w\oplus (1-e)\Lambda^-} \Lambda$.
In particular, we have that
\[\sttilt_{T_w\oplus(1-e)\Lambda^-} \Lambda\supset [T_w, T_{s_{j-1} w}\vee T_{s_j w}]= [T_w, T_{s_{j-1} w}\vee T_{s_j}].\]
By Jasso's theorem, we have that $\sttilt_{T_w\oplus(1-e)\Lambda^-} \Lambda$ is a two-regular poset.
Hence we obtain that    
\[\sttilt_{T_w\oplus(1-e)\Lambda^-} \Lambda
= [T_w, T_{s_{j-1} w}\vee T_{s_j w}]= [T_w, T_{s_{j-1} w}\vee T_{s_j}].\]
Note that Lemma\;\ref{reduction} implies 
\[T_{s_{j-1}w}\oplus e_i \Lambda/e_i\Lambda (1-e) \Lambda \in \sttilt_{T_w\oplus(1-e)\Lambda^-} \Lambda.\]
Hence, we get that
\[\bigoplus_{k=i}^j e_{i}\Lambda /e_i\Lambda (1-e_i-\cdots-e_k)\Lambda=
T_{s_{j-1}w}\oplus e_i \Lambda/e_i\Lambda (1-e) \Lambda\in \dip(T_{s_{j-1}w})\cap [T_w, T_{s_{j-1} w}\vee T_{s_j}].\]
In particular, the following hold.
\[T_{s_j s_{j-1} w}= \bigoplus_{k=i}^j e_{i}\Lambda /e_i\Lambda (1-e_i-\cdots-e_k)\Lambda. \] 
Accordingly, we obtain (1). Similar argument gives the assertion (2).  
\end{proof}

\begin{lemma}
\label{teqlemma}
Let $w=s_n s_{n-1}\cdots s_i$. 
\begin{enumerate}[{\rm (1)}]
\item  For any $\sigma\in \langle s_1,\dots,s_{n-1}\rangle$, we have
\[\sigma w\geq w.\]
\item Let $\sigma,\sigma'\in \langle s_1,\dots,s_{n-1}\rangle$. Then 
\[\sigma w\leq \sigma'w\Leftrightarrow \sigma\leq \sigma'.\]
\item We have the following.
\[[w,(s_1w)\vee\cdots\vee(s_{n-1}w)]=[w,(s_1\vee\cdots\vee s_{n-1})w]=\langle s_1,\dots, s_{n-1}\rangle w
\stackrel{\mathrm{poset}}{\simeq} (\Sym_{n},\leq). \]
\end{enumerate}
\end{lemma}
\begin{proof} We show (1).
Suppose that there exists $\sigma\in \langle s_1,\dots,s_{n-1}\rangle$ and $j\in \{1,\dots ,n-1\}$ such that
$\sigma w\geq w$ and $s_j \sigma w\not\geq w$.
Let $s_{i_{\ell}}\cdots s_{i_{n-i+2}} s_n\cdots s_i$ be a reduced expression of $\sigma w$.
We set $i_{k}:= k+i-1$ for $k\leq n-i+1$ and $i_{\ell+1}=j$.
Then by Theorem\;\ref{matsumoto}, there exists $(j<k)$ such that
\begin{enumerate}[{\rm (i)}]
\item $s_{i_{\ell+1}}\cdots s_{i_{j+1}}(i_j)>s_{i_{\ell+1}}\cdots s_{i_{j+1}}(i_j+1)$ 
\end{enumerate}
If $j\leq n-i+1$, then $s_{i_{\ell+1}}\cdots s_{i_{j+1}}(i_j+1)=n+1$. This contradicts to (i).
Thus $j>n-i+1$. Then Theorem\;\ref{matsumoto} says that   
$s_j \sigma w=s_{i_{\ell+1}}\cdots \widehat{s_{i_k}}\cdots \widehat{s_{i_j}}\cdots s_{i_1}s_n\cdots s_i$ and
this expression have to be a reduced expression of $s_j\sigma w$. This contradicts to $s_j \sigma w\not\geq w$.

We prove (2). 
First we assume that $\sigma\leq \sigma' $ and show that $\sigma w\leq \sigma' w$. 
We may assume that $\sigma'=s_i \sigma$.
By (1), there exists a reduced expression $s_{i_\ell}\cdots s_{i_1}$ of $\sigma$
such that $s_{i_\ell}\cdots s_{i_1}s_n\cdots s_i$ is a reduced expression of $\sigma w$.
Then we want to show that
 $s_is_{i_\ell}\cdots s_{i_1}s_n\cdots s_1$
is a reduced expression of $\sigma'w$. 
If not, then same argument  used in the proof of (1) implies that there exists $j<k$ such that
\[s_i \sigma w=s_{i_{\ell+1}}s_{i_{\ell}}\cdots \widehat{s_{i_k}}\cdots \widehat{s_{i_j}}\cdots s_{i_1} w,\]
where we put $s_{i_{\ell+1}}:=s_i$. This implies that $\ell(s_i\sigma)<\ell(\sigma)$.
Hence we reach a contradiction.



Next we assume that $\sigma w\leq \sigma' w$. Then assertion follows from (1). In fact,
there exists a reduced expression $s_{i_\ell}\cdots s_{i_1}$ of $\sigma$
such that $s_{i_\ell}\cdots s_{i_1}s_n\cdots s_i$ is a reduced expression of $\sigma w$.
If we take a path
\[s_{i_\ell}\cdots s_{i_1}s_n\cdots s_1\rightarrow s_{i_{\ell+1}}s_{i_{\ell}}\cdots s_{i_1} w \rightarrow 
\cdots\rightarrow s_{i_{\ell+\ell'}}\cdots s_{i_{\ell+1}}s_{i_{\ell}}\cdots s_{i_1} w=\sigma'w. \]
Then $s_{i_{\ell+\ell'}}\cdots s_{i_{\ell+1}}s_{i_{\ell}}\cdots s_{i_1}$ is a reduced expression of $\sigma'$.

We show the assertion (3). Let $J=\{1,\dots,n-1\}$. By Proposition\;\ref{fresultsym}, we have that
\[\bigvee_{j\in J}(s_jw)=(\bigvee_{j\in J}s_j)w=w_0(J)w.\]
Let $w'\in [w,w_0(J)w]$. Then there is a path
\[w\leftarrow s_{i_1}w\leftarrow\cdots\leftarrow s_{i_{\ell'}}\cdots s_{i_1}w=
w'\leftarrow \cdots \leftarrow s_{i_{\ell}}\cdots s_{i_{\ell'+1}}s_{i_{\ell'}}\cdots s_{i_1}w=w_0(J)w.\]
Let $\sigma =s_{i_{\ell'}}\cdots s_{i_1}$. By definition, $s_{i_{\ell'}}\cdots s_{i_1}$
is a reduced expression of $\sigma$ and $s_{i_{\ell}}\cdots s_{i_1}$ is a reduced expression of $w_0(J)$.
Also we have that $\sigma\leq w_0(J)$. Hence Proposition\;\ref{fresultsym}(2) implies that
$w'\in \langle s_j\mid j\in J\rangle w.$ Then the assertion follows from (1) and (2).
\end{proof}
Similarly, we have the following.
\begin{lemma}
\label{teqlemma2}
Let $w=s_1 s_{2}\cdots s_i$. 
\begin{enumerate}[{\rm (1)}]
\item  For any $\sigma\in \langle s_2,\dots,s_{n}\rangle$, we have
\[\sigma w\geq w.\]
\item Let $\sigma,\sigma'\in \langle s_2,\dots,s_{n}\rangle$. Then 
\[\sigma w\leq \sigma'w\Leftrightarrow \sigma\leq \sigma'.\]
\item We have the following.
\[[w,(s_2w)\vee\cdots\vee(s_n w)]=[w,(s_2\vee\cdots\vee s_n)w]=\langle s_2,\dots, s_n\rangle w
\stackrel{\mathrm{poset}}{\simeq} (\Sym_{n},\leq). \]
\end{enumerate}
\end{lemma}

\begin{lemma}
\label{keylemma} Let $w^+_i=s_n\cdots s_i$.
We put $M_i^+\in \ind \add T_{w_i^+}$ such that $M_i^+ \not\in \add T_{s_nw_i^+}$. Then
\[\sttilt_{M_i^+} \Lambda =\rho( [w_i^+, (s_1w_i^+)\vee\cdots \vee (s_{n-1} w_i^+)])
=\rho( \langle s_1,\dots s_{n-1} \rangle w_i^+). \]
\end{lemma}
\begin{proof}
Since $s_j w_i^+\geq w_i^+$ for any $j\in \{1,\dots, n-1\}$, we see that $T_{w_i^+}$ have to be the minimum element
of $\sttilt_{M_i^+} \Lambda$. (Note that $M_i^+\in \add T_{s_1w_i^+}\cap\cdots\cap \add T_{s_{n-1}w_i^+}$.)
Let $T$ be a maximum element of $\sttilt_{M_i^+} \Lambda$. Then we obtain that
\[T_{s_1w_i^+}\vee\cdots\vee T_{s_{n-1}w_i^+}\leq T.\]
In particular, Lemma\;\ref{teqlemma} implies that
\[\sttilt_{M_i^+} \Lambda\supset [T_{w_i^+},T_{s_1w_i^+}\vee\cdots\vee T_{s_{n-1}w_i^+}]
=\rho(\langle s_1,\dots, s_{n-1} \rangle w_i^+)\simeq (\Sym_n,\leq).\]
Since $(\Sym_n,\leq)$ is a $(n-1)$-regular poset, we conclude that
\[\sttilt_{M_i^+} \Lambda= [T_{w_i^+},T_{s_1w_i^+}\vee\cdots\vee T_{s_{n-1}w_i^+}]
=\rho(\langle s_1,\dots, s_{n-1} \rangle w_i^+).\]\end{proof}
Similarly we obtain the following.
\begin{lemma}
\label{keylemma1'} Let $w^-_i=s_1\cdots s_i$.
We put $M_i^-\in \ind \add T_{w_i^-}$ such that $M_i^- \not\in \add T_{s_1w_i^-}$. Then
\[\sttilt_{M_i^-} \Lambda =\rho( [w_i^-, (s_2w_i^-)\vee\cdots \vee (s_n w_i^-)])
=\rho( \langle s_2,\dots s_n \rangle w_i^-). \]
\end{lemma}
\begin{lemma}
\label{keylemma2}
$\sttilt_{P_{\sigma(i)}} \Lambda=\rho([\bigwedge_{k\neq i}(s_k w_0), w_0 ])=\rho(\langle s_k\mid k\neq i \rangle w_0.)$
\end{lemma}
\begin{proof}
Note that for any $k\neq i$, we have $P_{\sigma(i)}\in \add T_{s_k w_0}$.
Let $N_i$ be the minimum element of $\sttilt_{P_{\sigma(i)}} \Lambda$. Then 
$N_i\leq \bigwedge_{i\neq k} T_{s_kw_0}$ and 
\[\sttilt_{P_{\sigma(i)}} \Lambda \supset [\bigwedge_{i\neq k} T_{s_kw_0},\Lambda].\]
Now we let $\bigwedge_{i\neq k} (s_{k}w_0)=w w_0$ and $w'=\bigvee_{i\neq k} s_k$.
Then $w'w_0\leq s_k w_0$ for any $k\neq i$. Thus we conclude that
\[w'w_0\leq ww_0.\] 
Therefore we obtain
\[w'\geq w.\]
On the other hands, we have
$ ww_0\leq s_{k} w_0(\Leftrightarrow w\geq s_k)$ for any $k\neq i$.
In particular, $w\geq w'$. Hence we obtain $w=w'$ and
\[[\bigwedge_{k\neq i}(s_k w_0),w_0]=[(\bigvee_{k\neq i} s_k)w_0,w_0]
=\langle s_k\mid k\neq i\rangle w_0\simeq (\Sym_i\times \Sym_{n-i+1},\leq^{\mathrm{op}}).\]
Since $(\Sym_i\times \Sym_{n-i+1},\leq^{\mathrm{op}})$ is a $(n-1)$-regular poset,
  we obtain the assertion.
\end{proof}
\begin{lemma}
\label{impotantlemma}
We have the following.
\begin{enumerate}[{\rm (1)}]
\item $M_1^+=P_{\sigma(n)}$ and $M_n^-=P_{\sigma(1)}$.
\item If $i\neq 1$, then $M_i^+=e_{i}\Lambda /e_i\Lambda (1-e_i-\cdots-e_n)\Lambda.$
Furthermore, we have that
\[\sttilt_{M_i^+}\Lambda \cap \sttilt_{P_{\sigma(n-i+1)}} \Lambda \neq \emptyset.\]
\item If $i\neq n$, then $M_i^-=e_{i}\Lambda /e_i\Lambda (1-e_i-\cdots-e_1)\Lambda.$
Moreover, we have that
\[\sttilt_{M_i^-}\Lambda \cap \sttilt_{P_{\sigma(n-i+1)}} \Lambda \neq \emptyset.\]
\end{enumerate}
\end{lemma}
\begin{proof}
We show (1). Note that $w_0=s_1(s_2s_1)\cdots(s_n\cdots s_1)\in \langle s_1,\dots,s_{n-1}\rangle w_1^+$
and $w_1^+=s_n\cdots s_1\in \langle s_{1},\dots,s_{n-1} \rangle w_0.$
We also note that $w_0=(s_n)(s_{n-1}s_n)\cdots(s_1\cdots s_n)\in \langle s_2,\dots,s_n\rangle w_n^-$ and
$w_n^-=s_1\cdots s_n \in \langle s_{2},\dots,s_n \rangle w_0.$
Hence we have $\langle s_1,\dots,s_{n-1}\rangle w_1^+=\langle s_1,\dots,s_{n-1}\rangle w_0$
and $\langle s_2,\dots,s_n\rangle w_n^-=\langle s_2,\dots,s_n\rangle w_0$.
Then the assertion follows from
Lemma\;\ref{keylemma}, Lemma\;\ref{keylemma1'} and Lemma\;\ref{keylemma2}.
In fact, we see that
\[\sttilt_{M_1^+}\Lambda=\sttilt_{P_{\sigma(n)}} \Lambda,\ \sttilt_{M_n^-}\Lambda=\sttilt_{P_{\sigma(1)}} \Lambda.\]
 
Next we prove (2). We claim that
\[\langle s_k \mid k\neq n-i+1 \rangle w_0=\{w\in \Sym_{n+1}\mid w(a)\leq n-i+1\ \mathrm{for\ any\ } a\geq i+1  \}.  \]
Since $w_0(a)=n-a+2$, we obtain 
\[\langle s_k \mid k\neq n-i+1 \rangle w_0\subset \{w\in \Sym_{n+1}\mid w(a)\leq n-i+1\ \mathrm{for\ any\ } a\geq i+1  \}.\]
Then 
\[\langle s_k \mid k\neq n-i+1 \rangle w_0=\{w\in \Sym_{n+1}\mid w(a)\leq n-i+1\ \mathrm{for\ any\ } a\geq i+1  \}.  \]
follows from the fact that
\[ \langle s_k \mid k\neq n-i+1 \rangle w_0\stackrel{1:1}{\leftrightarrow}\Sym_{i}\times \Sym_{n-i+1} 
\stackrel{1:1}{\leftrightarrow}  \{w\in \Sym_{n+1}\mid w(a)\leq n-i+1\ \mathrm{for\ any\ } a\geq i+1  \}.\]

Let $w=(s_{n-1}\cdots s_{1})(s_{n-1}\cdots s_{2})\cdots(s_{n-1}\cdots s_{i-1})w_i^+$. One can easily check that
\[w(a)\leq n-i+1 \]
for any $a\geq i+1$. Hence $w\in \langle s_k \mid k\neq n-i+1 \rangle w_0 \cap \langle s_1\dots,s_{n-1}\rangle w_i^+$.
Then the assertion follows from Lemma\;\ref{titoj} (1), Lemma\;\ref{keylemma} and Lemma\;\ref{keylemma2}. 

Similar argument implies the assertion (3).
\end{proof}

\begin{lemma}
\label{pathha0jyanai}
We have the following.
\begin{enumerate}[{\rm (1)}]
\item $\sigma(i)=n-i+1$.
\item $\sttilt_{P_n^-}\Lambda=\rho(\langle s_1,\dots,s_{n-1} \rangle)$.
\item $\sttilt_{P_1^-} \Lambda=\rho(\langle s_{2},\dots,s_{n}\rangle)$.
\item $\supp(P_1)=\supp(P_n)=Q_0$.
\end{enumerate}
\end{lemma}
\begin{proof}
We prove (1). We first consider the case that $n$ is odd. Let $i=\frac{n+1}{2}$.
By Lemma\;\ref{localstroflambda} (4), either 
(i)\;$\sigma(a)=n+1-a$ for any $a\in\{1,\dots,n\}$ or (ii)\;$\sigma(a)=a$ for any $a\in\{1,\dots,n\}$.
occurs. In particular, we have $\sigma(i)=i$. Now it is sufficient to show that $\sigma(i-1)=i+1$.
If not, then we have $\sigma(i-1)=i-1$.
Let a minimal projective presentation
\[P_{M^+_{i+1}}:=[P_i^r\to P_{i+1} (\to e_{i+1} \Lambda/e_{i+1}\Lambda e_i \Lambda=M^+_{i+1})]\] 
of $M^+_{i+1}$. By Lemma\;\ref{impotantlemma}, we conclude that
$M^+_{i+1}\oplus P_{\sigma(n-i)}=M^+_{i+1}\oplus P_{i-1}$ is $\tau$-rigid.
Therefore $\Hom_{\Kb(\proj \Lambda)}(P_{M^+_{i+1}},P_{i-1}[1])=0$. This implies that
$\alpha_{i-1}\in e_{i-1}\Lambda e_i$ factors through $i+1$. (Note that $r>0$.) Accordingly, we reach a contradiction.

Assume that $n$ is even and let $i=\frac{n}{2}$. It is sufficient to show that
$\sigma(i)=n-i+1=i+1$. If not, then $\sigma(i)=i$.
Consider a minimal projective presentation
\[P_{M^+_{i+1}}:=[P_i^r\to P_{i+1} (\to e_{i+1} \Lambda/e_{i+1}\Lambda e_i \Lambda=M^+_{i+1})]\] 
of $M^+_{i+1}$. Then as in the case that $n$ is odd,
we see that $\Hom_{\Kb}(P_{M^+_{i+1}},P_i[1])=0$.
This is a contradiction. 

We prove (2). Note that $w:=s_1(s_2s_1)\cdots(s_{n-1}\cdots s_1)=s_1\vee\cdots\vee s_{n-1}$ is a maximum element of
$\langle s_1,\dots,s_{n-1}\rangle$.
Let $M$ be a maximum element of $\sttilt_{P_n^-} \Lambda$. Then $M\geq T_{s_1}\vee\cdots\vee T_{s_{n-1}}=T_{w}.$
Since $[0,T_w]=\rho( [1,w])=\rho(\langle s_1,\dots,s_{n-1}\rangle) \simeq (\Sym_{n},\leq )$, one obtains that
\[\sttilt_{P_n^-} \Lambda=[0,T_w]=\rho (\langle s_1,\dots,s_{n-1}\rangle).  \]
Similarly, one sees the assertion (3).

We show (4). $\supp(P_1)\neq Q_0$ implies that $(P_1,P_n)$ is a $\tau$-rigid pair.
In particular, we have that $\sttilt_{P_1} \Lambda\cap \sttilt_{P_n^-}\neq \emptyset$.
By Lemma\;\ref{keylemma2} and (1), one obtains that
\[\sttilt_{P_1} \Lambda=\rho(\langle s_1,\dots,s_{n-1} \rangle w_0).\]
On the other hand, the assertion (2) of this Lemma implies that
\[\sttilt_{P_n^-}\Lambda=\rho(\langle s_1,\dots,s_{n-1} \rangle).\]
Note that for any element $w\in \langle s_1,\dots,s_{n-1} \rangle w_0$, we have $w(n+1)\neq n+1$.
Also note that for any element $w\in \langle s_1,\dots,s_{n-1}\rangle $, we have $w(n+1)= n+1$.
This shows that \[\langle s_1,\dots,s_{n-1} \rangle w_0\cap \langle s_1,\dots,s_{n-1} \rangle=\emptyset.\] 
We conclude that
\[\supp (P_1)=Q_0.\]
Similar argument implies that
\[\supp (P_n)=Q_0.\]
\end{proof}
\begin{lemma}
\label{impotantlemma2}
We have the following.
\begin{enumerate}[{\rm (1)}]
\item $P_1\oplus X_1$, $P_n\oplus X_n$ are $\tau$-rigid and $P_i\oplus M^{\pm}_i$ is $\tau$-rigid for any $i\neq 1,n$.
\item For any $i\neq 1,n$, we have a minimum projective presentation
\[P_{i-1}\stackrel{\alpha_{i-1}^*}{\to} P_i\to e_i \Lambda/e_i\Lambda e_{i-1}\Lambda=M^+_i  \]
of $e_i \Lambda/e_i\Lambda e_{i-1}\Lambda$ and
a minimum projective presentation
\[P_{i+1}\stackrel{\alpha_{i}}{\to} P_i\to e_i \Lambda/e_i\Lambda e_{i+1}\Lambda=M^-_i  \]
of $e_i \Lambda/e_i\Lambda e_{i+1}\Lambda$.
 Furthermore, we obtain that
\[\alpha_{i-1}^* \Lambda= e_i\Lambda e_{i-1}\Lambda,\ e_i \Lambda \alpha_{i-1}^*=e_{i}\Lambda e_{i-1},\ 
 \alpha_i\Lambda = e_i\Lambda e_{i+1} \Lambda\ \mathrm{and\ }e_i\Lambda \alpha_i=e_i\Lambda e_{i+1}.
\] 
\item We have a minimum projective presentation
\[P_{2}\stackrel{\alpha_{1}}{\to} P_1\to e_1 \Lambda/e_1\Lambda e_{2}\Lambda=X_1\]
of $X_1$. Moreover, we obtain that
\[\alpha_{1} \Lambda= e_1\Lambda e_{2}\Lambda\ \mathrm{and\ }e_1 \Lambda \alpha_{1}=e_{1}\Lambda e_{2}.\]
\item We have a minimum projective presentation
\[P_{n-1}\stackrel{\alpha_{n-1}^*}{\to} P_n\to e_n \Lambda/e_n\Lambda e_{n-1}\Lambda=X_n\]
of $X_n$. Furthermore, we obtain that
\[\alpha_{n-1}^* \Lambda= e_n\Lambda e_{n-1}\Lambda\ \mathrm{and\ }e_n \Lambda \alpha_{n-1}^*=e_{n}\Lambda e_{n-1}.\]
\end{enumerate}
\end{lemma}
\begin{proof}
We prove (1). 
By Lemma\;\ref{pathha0jyanai} (2), we see that $n\not\in \supp(T_{s_n w_1^+})$ and $n\in \supp (T_{w_1^+})$.
This implies that $T_{ w_1^+}=T_{s_n w_1^+}\oplus M_1^+$ (see definition of $M_1^+$).
 By Lemma\;\ref{impotantlemma} and Lemma\;\ref{pathha0jyanai},
 we obtain that $M_1^+=P_1$.
In particular, $\add (T_{s_nw_1^+}\oplus P_1)\ni X_1\oplus P_1$ is $\tau$-rigid.
Similarly, we can check that $P_n\oplus X_n$ is $\tau$-rigid.
Also $P_i\oplus M_i^{\pm}$ are $\tau$-rigid by Lemma\;\ref{impotantlemma} (2), (3) and Lemma\;\ref{pathha0jyanai} (1).

We show (2). Let 
\[\bigoplus_{t=1}^r P_{i-1}^{(t)}=P_{i-1}^r \stackrel{f}{\to} P_i\to e_i \Lambda/e_i\Lambda e_{i-1}\Lambda=M^+_i  \]
be a minimal projective presentation
of $M^+_i=e_i \Lambda/e_i\Lambda e_{i-1}\Lambda$.
 (1) implies that \[\Hom_{\Kb(\proj \Lambda)}(P_{M^+_{i}},P_{i}[1])=0.\]
 Now we put $f=(f^{(t)}:P_{i-1}^{(t)}\to P_i)$ and consider $\phi\in\Hom_{\Kb(\proj \Lambda)}(P_{M^+_i}, P_i[1] ) $
 given by $\varphi^{(t)}:P_{i-1}^{(t)}\to P_i$, where
 $\varphi^{(t)}=\left\{
 \begin{array}{cl}
 \alpha^*_{i-1} & t=1 \\
 0& t\neq 1\\
 \end{array}\right.
 $.
 Then there exists $h\in \End_{\Lambda}(P_i)$ such that
 \[(\ast)\ \  h\circ f^{(t)}=\varphi^{(t)}\]
 for any $t$. This shows that $h$ has to be an isomorphism and $r=1$.
 Let $x=f(e_{i-1})$ and $y=h(e_i)$. Then $x\Lambda=e_i\Lambda e_{i-1} \Lambda$
 and $yx=\alpha_{i-1}^*$.
 Since $x\Lambda=e_i\Lambda e_{i-1} \Lambda$, there exists $y'\in e_{i-1}\Lambda e_{i-1}\setminus \Rad(e_{i-1} \Lambda e_{i-1})$
 such that $xy'=\alpha_{i-1}^*$. Hence we obtain
 \[\alpha_{i-1}^*\Lambda=xy' \Lambda =x\Lambda=e_i \Lambda e_{i-1} \Lambda.\]
   $\Hom_{\Kb}(P_{M^+_{i}},P_{i}[1])=0$ implies that
  for any morphism  $g$ from $P_{i-1}$ to $P_i$, there exists $h'\in \End_{\Lambda}(P_i)$ such that $g=h'\circ f$. This says that
  $e_i \Lambda e_{i-1}=e_i \Lambda x$.
 Therefore, we see that
 \[e_i\Lambda\alpha^*_{i-1}=e_i \Lambda yx=e_i \Lambda x= e_i \Lambda e_{i-1}.\] 
 By applying same argument to the minimum projective presentation 
 \[\bigoplus_{t=1}^r P_{i+1}^{(t)}=P_{i+1}^r \stackrel{f}{\to} P_i\to e_i \Lambda/e_i\Lambda e_{i+1}\Lambda=M^-_i\]
 of $M_i^-$, we have that $r=1$ and
 \[\alpha_i\Lambda  = e_i\Lambda e_{i+1} \Lambda,\ e_i\Lambda \alpha_i=e_i\Lambda e_{i+1}.\]
We now get the assertion (2).

Similarly one obtains (3) and (4).
\end{proof}
By Lemma\;\ref{pathha0jyanai} and Lemma\;\ref{impotantlemma2}, we have the following.
\begin{proposition}
\label{onlyifpart}
$\sttilt \Lambda\simeq (\Sym_{n+1},\leq)$ only if $\Lambda$ satisfies the Condition\;\ref{nscd}. 
\end{proposition}
\begin{proof}
Condition\;\ref{nscd} (a) follows from Lemma\;\ref{determiningarrow} (4) and Condition\;\ref{nscd} (b)
follows from Lemma\;\ref{impotantlemma2}. Hence
it is sufficient to show that
 \[\alpha_1\cdots\alpha_{n-1}\neq 0\neq \alpha_{n-1}^*\cdots\alpha_1^*.\]
If  $\alpha_1\cdots\alpha_{n-1}= 0$, then Lemma\;\ref{impotantlemma2} implies that
$n\not\in\supp(P_1)$. This contradicts to Lemma\;\ref{pathha0jyanai} (4). Therefore,
we obtain \[\alpha_1\cdots\alpha_{n-1}\neq 0.\]
Likewise, we also obtain \[\alpha_{n-1}^*\cdots\alpha_1^*\neq 0.\] 
\end{proof}
\section{Some remarks on g-vectors}
In this section,  we see that for two algebras satisfying the Condition\;\ref{nscd},
 an poset isomorphism $\sttilt \Lambda$ from $\sttilt \Gamma$ preserves g-vectors.
\begin{proposition}
\label{propforgvec}
Let $\rho$, $\rho'$ be poset isomorphisms from $(\Sym_{n+1},\leq)$ to $\sttilt \Lambda$.
If $\rho(s_i)=\rho'(s_i)$ holds for any $i$, then we have $\rho=\rho'$. 
\end{proposition}
\begin{proof}
We show the following claim.
\begin{claim}
\label{propforgvecclaim}
Let $w\in \Sym_{n+1}$ and $s_{i_{\ell}}\cdots s_{i_1}$ a reduced expression of $w$.
Assume that $\ell(s_jw)=\ell+1$ and put $w'=s_js_{i_{\ell-1}}\cdots s_{i_1}$. Then we have the following.
\begin{enumerate}[{\rm (a)}]
\item If $\ell(w')=\ell$, then $s_jw$ is a unique element of $\dip(w)\cap[w, w\vee w']$.
\item If $\ell (w')=\ell-2$, then $s_jw=s_{i_{\ell}}w'\vee s_j w'$.
\end{enumerate} 
\end{claim}
\begin{pfclaim}
We show the assertion (a). In the case that $|i_{\ell}-j|>1$, it is obvious that $s_jw=w\vee w'$.
Thus we may assume that $|i_{\ell}-j|=1$. Note that $\ell(s_{i_{\ell}}s_jw)=\ell(s_js_{i_{\ell}}w')=\ell+2$.
If not, then $\ell(s_{i_{\ell}}s_jw)=\ell(s_js_{i_{\ell}}w')=\ell$ and $s_{i_{\ell}}s_jw=s_js_{i_{\ell}}w'=s_jw\wedge s_{i_{\ell}}w'$.
Thus  we have that $s_{i_{\ell-1}}\cdots s_{i_1}\leq s_jw,s_{i_{\ell}}w'$ and  
\[s_{i_{\ell-1}}\cdots s_{i_1}<s_{i_{\ell}}s_jw.\] By considering lengths, we see that there exists $k$ such that
$s_ks_{i_{\ell-1}}\cdots s_{i_1}=s_{i_{\ell}}s_jw$. Hence $s_k=s_{i_{\ell}}s_js_{i_{\ell}}$, this is a contradiction.

Therefore, there are two paths 
\[s_{i_{\ell}}s_jw\to s_j w\to w \to s_{i_{\ell-1}}\cdots s_{i_1}\ \mathrm{and}
\ s_{i_{\ell}}s_jw=s_js_{i_{\ell}}w'\to s_{i_{\ell}}w'\to w' \to s_{i_{\ell-1}}\cdots s_{i_1}. \] 
This gives the assertion (a).

Next we show the assertion (b). Since $\ell(w')=\ell-2$, we have that 
$|i_{\ell}-j|=1$ and $\ell(s_jw)=\ell(s_{i_{\ell}}s_js_{i_{\ell}}w')=\ell+1$.
Then we have two paths 
\[s_jw\to w \to s_{i_{\ell}-1}\cdots s_{i_1}=s_jw'\to w'\ \mathrm{and\ }
s_j w=s_{i_{\ell}}s_js_{i_{\ell}}w'\to s_js_{i_{\ell}}w'\to s_{i_{\ell}}w'\to w'.\]
This implies the assertion (b).
\end{pfclaim}

Claim\;\ref{propforgvecclaim} says that an  poset automorphism $\varphi$ is uniquely determined by 
$\varphi(s_1),\dots,\varphi(s_n)$. In, particular, if $\varphi(s_i)=s_i$ holds for any $i$, then $\varphi=\mathrm{id}$. 
This gives the assertion.
\end{proof}
\begin{corollary}
\label{propforgvec2}
Let $\Lambda=kQ/I$, $\Gamma=kQ'/I'$ be algebras satisfying the Condition\;\ref{nscd}. 
Assume that $Q^{\circ}$ and $(Q')^{\circ}$ are the double quiver of
$1\to 2\to \cdots\to n.$
 Then there is 
a unique poset isomorphism $\rho:\sttilt \Lambda\stackrel{\sim}{\to}\sttilt \Gamma$
satisfying $\rho (e_i\Lambda/e_i\Lambda(1-e_i) \Lambda)=e_i\Gamma/e_i\Gamma(1-e_i)\Gamma$.
Moreover, $\rho$ preserves g-vectors i.e. we have that
\[g^T=g^{\rho (T)},\]
for any $T\in \sttilt \Lambda.$
\end{corollary}
\begin{proof}
By Proposition\;\ref{rigidconditionforx}, the map $X_{\mathbf{i}}(\Lambda)\to X_{\mathbf{i}}(\Gamma)$ 
induces a desired poset isomorphism.
Uniqueness follows from Proposition\;\ref{propforgvec}. 
\end{proof}

\end{document}